\documentclass[reqno]{amsart}

\usepackage{tikz-cd}

\usepackage{fullpage,dsfont}

\usepackage{hyperref} 

\usepackage{parskip}
\makeatletter 
\def\thm@space@setup{%
 \thm@preskip=\parskip \thm@postskip=0pt
}
\def\th@remark{%
  \thm@headfont{\itshape}%
  \normalfont 
  \thm@preskip\parskip \thm@postskip=0pt
}
\makeatother

\usepackage[nobysame,alphabetic,initials,msc-links]{amsrefs}

\usepackage[final]{pdfpages}

\usepackage{multirow}

\usepackage{array}

\DefineSimpleKey{bib}{how}
\DefineSimpleKey{bib}{mrclass}
\DefineSimpleKey{bib}{mrnumber}
\DefineSimpleKey{bib}{fjournal}
\DefineSimpleKey{bib}{mrreviewer}

\renewcommand{\PrintDOI}[1]{%
  \href{http://dx.doi.org/#1}{{\tt DOI:#1}}%
}
\renewcommand{\eprint}[1]{#1}
\BibSpec{book}{%
    +{}  {\PrintPrimary}                {transition}
    +{.} { \PrintDate}                  {date}
    +{.} { \textit}                     {title}
    +{.} { }                            {part}
    +{:} { \textit}                     {subtitle}
    +{,} { \PrintEdition}               {edition}
    +{}  { \PrintEditorsB}              {editor}
    +{,} { \PrintTranslatorsC}          {translator}
    +{,} { \PrintContributions}         {contribution}
    +{,} { }                            {series}
    +{,} { \voltext}                    {volume}
    +{,} { }                            {publisher}
    +{,} { }                            {organization}
    +{,} { }                            {address}
    +{,} { }                            {status}
    +{,} { \PrintDOI}                   {doi}
    +{,} { \PrintISBNs}                 {isbn}
    +{}  { \parenthesize}               {language}
    +{}  { \PrintTranslation}           {translation}
    +{;} { \PrintReprint}               {reprint}
    +{.} { }                            {note}
    +{.} {}                             {transition}
    +{}  {\SentenceSpace \PrintReviews} {review}
}
\BibSpec{article}{%
    +{}  {\PrintAuthors}                {author}
    +{,} { \textit}                     {title}
    +{.} { }                            {part}
    +{:} { \textit}                     {subtitle}
    +{,} { \PrintContributions}         {contribution}
    +{.} { \PrintPartials}              {partial}
    +{,} { }                            {journal}
    +{}  { \textbf}                     {volume}
    +{}  { \PrintDatePV}                {date}
    +{,} { \issuetext}                  {number}
    +{,} { \eprintpages}                {pages}
    +{,} { }                            {status}
    +{,} { \PrintDOI}                   {doi}
    +{,} { \eprint}        {eprint}
    +{}  { \parenthesize}               {language}
    +{}  { \PrintTranslation}           {translation}
    +{;} { \PrintReprint}               {reprint}
    +{.} { }                            {note}
    +{.} {}                             {transition}
    +{}  {\SentenceSpace \PrintReviews} {review}
}
\BibSpec{collection.article}{%
    +{}  {\PrintAuthors}                {author}
    +{,} { \textit}                     {title}
    +{.} { }                            {part}
    +{:} { \textit}                     {subtitle}
    +{,} { \PrintContributions}         {contribution}
    +{,} { \PrintConference}            {conference}
    +{}  {\PrintBook}                   {book}
    +{,} { }                            {booktitle}
    +{,} { \PrintDateB}                 {date}
    +{,} { pp.~}                        {pages}
    +{,} { }                            {publisher}
    +{,} { }                            {organization}
    +{,} { }                            {address}
    +{,} { }                            {status}
    +{,} { \PrintDOI}                   {doi}
    +{,} { \eprint}        {eprint}
    +{}  { \parenthesize}               {language}
    +{}  { \PrintTranslation}           {translation}
    +{;} { \PrintReprint}               {reprint}
    +{.} { }                            {note}
    +{.} {}                             {transition}
    +{}  {\SentenceSpace \PrintReviews} {review}
}
\BibSpec{misc}{%
  +{}{\PrintAuthors}  {author}
  +{,}{ \textit}      {title}
  +{.}{ }             {how}
  +{}{ \parenthesize} {date}
  +{,} { available at \eprint}        {eprint}
  +{,}{ available at \url}{url}
  +{,}{ }             {note}
  +{.}{}              {transition}
}

\usepackage{amssymb, amsfonts, amsxtra, amsmath}
\usepackage{mathrsfs}
\usepackage{mathdots}
\usepackage{wasysym}
\usepackage[all]{xy}
\usepackage{bbm}
\usepackage{calc}
\usepackage{accents}

\numberwithin{equation}{section}

\newtheorem{Theorem}{Theorem}[section]
\newtheorem*{Theorem*}{Theorem}
\newtheorem{Def}[Theorem]{Definition}
\newtheorem{Lem}[Theorem]{Lemma}
\newtheorem{Prop}[Theorem]{Proposition}
\newtheorem{Cor}[Theorem]{Corollary}

\newtheorem{Rem}[Theorem]{Remark}

\newcommand\bp{\begin{proof}}
\newcommand\ep{\end{proof}}

\mathchardef\mhyph="2D

\DeclareMathOperator{\ad}{\mathrm{ad}}

\DeclareMathOperator{\Ad}{\mathrm{Ad}}

\DeclareMathOperator{\Aut}{\mathrm{Aut}}
\DeclareMathOperator{\End}{\mathrm{End}}

\DeclareMathOperator{\fin}{\mathrm{f}}
\DeclareMathOperator{\loc}{\mathrm{loc}}

\DeclareMathOperator{\Hom}{\mathrm{Hom}}
\DeclareMathOperator{\id}{\mathrm{id}}

\DeclareMathOperator{\Tr}{\mathrm{Tr}}

\DeclareMathOperator{\braid}{\mathrm{br}}

\DeclareMathOperator{\Char}{\mathrm{Char}}

\DeclareMathOperator{\Ker}{\mathrm{Ker}}
\DeclareMathOperator{\Spec}{\mathrm{Spec}}

\newcommand{\cop}{\mathrm{cop}}
\newcommand{\op}{\mathrm{op}}
\newcommand{\wt}{\mathrm{wt}}

\newcommand{\msB}{\mathscr{B}}

\newcommand{\msE}{\mathscr{E}}

\newcommand{\msK}{\mathscr{K}}

\newcommand{\msP}{\mathscr{P}}
\newcommand{\msQ}{\mathscr{Q}}
\newcommand{\msR}{\mathscr{R}}

\newcommand{\msS}{\mathscr{S}}

\newcommand{\mfa}{\mathfrak{a}}

\newcommand{\mfb}{\mathfrak{b}}

\newcommand{\mfg}{\mathfrak{g}}

\newcommand{\mfh}{\mathfrak{h}}

\newcommand{\mfk}{\mathfrak{k}}

\newcommand{\mfn}{\mathfrak{n}}

\newcommand{\mfp}{\mathfrak{p}}

\newcommand{\mfs}{\mathfrak{s}}
\newcommand{\mfsl}{\mathfrak{sl}}
\newcommand{\mfso}{\mathfrak{so}}
\newcommand{\mfsp}{\mathfrak{sp}}
\newcommand{\mfsu}{\mathfrak{su}}
\newcommand{\mft}{\mathfrak{t}}
\newcommand{\mfu}{\mathfrak{u}}

\newcommand{\mcH}{\mathcal{H}}

\newcommand{\mcK}{\mathcal{K}}

\newcommand{\mcO}{\mathcal{O}}

\newcommand{\mcS}{\mathcal{S}}

\newcommand{\mcU}{\mathcal{U}}

\newcommand{\mbr}{\mathbf{r}}

\newcommand{\C}{\mathbb{C}}

\newcommand{\N}{\mathbb{N}}
\newcommand{\Q}{\mathbb{Q}}
\newcommand{\R}{\mathbb{R}}

\newcommand{\Z}{\mathbb{Z}}

\newcommand{\opp}{\mathrm{op}}

\newcommand\lhdb{\blacktriangleleft}

\newcommand{\spher}{\mathrm{s}}

\newcommand{\Elw}{\msK}

\newcommand{\com}{\mathrm{com}}

\newcommand{\quasiK}{\mathfrak{X}}
\newcommand{\signaut}{\Ad(\mathcal{S})}

\newcommand{\dbwidetilde}[1]{\accentset{\textrm{\scalebox{1.4}[1]{$\approx$}}}{#1}}
\newcommand{\dbwidehat}[1]{\widehat{\vphantom{\rule{1pt}{9pt}}\smash{\widehat{\!#1}}}}
\newcommand{\vardbwidehat}[1]{\;\;\widehat{\vphantom{\rule{1pt}{9pt}}\smash{\widehat{\!\!#1}}}}
\newcommand{\widehattilde}[1]{\;\;\widehat{\vphantom{\rule{1pt}{9pt}}\smash{\widetilde{\!\!#1}}}}
\newcommand{\dbbackslash}{\backslash \! \backslash}

\title{Quantum flag manifolds, quantum symmetric spaces and their associated universal $K$-matrices}

\author{Kenny De Commer}
\address{Vrije Universiteit Brussel}
\email{kenny.de.commer@vub.be}

\author{Marco Matassa}
\address{Vrije Universiteit Brussel/OsloMet – storbyuniversitetet}
\email{marco.matassa@oslomet.no}

\thanks{The work of K.~De Commer was partially supported by the FWO grant G.0251.15N and the grant H2020-MSCA-RISE-2015-691246-QUANTUM DYNAMICS. The work of M.~Matassa was supported by the FWO grant G.0251.15N. K. De Commer would like to thank D. Jordan, A. Mudrov and T. Weelinck for discussions around the topics of this paper.}

\begin{document}
\maketitle

\begin{abstract}
Let $U$ be a connected, simply connected compact Lie group with complexification $G$. Let $\mathfrak{u}$ and $\mathfrak{g}$ be the associated Lie algebras.  Let $\Gamma$ be the Dynkin diagram of $\mfg$ with underlying set $I$, and let $U_q(\mfu)$ be the associated quantized universal enveloping $*$-algebra of $\mfu$ for some $0<q$ distinct from $1$. Let $\mcO_q(U)$ be the coquasitriangular quantized function Hopf $*$-algebra of $U$, whose Drinfeld double $\mcO_q(G_{\R})$ we view as the quantized function $*$-algebra of $G$ considered as a real algebraic group. We show how the datum $\nu = (\tau,\epsilon)$ of an involution $\tau$ of $\Gamma$ and a $\tau$-invariant function $\epsilon: I \rightarrow \R$ can be used to deform $\mcO_q(G_{\R})$ into a $*$-algebra $\mcO_q^{\nu,\id}(G_{\R})$ by a modification of the Drinfeld double construction. We then show how, by a generalized theory of universal $K$-matrices, a specific $*$-subalgebra $\mcO_q(G_{\nu}\dbbackslash G_{\R})$ of $\mcO_q^{\nu,\id}(G_{\R})$ admits $*$-homomorphisms into both $U_q(\mfu)$ and $\mcO_q(U)$, the images being coideal $*$-subalgebras of respectively $U_q(\mfu)$ and $\mcO_q(U)$. We illustrate the theory by showing that two main classes of examples arise by such coideals, namely quantum flag manifolds and quantum symmetric spaces (except possibly for certain exceptional cases). In the former case this connects to work of the first author and Neshveyev, while for the latter case we heavily rely on recent results of Balagovi\'{c} and Kolb.   
\end{abstract}

\section*{Introduction}

Let $U$ be a connected, simply connected compact Lie group with complexification $G$, and let $\mfu$ and $\mfg$ be the respective Lie algebras. Fix Chevalley-Serre generators for $\mfg$ which are compatible with the compact form $\mfu$ (see Section \ref{SubsecQUE}), and let $\mfb$ and $\mfh \subseteq \mfb$ be the respective positive Borel and Cartan subalgebra in $\mfg$. From this data a natural Poisson-Lie group structure $\{-,-\}$ can be constructed on $U$ \cite{LW90}. Let $\Gamma$ be the Dynkin diagram of $\mfg$, with underlying set $I$. Then one has the following two important classes of Poisson homogeneous spaces for $U$:
\begin{enumerate}
\item \emph{Flag manifolds} $K_S\backslash U = P_S\backslash G$ for $S \subseteq I$ a subset of the simple roots, $P_S\subseteq G$ the associated parabolic subgroup and $K_S = P_S\cap U$ the compact form of the reductive Levi factor of $P_S$.
\item \emph{Symmetric spaces} $U^{\theta}\backslash U$ for $U^{\theta}$ the set of fixed points of a Lie group involution $\theta$ of $U$ in maximally split position with respect to the fixed Cartan subalgebra of $\mfu$ \cite{FL04}. 
\end{enumerate}
Note that the behaviour of the above two classes is different: whereas in the first case $K_S \subseteq U$ will be a Poisson-Lie subgroup, one has in the symmetric case only that $U^{\theta}$ is coisotropic. 

This discrepancy persists when turning to their quantizations. Let $U_q(\mfg)$ be the standard Drinfeld-Jimbo quantized enveloping algebra of $\mfg$, where we take $q>0$ a fixed number distinct from 1. Dually, one has a Hopf algebra $\mcO_q(G)$ quantizing the algebra of regular functions on $G$ as a complex affine group variety. When endowed with appropriate $*$-structures, reflecting the choice of a compact real form, we will denote the resulting Hopf $*$-algebras as $U_q(\mfu)$ and $\mcO_q(U)$. For $P_S \subseteq G$ a parabolic subgroup with Lie algebra $\mfp_S$ and $\mfk_S = \mfp_S \cap \mfu$ the Lie algebra of $K_S = P_S\cap U$, we have a natural Hopf $*$-subalgebra $U_q(\mfk_S) \subseteq U_q(\mfu)$ and dually a surjection of Hopf $*$-algebras $\mcO_q(U) \twoheadrightarrow \mcO_q(K_S)$. This allows one to make sense immediately of the associated \emph{quantum flag manifold} through its coordinate $*$-algebra $\mcO_q(K_S\backslash U)$. These quantum flag manifolds are studied intensively both from the algebraic and the operator algebraic viewpoint, see e.g.~ \cites{SV91, Soi92, Do94a, Do94b, DK94, DT99, DS99, HK04, Kr04, CFG08, NT12, DCN15, OB17}. 

Quantum symmetric spaces turn out to be harder to construct. After an initial period in which particular cases were studied \cites{NS95, Dij96, Nou96, BF97, DN98}, using for example the formalism of the \emph{reflection equation}, Letzter developed in a series of papers \cites{Let99,Let00,Let02,Let03,Let04,Let08} a uniform approach to quantum symmetric spaces through a concrete construction of their associated \emph{coideal subalgebras} $U_q(\mfu^{\theta}) \subseteq U_q(\mfu)$. This construction was extended to the Kac-Moody case by Kolb \cite{Kol14}, who also in joint work with Balagovi\'{c} elucidated the precise connection to the reflection equation in this full generality \cites{BK15b,Kol17} through the formalism of the \emph{universal $K$-matrix}. Associated quantized function algebras $\mcO_q(U^{\theta}\backslash U)$ for the corresponding homogeneous spaces can then be constructed by a general procedure. 

Our main aim will be to realize the above coideals $\mcO_q(K_S\backslash U)$ and $\mcO_q(U^{\theta}\backslash U)$ through the method of quantum characters \cite{DM03b}, building on and extending the techniques developed in \cite{KoSt09,Kol17}. Let $\nu = (\tau,\epsilon)$ be a couple consisting of an involution $\tau$ of $\Gamma$ and a $\tau$-invariant function $\epsilon: I \rightarrow \R$. Through a straightforward modification by $\nu$ of the commutation relations for $U_q(\mfu)$, one arrives at a quantized enveloping $*$-algebra $U_q(\mfg_{\nu})$ where $\mfg_{\nu}$ is a real Lie algebra determined directly in terms of $\nu$, see Section \ref{SubsecNuDef}. We note the following particular cases:
\begin{itemize}
\item When $\tau = \id$ and $\epsilon(I) \subseteq \{0,1\}$, one has $\mfg_{\nu} = \mfk_S \oplus \mfn_S^-$ where $\mfk_S$ and $\mfn_S^-$ are respectively the compact Levi part and the nilradical of the negative parabolic algebra $\mfp_S^-$ associated to the support $S$ of $\epsilon$. 
\item When on the other hand $\tau$ is an arbitrary involution and $\epsilon(I) \subseteq \{\pm1\}$, we have that $\mfg_{\nu}$ is the real form of $\mfg$ obtained by modifying the compact form $\mfu$ of $\mfg$ with $\nu$. The couple $(\tau,\epsilon)$ then encodes the so-called \emph{Vogan diagram} of $\mfg_{\nu}$.  
\end{itemize}
Returning back to the general case, $U_q(\mfg_{\nu})$ will still be a quasitriangular Hopf $*$-algebra, with a dual Hopf $*$-algebra $\mcO_q(G_{\nu})$. As $\mcO_q(G_{\nu})$ is coquasitriangular, one can consider its Drinfeld double $\mcO_q^{\nu}(G_{\R})$, which is a quantization of the function algebra of $G = G_{\R}$ as a real affine group along a Poisson-Lie group structure $\{-,-\}_{\nu}$. Now the precise form of the double construction allows the flexibility to let the latter depend rather on two \emph{distinct} data $\nu = (\tau,\epsilon)$ and $\mu = (\tau',\eta)$, creating a Poisson manifold $(G_{\R},\{-,-\}_{\nu,\mu})$ which is a Poisson bitorsor between $(G_{\R},\{-,-\}_{\nu})$ and $(G_{\R},\{-,-\}_{\mu})$ \cites{Wei90,Lu90}. Similarly, one can construct a $*$-algebra $\mcO_q^{\nu,\mu}(G_{\R})$ with a left coaction by $\mcO_q^{\nu}(G_{\R})$ and a right coaction by $\mcO_q^{\mu}(G_{\R})$. 

We will be interested in the case $\mu = \id$, where $\mcO_q^{\id}(G_{\R}) = \mcO_q(G_{\R})$ and $\mcO_q(G_{\id}) =\mcO_q(U)$. As we have a Hopf $*$-algebra surjection $\mcO_q^{\nu}(G_{\R}) \twoheadrightarrow \mcO_q(G_{\nu})$, we can consider the $*$-algebra $\mcO_q(G_{\nu}\dbbackslash G_{\R})$ of $\mcO_q(G_{\nu})$-coinvariant elements, together with its natural $\mcO_q(G_{\R})$-coaction. The method of quantum characters then gives a one-to-one correspondence between $*$-characters $\chi:\mcO_q(G_{\nu}\dbbackslash G_{\R}) \rightarrow \C$ and $\mcO_q(G_{\R})$-equivariant $*$-homomorphisms $\Phi: \mcO_q(G_{\nu}\dbbackslash G_{\R}) \rightarrow \mcO_q(G_{\R})$. The image of $\Phi$ will  be a coideal $*$-subalgebra $\mcO_q(L \dbbackslash G_{\R})$ of $\mcO_q(G_{\R})$, giving rise through appropriate projection maps to coideal $*$-subalgebras $\mcO_q(K\backslash U)$ and $U_q^{\fin}(\mfk')$ of respectively $\mcO_q(U)$ and $U_q(\mfu)$. Our main theorems can now be stated as follows.

\begin{Theorem*}[Theorem \ref{TheoFlagK}, Theorem \ref{TheoIdUnivFlag} and Theorem \ref{TheoMainFlag2}]
Let $\mcO_q(K_S\backslash U)$ be a quantum flag manifold, and let $S_0 = \tau_0(S)$ be the image of $S$ under the Dynkin diagram automorphism $\tau_0$ induced by the longest word in the Weyl group of $\mfg$. Let $\epsilon$ be the characteristic function of $S_0$, and let $\nu = (\id,\epsilon)$. Then there exists a $*$-character $\chi: \mcO_q(G_{\nu}\dbbackslash G_{\R}) \rightarrow \C$ such that:
\begin{itemize}
\item The equality $\mcO_q(K\backslash U) = \mcO_q(K_S\backslash U)$ holds. 
\item The inclusion $U_q^{\fin}(\mfk')\subseteq U_q(\mfk_S)$ holds, and moreover their completions coincide. 
\end{itemize}
\end{Theorem*} 

Here `completion' is understood in the weak sense, and can equivalently be formulated as equality of their images in any admissible finite dimensional representation of $U_q(\mfu)$.  
 
\begin{Theorem*}[Theorem \ref{TheoSymmK}, Theorem \ref{TheoEqUnivEnv} and Theorem \ref{TheoMainSymmFunct}]
Let $\mcO_q(U^{\theta}\backslash U)$ be a quantum symmetric space. Let $\nu =(\tau,\epsilon)$ be such that $\epsilon(I) \subseteq \{\pm1\}$ with $\mfg_{\nu}$ inner equivalent to $\mfg_{\theta}$ inside $\mfg$, for $\mfg_{\theta}$ the real form of $\mfg$ associated to $\theta$. Then there exists a $*$-character $\chi: \mcO_q(G_{\nu}\dbbackslash G_{\R}) \rightarrow \C$ such that:
\begin{itemize}
\item The equality $\mcO_q(K\backslash U) = \mcO_q(U^{\theta}\backslash U)$ holds, except possibly for $U^{\theta}\subseteq U$ of type $EIII$, $EIV$,  $EVI$, $EVII$ or $EIX$, using notation as in \cite{Ar62}.
\item The inclusion $U_q^{\fin}(\mfk') \subseteq U_q(\mfu^{\theta})$ holds, and their completions coincide. 
\end{itemize}
\end{Theorem*} 
 
Note that in both cases, the character $\chi$ will be constructed upon realising $\mcO_q(G_{\nu}\dbbackslash G_{\R})$ as a $\nu$-modified \emph{braided Hopf algebra} structure on $\mcO_q(G)$, for which the $*$-characters are then determined by a modified theory of universal $K$-matrices, see also \cite{KoSt09}*{Section 3.5}.

Let us end the introduction by motivating the above results from the Poisson-Lie point of view, without specifying precise details. Drinfeld has shown that a Poisson homogeneous manifold for a Poisson-Lie group is completely determined, up to local isomorphism, by a Lagrangian subalgebra in the Drinfeld double of its associated infinitesimal Lie bialgebra \cite{Dri93}. For our compact group $U$ in question, the Drinfeld double Lie bialgebra of $\mfu$ will be $\mfg$ with a particular real Lie bialgebra structure which integrates to the real Poisson-Lie group structure $\{-,-\} = \{-,-\}_{\id}$ on $G$ mentioned before. We then have the following:
\begin{itemize}
\item In the case of flag manifolds $K_S\backslash U$, the associated Lagrangian subgroup is $L = K_SN_S\subseteq G$, with $N_S$ the unipotent radical of the associated parabolic $P_S$.
\item In the case of symmetric spaces, the associated Lagrangian subgroup is $L = G_{\theta}\subseteq G$, with $G_{\theta}$ the real form of $G$ determined by $\theta$.
\end{itemize}
The given $U$-homogeneous space can then be reconstructed as the $U$-orbit at the unit of the space $L\backslash G$.  In \cite{STS85,STS94} one can find similar ideas in the purely complex Poisson and quantum setting, using however the (twisted) Heisenberg double instead of (twisted) Drinfeld double. We make a brief comparison with this alternative viewpoint in Appendix \ref{Ap0}.

In general, in its given position $L$ will not be a Poisson-Lie subgroup of $G$ with its usual Poisson bracket $\{-,-\}$. However, we can take an inner conjugate copy $L' = G_{\nu}$ of $L$ such that it will become a Poisson-Lie subgroup of $G$ with the Poisson-Lie group structure $\{-,-\}_{\nu}$, for $\nu$ suitably chosen. Denoting $G' = (G,\{-,-\}_{\nu})$ and $G'' = (G,\{-,-\}_{\nu,\id})$, one may then expect an isomorphism of Poisson $G$-spaces
\[
L'\backslash G'' \cong L\backslash G.
\] 
In our two specific settings, this isomorphism looks as follows, using notation as in the above theorems: 
\begin{itemize}
\item In the flag case, we have $K_{S_0}N_{S_0}^-\backslash G \cong K_SN_S \backslash G$.  
\item In the symmetric case, we have $G_{\nu} \backslash G \cong G_{\theta}\backslash G$. 
\end{itemize}
We note however that a complication results from establishing our results in the purely algebraic framework, for the quotient space $L\backslash G$ or $L'\backslash G''$ will not necessarily be a real affine variety, i.e.~ the fixed point set of a complex-conjugate involution on a complex affine variety. There will however be a natural real affine variety $L \dbbackslash G$ with a Zariski-dense embedding $L\backslash G \subseteq L \dbbackslash G$, allowing us to continue to work within the usual framework of unital algebras when quantizing varieties. One can view $L\dbbackslash  G$ as a GIT quotient in the setting of real algebraic geometry.

We end by observing that although these ideas are new when taking into account both the various twistings and $*$-structures, various of the basic techniques themselves, in particular the use of the reflection equation algebra, are well-established. We mention in this light especially \cite{STS85,STS94,KoSt09} and the work of Mudrov and collaborators \cites{DM02,DM03a,DM03b,DM04,Mud07a, Mud07b, Mud12, AM13, Mud13a, Mud13b, AM14, AM15}. We also mention that in a more restricted setting, closely related constructions were performed in \cite{DeC13}. Finally, we mention that this paper is a step towards proving part of the conjecture posed as \cite{DCNTY17}*{Conjecture 4.1}. However, to achieve this aim in full, one needs to complement the results of this paper with more refined representation-theoretic results, which falls outside the scope of the current paper. 

The precise structure of this paper is as follows. In the \emph{first section}, we establish the necessary preliminaries on quantized enveloping algebras and the associated quantized function algebras. In the \emph{second section}, we construct quantum bitorsors for complex quantum groups, and examine the quantum orbit spaces with respect to certain quantum subgroups. We then relate these to  quantum homogeneous spaces for the associated compact quantum groups through the method of quantum characters, establishing the connection to the twisted reflection equation and modified universal $K$-matrices. In the short \emph{third section}, we examine in more detail the case of quantum flag manifolds, and in the \emph{fourth section} we look at the connection to the quantum symmetric spaces of Letzter. In the \emph{Appendix} \ref{Ap0}, we consider some variations of the results in Section \ref{SecTwistBraid}, connecting to the work in \cite{STS85,STS94}. In the \emph{Appendix} \ref{Ap1}, we establish in detail a technical result concerning the relation between Satake and Vogan diagrams for involutions of semisimple compact Lie algebras. This result will be verified directly by diagram checking - it would however be nice to find a more conceptual proof. In \emph{Appendix} \ref{Ap2} we establish certain results on spherical vectors in quantized exterior products. In \emph{Appendix} \ref{Ap3}, we gather some explicit computations regarding the case of a symmetric pair of type $FII$. 

\section{Preliminaries}\label{SecPrelim}

\subsection{Quantized enveloping algebras}\label{SubsecQUE}

\begin{Def}
A \emph{Lie $*$-algebra} consists of a complex Lie algebra $\mfg$ together with an antilinear, antimultiplicative involution $*: \mfg \rightarrow \mfg$.  
\end{Def}

There is a one-to-one correspondence between Lie $*$-algebras and real Lie algebras by means of the correspondence
\[
(\mfg,*) \mapsto \mfg_* = \{X \in \mfg \mid X^* = -X\},
\]
\[
\mfs \mapsto (\mfs_{\C},*),\qquad (X+iY)^* = -X + iY,\quad X,Y\in \mfs,
\]
where $\mfs_{\C} = \mfs \underset{\R}{\otimes} \C$ is the complexification. One calls $\mfg_*$ the \emph{real form} of $\mfg$ associated to $*$. 

Let now $\mfg$ be a complex semisimple Lie algebra. We fix a Borel subalgebra $\mfb$ and associated Cartan subalgebra $\mfh$, and let $\mfb^-$ be the opposite Borel subalgebra. We write  $\{\alpha_r\mid r\in I\}$ for the set of simple positive roots, and let  $\Gamma$ be the corresponding Dynkin diagram on the set $I$. We write respectively $Q \supseteq \Delta \supseteq \Delta^+$ for the root lattice, the root system and the positive roots. We let $W$ be the Weyl group of $\mfg$, and we fix a $W$-invariant positive-definite bilinear real form  $(-,-)$ on $Q\otimes_{\Z}\R$. We write $d_r = (\alpha_r,\alpha_r)/2$, and use the standard notation $\alpha^{\vee} = \frac{2\alpha}{(\alpha,\alpha)}$ for coroots. We write  $P$ for the weight lattice, $P^+$ for its positive cone and $\{\varpi_r\mid r\in I\}$ for the set of fundamental weights. We let $A = (a_{rs})_{rs}$ be the associated Cartan matrix under the convention
\[
a_{rs} = (\alpha_r^{\vee},\alpha_s) = 2\frac{(\alpha_r,\alpha_s)}{(\alpha_r,\alpha_r)}.
\] 

We further fix Chevalley-Serre generators 
\[
h_r \in \mfh,\qquad e_r\in \mfb,\qquad f_r\in \mfb^-.
\]
Concretely, this means that we identify $\mfg$ with the abstract complex Lie algebra generated by 
\[
\msS = \{h_r,e_r,f_r \mid r\in I\}
\] 
such that 
\[
\lbrack h_r,h_s\rbrack =0,\quad \lbrack h_r,e_s\rbrack = a_{rs}e_s,\quad \lbrack h_r,f_s\rbrack = -a_{rs} f_s,\quad \lbrack e_r,f_s\rbrack = \delta_{rs} h_r
\]
and for $r\neq s$ the \emph{Serre relations}
\[
\ad_{e_r}^{1-a_{rs}}(e_s) = \ad_{f_r}^{1-a_{rs}}(f_s) = 0,
\]
where $\ad_x(y) = \lbrack x,y\rbrack$. We can then endow $\mfg$ with the unique Lie $*$-algebra structure such that 
\[
h_r^* = h_r,\qquad e_r^* = f_r.
\] 
The associated real Lie algebra $\mfu = \{X \in \mfg \mid X^* = -X\}$ is called the \emph{compact real form} of $\mfg$.   

We now introduce the quantized enveloping algebra of $\mfg$ and $\mfu$, see e.g.~ \cites{KS97,NT13} for details on the associated $*$-structures. Fix $0<q$ with $q\neq 1$. We denote by $U_q(\mfg)$ the quantized enveloping algebra of $\mfg$. Specifically, $U_q(\mfg)$ is generated by $K_{\omega},E_r,F_r$, where $r\in I$ and $\omega$ takes values in the integral weight lattice $P$, with commutation relations
\[
K_0 = 1,\quad K_{\omega}K_{\chi} = K_{\omega+ \chi},\quad K_{\omega} E_r = q^{(\omega,\alpha_r)} E_rK_{\omega}, \quad K_{\omega} F_r = q^{-(\omega,\alpha_r)}F_rK_{\omega},\quad \lbrack E_r,F_s\rbrack = \delta_{rs} \frac{K_{\alpha_r}-K_{\alpha_r}^{-1}}{q^{d_r}-q^{-d_r}}
\]
and the \emph{quantum Serre relations}, whose precise form we will not need in what follows. We will in the following use the shorthand $K_r = K_{\alpha_r}$ and $q_r = q^{d_r}$. We endow $U_q(\mfg)$ with the Hopf algebra structure 
\[
\Delta(E_r) = E_r\otimes 1+ K_r\otimes E_r,\qquad \Delta(F_r) = F_r\otimes K_r^{-1}+ 1\otimes F_r,\qquad \Delta(K_{\omega}) = K_{\omega}\otimes K_{\omega}.
\] 
We denote by $\varepsilon$ the counit, given by $\varepsilon(E_r) = \varepsilon(F_r) = 0$ and $\varepsilon(K_{\omega})=1$, and the antipode map by $S$, determined by 
\[
S(E_r) = - K_r^{-1}E_r,\qquad S(F_r) = - F_rK_r,\qquad S(K_{\omega}) = K_{\omega}^{-1}. 
\]
For $\alpha \in Q$ we write
\[
U_q(\mfg)_{\alpha} = \{X \in U_q(\mfg)\mid K_{\omega} X = q^{(\omega,\alpha)}XK_{\omega}\}.
\]

We write $U_q(\mfb) = U_q(\mfb^+)$ for the positive Borel part of $U_q(\mfg)$, generated by the $K_{\omega}$ and $E_r$, and $U_q(\mfb^-)$ for the negative Borel part generated by the $K_{\omega}$ and $F_r$. We write $U_q(\mfn) = U_q(\mfn^+)$ for the unital algebra generated by the $E_r$, $U_q(\mfn^-)$ for the unital algebra generated by the $F_r$, and $U_q(\mfh)$ for the algebra generated by the $K_{\omega}$. We denote $U_q(\mfu)$ for $U_q(\mfg)$ as a Hopf $*$-algebra with the $*$-structure
\[
K_{\omega}^* =K_{\omega},\qquad E_r^* = F_rK_r,\qquad F_r^* = K_r^{-1}E_r.
\]
Note that the antipode map $S$ is not $*$-preserving. To correct this, one introduces the \emph{unitary antipode} $R: U_q(\mfu) \rightarrow U_q(\mfu)$, which is a $*$-preserving, involutive, anti-multiplicative and anti-comultiplicative map determined on generators by 
\begin{equation}\label{EqUnitaryAntipode}
R(E_r)  = -q_rK_r^{-1}E_r,\qquad R(F_r) = -q_r^{-1}F_rK_r,\qquad R(K_{\omega}) = K_{\omega}^{-1}.
\end{equation}
We call \emph{admissible representation} of $U_q(\mfg)$ any representation on a finite dimensional complex vector space in which the $K_{\omega}$ assume positive values. Each admissible representation is spanned by joint eigenvectors of the $K_{\omega}$, called \emph{weight vectors}, and for $\xi$ a weight vector there exists a unique $\wt(\xi) \in P$ with 
\[
K_{\omega} \xi = q^{(\omega,\wt(\xi))}\xi,\qquad \forall \omega \in P. 
\]
Commonly we will use weight vectors when displaying a formula, with the implicit understanding that the formula extens (bi-)linearly to other vectors.

One can choose natural representatives $V_{\varpi}$ for the isomorphism classes of irreducible admissible representations, indexed by the positive integral weights $\varpi \in P^+$, and characterized by the existence of a one-dimensional space of highest weight vectors at weight $\varpi$, vanishing under the $E_r$. We can densily embed
\begin{equation}\label{EqCompletion}
U_q(\mfg) \subseteq \mcU_q(\mfg) := \prod_{\varpi} \End(V_{\varpi}),
\end{equation}
where $\mcU_q(\mfg)$ is endowed with the weak topology, that is $x_{\alpha} \rightarrow x$ if $\pi_{\varpi}(x_{\alpha})\rightarrow \pi_{\varpi}(x)$ for all $\varpi \in P^+$. The coproduct $\Delta$ then extends continuously to a homomorphism
\[
\Delta:  \mcU_q(\mfg) \rightarrow \mcU_q(\mfg) \hat{\otimes}\mcU_q(\mfg) := \prod_{\varpi,\varpi'} \End(V_{\varpi})\otimes \End(V_{\varpi'}),
\]
coassociative in the natural way, where the symbol $\hat{\otimes}$ denotes the weak closure of a tensor product. Similarly the antipode and unitary antipode can be extended to $\mcU_q(\mfg)$. In general, we will use also the notation $\mcU_q(\mfn)$ etc.~ to denote the weak closure of the respective subalgebra of $\mcU_q(\mfg)$. We note that if $H \subseteq U_q(\mfg)$ is a $*$-subalgebra, the weak closure of $H$ will coincide with its bicommutant inside $\mcU_q(\mfg)$.

We endow each $V_{\varpi}$ with a Hilbert space structure, unique up to a non-zero positive constant, such that it becomes a $*$-representation of $U_q(\mfu)$. The inclusion \eqref{EqCompletion} then becomes an embedding of $*$-algebras, and we will consequently write the right hand side $*$-algebra as $\mcU_q(\mfu)$. In the following, we write $V^*$ for the contragredient representation of $V$, where $V^* = \{\xi^*\mid \xi \in V\}$ is a copy of the conjugate-linear Hilbert space of $V$ realized as the dual of $V$ by the scalar product, $\xi^* = \langle \xi,-\rangle$, and endowed with the left $U_q(\mfg)$-module structure
\begin{equation}\label{EqContragredient}
X \cdot \xi^* = \xi^* \circ S(X).
\end{equation}
To make this a $*$-representation of $U_q(\mfu)$, the space $V^*$ has to be endowed with a Hilbert space structure different from the canonical one, but this will not come in to play in what follows. 

Let $\msR$ be the universal $R$-matrix for $U_q(\mfu)$, so 
\[
\msR \in \mcU_q(\mfb^+)\hat{\otimes} \mcU_q(\mfb^-) \subseteq \prod_{\varpi,\varpi'} \End(V_{\varpi})\otimes \End(V_{\varpi'})
\] 
and
\begin{equation}\label{EqPropR1}
(\Delta\otimes \id)\msR = \msR_{13}\msR_{23},\qquad (\id\otimes \Delta)\msR = \msR_{13}\msR_{12},\qquad \msR\Delta(-)\msR^{-1} = \Delta^{\opp},\qquad \msR^* = \msR_{21}. 
\end{equation}
It is completely determined by the above relations and the rule
\begin{equation}\label{EqPropR2}
\msR(\xi\otimes \eta) = q^{-(\wt(\xi),\wt(\eta))}\xi\otimes \eta
\end{equation}
for a highest weight vector $\xi$ and a lowest weight vector $\eta$. We then have that $\msR = \widetilde{\msR} \msQ$, where  for general weight vectors $\xi,\eta$
\begin{equation}\label{EqPropR3}
\msQ(\xi\otimes \eta) = q^{-(\wt(\xi),\wt(\eta))}\xi\otimes \eta,\qquad \widetilde{\msR} = \sum_{\alpha \in Q^+} \widetilde{\msR}_{\alpha}
\end{equation}
with $\widetilde{\msR}_{\alpha} \in U_q(\mfn)_{\alpha} \otimes U_q(\mfn^-)_{-\alpha}$ and where the sum converges weakly. We have for example
\begin{equation}\label{EqPropR4}
\widetilde{\msR}_{0} = 1\otimes 1,\qquad \widetilde{\msR}_{\alpha_r} = (q_r^{-1}-q_r) E_r \otimes F_r.
\end{equation}

\subsection{Quantized function algebras}\label{SubsecQFA}

Let $\mcO_q(G) = (\mcO_q(G),\Delta,\varepsilon,S)$ be the dual Hopf algebra of matrix coefficients for $U_q(\mfg)$ in admissible representations. We write the pairing as
\[
U_q(\mfg) \times \mcO_q(G) \rightarrow \C,\qquad (X,f) \mapsto (X,f) = (f,X) = X(f) = f(X).
\]
We equip $\mcO_q(G)$ with the $*$-structure dual to that of $U_q(\mfu)$,
\[
(X,f^*) = \overline{(S(X)^*,f)},\qquad X \in U_q(\mfu),f\in \mcO_q(G),
\]
and write the resulting Hopf $*$-algebra as $\mcO_q(U)$ in the appropriate contexts. 
\begin{Rem}
Performing the analogues of the above constructions at $q=1$, we obtain the Hopf algebra $\mcO(G)$ of regular functions on the connected, simply connected complex affine group $G$ having $\mfg$ as its complex Lie algebra. The given $*$-structure on $\mcO(G)$ endows $G \cong \Spec(\mcO(G))$ with a complex conjugate involution, determined by 
\[
f(\bar{g}) := \overline{f^*(g)},\qquad g\in G,f\in \mcO(G).
\]
The real affine group variety $U \subseteq G$ of $*$-preserving characters, i.e.~ of elements $g\in G$ with $g = \bar{g}$, is then the connected, simply connected compact Lie group with Lie algebra $\mfu$. 
\end{Rem}

When $(V_{\pi},\pi)$ is a $U_q(\mfu)$-representation, we consider
\[
Y_{\pi}  \in \End(V_{\pi})\otimes \mcO_q(G)
\]
for the associated corepresentation matrix of matrix coefficients, and 
\[
Y(\xi,\eta) = (\xi^*\otimes \id)Y_{\pi}(\eta\otimes \id) \in \mcO_q(G),\qquad \xi,\eta \in V_{\pi}
\]
for the matrix coefficients. When considering these as unitary corepresentations of $\mcO_q(U)$, we will rather write the coefficients as $U(\xi,\eta)$. 

We can consider the Hopf algebra pairings between $\mcO_q(G)$ and $U_q(\mfb^{\pm})$ obtained by restriction, and we let $\mcO_q(B) = \mcO_q(B^+)$ and $\mcO_q(B^-)$ be the respective coimages of $\mcO_q(G)$ under this restriction. We write the corresponding Hopf algebra quotient homomorphisms as
\[
\msP_{\pm}: \mcO_q(G) \twoheadrightarrow \mcO_q(B^{\pm}).
\]
We write
\[
T_{\pi}^{\pm} = (\id\otimes \msP_{\pm})Y_{\pi} \in \End(V_{\pi})\otimes \mcO_q(B^{\pm}).
\] 

We let 
\[
\mbr: \mcO_q(G) \times \mcO_q(G) \rightarrow \C,\qquad (f,g)\mapsto (\msR,f\otimes g)
\]
be the natural skew pairing of $\mcO_q(G)$ with itself. Then $\mbr$ factors over a skew pairing
\[
\mcO_q(B)\times \mcO_q(B^-) \rightarrow \C. 
\]
In fact, we get homomorphisms of Hopf algebras
\begin{equation}\label{EqDefiotapm}
\iota_{-}: \mcO_q(B) \rightarrow U_q(\mfb^-)^{\cop},\quad f \mapsto (f\otimes \id)\msR,\qquad \iota_+: \mcO_q(B^-) \rightarrow U_q(\mfb^+)^{\cop},\quad f \mapsto (\id\otimes f)\msR^{-1}.  
\end{equation}
The following result is well-known in the formal setting, and is an instance of Drinfeld's duality between quantized universal enveloping algebras and quantized function algebras \cite{Dri87,Gav02}. A proof in the non-formal setting can be found in \cite{Jos95}*{Corollary 9.2.12}. We repeat the proof, mainly to introduce notation. 

\begin{Prop}\label{PropiotIso}
The maps $\iota_{\pm}$ in \eqref{EqDefiotapm} are isomorphisms. 
\end{Prop}
\begin{proof}
By \eqref{EqPropR3} and \eqref{EqPropR4}, we have
\[
\iota_-(T^+_{\varpi}(\xi,\xi)) = K_{-\wt(\xi)},\qquad \iota_-(T^+_{\varpi_r}(\xi_{\varpi_r},F_r\xi_{\varpi_r})) = (q_r^{-1}-q_r) F_rK_{\alpha_r-\varpi_r},
\]
\[
\iota_+(T_{\varpi}^-(\xi,\xi)) = K_{\wt(\xi)},\qquad \iota_+(T^-_{\varpi_r}(F_r\xi_{\varpi_r},\xi_{\varpi_r})) = q_r^{-2}(q_r-q_r^{-1})E_rK_{\varpi_r-\alpha_r}.
\]
Since the image of $\iota_{\pm}$ will be closed under the antipode map, this proves that $\iota_{\pm}$ are surjective. If then $f\in \mcO_q(B)$ with $\iota_-(f) = 0$, we have for all $g\in \mcO_q(G)$ that 
\[
(S^{-1}(g),\iota_-(f)) = (f,\iota_+(\msP_-(g))) =0. 
\]
As $\msP_-$ and $\iota_+$ are surjective, and the pairing between $\mcO_q(B)$ and $U_q(\mfb)$ is non-degenerate by definition, it follows that $f=0$, and $\iota_-$ injective. Similarly one shows $\iota_+$ injective.
\end{proof}
In the following we write
\begin{equation}\label{EqNameImIota}
\iota_+^{-1}(E_r) = X_r^+,\qquad \iota_-^{-1}(F_r) = X_r^-\qquad \iota_{\pm}^{-1}(K_{\omega}) = L_{\omega}^{\pm}.
\end{equation}
As in the proof of Proposition \ref{PropiotIso}, we have
\[
T^+_{\varpi}(\xi,\xi) = L_{-\wt(\xi)}^-,\qquad T^+_{\varpi_r}(\xi_{\varpi_r},F_r\xi_{\varpi_r}) = (q_r^{-1}-q_r) X_r^-L_{\alpha_r-\varpi_r}^-,
\]
\[
T_{\varpi}^-(\xi,\xi) = L_{\wt(\xi)}^+,\qquad T^-_{\varpi_r}(F_r\xi_{\varpi_r},\xi_{\varpi_r}) = q_r^{-2}(q_r-q_r^{-1})X_r^+L_{\varpi_r-\alpha_r}^+,
\]
The skew pairing between generators is then determined by 
\[
\mbr(L_{\varpi}^-,L_{\chi}^+) = q^{(\varpi,\chi)},\qquad \mbr(X_r^-,X_s^+) = \frac{\delta_{rs}}{q_r-q_r^{-1}},\qquad \mbr(L_{\varpi}^-,X_r^+) = \mbr(X_r^-,L_{\varpi}^+)=0. 
\]

Let $\mcO_q(\overline{G})$ be an antilinear, anti-homomorphic, cohomomorphic copy of $\mcO_q(G)$, with the copy of the element $f$ written as $f^{\dag}$. Then we can view the tensor product Hopf algebra $ \mcO_q(G) \otimes \mcO_q(\overline{G})$ as a Hopf $*$-algebra, which we will denote by $\mcO_q^{\com}(G_{\R})$, by the $*$-structure
\[
(f\otimes g^{\dag})^{\dag}= g\otimes f^{\dag}. 
\]
In the following we will drop the tensor product symbol, and simply write elements of $\mcO_q^{\com}(G_{\R})$ as $fg^{\dag} = g^{\dag}f$.

\begin{Rem} 
At $q=1$, one has that $\mcO^{\com}(G_{\R})$ is the $*$-algebra of regular functions on $G$ viewed as a real algebraic variety by the embedding
\[
G \hookrightarrow G\times \bar{G},\quad x \mapsto (x,\bar{x}),
\]
where $\bar{G}$ is an anti-holomorphic copy of $G$ and where $G\times \bar{G}$ is endowed with the complex conjugation $\overline{(x,\bar{y})} = (y,\bar{x})$. We then have
\[
(fg^{\dag})(x,\bar{y}) = f(x) \overline{g(y)}.
\]
Alternatively, we may view $\mcO^{\com}(G_{\R})$ as generated by the holomorphic and anti-holomorphic regular functions on $G$. For the above particular quantization $\mcO_q^{\com}(G_{\R})$, the quasi-classical Poisson structure is such that the holomorphic and antiholomorphic functions Poisson commute. In the following section, we will consider deformations where there a is more interesting interaction between the holomorphic and antiholomorphic structures. 
\end{Rem}

Similarly, we can form $\mcO_q(\overline{B})$ and $\mcO_q^{\com}(B_{\R})$. In the latter case, we will identify $\mcO_q(\overline{B})$ with $\mcO_q(B^-)$ by the identification
\begin{equation}\label{EqIdentBorel}
\mcO_q(B^-) \rightarrow \mcO_q(\overline{B}),\qquad f \mapsto (f^*)^{\dag},
\end{equation}
where $*: \mcO_q(B^-) \rightarrow \mcO_q(B)$ is again determined by 
\[
(X,f^*) = \overline{(S(X)^*,f)},\qquad f\in \mcO_q(B^-),X \in U_q(\mfb),
\]
so in particular 
\begin{equation}\label{EqMsP+}
\msP_+(f^*) = \msP_-(f)^*.
\end{equation} 
Moreover, from
\[
((f^*\otimes \id)\msR)^* = (f\otimes \id)(((S\otimes \id)\msR)^*) = (f\otimes \id)((\msR^{-1})^*) = (\id\otimes f)(\msR^{-1})
\]
we see that $*$ is compatible with the $\iota_{\pm}$-maps,
\[
\iota_-(f^*) = \iota_+(f)^*,\qquad f \in \mcO_q(B^-).
\] 
We then also write $*$ for the inverse map $*:\mcO_q(B) \rightarrow \mcO_q(B^-)$. Note that as we are assuming our admissible representations $\pi$ to be $*$-preserving for the compact $*$-structure, we also have that
\begin{equation}\label{EqStarT}
(T_{\pi}^{\pm})^* = (T_{\pi}^{\mp})^{-1}. 
\end{equation}

\subsection{Lusztig braid operators}

The following results will only be needed from Section \ref{SecSymm} onwards. 

For $r \in I$, we let $T_{r}$ be the \emph{Lusztig braid operator}
\[
T_r\in \mcU_q(\mfu),\qquad T_r \xi = \underset{-a+b-c = (\wt(\xi),\alpha_r^{\vee})}{\sum_{a,b,c\geq 0}} (-1)^b q_r^{b-ac} E_r^{(a)}F_r^{(b)}E_r^{(c)}\xi,\qquad \xi \in V_{\pi},
\]
where 
\[
E_r^{(a)} = \frac{1}{[a]_{q_r}!}E_r^a
\] 
in standard notation with the convention 
\[
[a]_q = \frac{q^a-q^{-a}}{q-q^{-1}},\qquad [a]_q! = [1]_q [2]_q \ldots [a-1]_q [a]_q.
\] 
We will need to know the behaviour of the $T_r$ under the $*$-operation. In the following, we will interpret the maximal torus $T = \exp(i(\mfh \cap \mfu)) \subseteq U$ as the space of unitary characters of the integral weight lattice $P$, $\omega\mapsto t^{\omega}$, so that we have a natural embedding
\[
T \subseteq \mcU_q(\mfu),\qquad t\xi = t^{\wt(\xi)}\xi,\qquad \xi\in V_{\pi}.
\]
We can then consider
\[
\mcS_{r} = e^{\pi i \alpha_r^{\vee}}\in T \subseteq \mcU_q(\mfu),\qquad \mcS_r\xi = (-1)^{(\wt(\xi),\alpha^{\vee}_r)} \xi,\qquad \xi \in V_{\pi}.  
\]

\begin{Lem}\label{LemTstarR}
We have
\begin{equation}\label{EqForm}
T_r^* = R(T_r) = T_{r} \mcS_{r} =\mcS_rT_r,
\end{equation}
where we recall that $R$ is the unitary antipode of \eqref{EqUnitaryAntipode}, extended to $\mcU_q(\mfu)$. 
\end{Lem}
\begin{proof}
As in \cite{Jan96}*{Section 8}, it is sufficient to consider the case $U_q(\mfsu(2))$ for $\mfsu(2)$ with its positive root $\alpha$ such that $(\alpha,\alpha) =2$ and in particular $\alpha^{\vee} = \alpha$. We then write the generators of $U_q(\mfsu(2))$ as $\{K,E,F\}$ with Lusztig braid operator $T$. It is also sufficient to verify \eqref{EqForm} in the Hilbert space $V=V_n$ with orthonormal basis $e_0,\ldots,e_n$ and with irreducible representation 
\[
Ke_k = q^{n-2k} e_k,\qquad Ee_k = q^{(n-2k+2)/2}[n+1-k]_q^{1/2}[k]_q^{1/2}e_{k-1},\qquad Fe_k = q^{-(n-2k)/2}[n-k]_q^{1/2}[k+1]_q^{1/2} e_{k+1},
\]
with the right hand expressions $=0$ if ill-defined.
Then
\begin{equation}\label{EqActT}
Te_k = (-1)^{n-k} q^{n/2}q^{k(n-k)}e_{n-k},
\end{equation}
and it follows that 
\[
T^*e_k = (-1)^n Te_k = (-1)^{n-2k}Te_k  = (-1)^{(\wt(v_k),\alpha)}Te_k .
\]
Let us now show that $R(T) = T^*$, so $R(T)^* = T$. Note that the conjugate Hilbert space $\overline{V}$ (with its usual Hilbert space structure) can be endowed with the $*$-representation $X \bar{v} = \overline{R(X)^*v}$. Since $\bar{v}$ has negative the weight of $v$, we then have, following again the notation of \cite{Jan96}*{Section 8},
\[
R(T)^*v = \overline{T \bar{v}} = \underset{-a+b-c =-m}{\sum_{a,b,c\geq 0}} (-1)^b q^{b-ac} (-q)^{a-b+c} F^{(a)}E^{(b)}F^{(c)} v = (-q)^m \;{}^{\omega}T v,
\]
where $m =  (\wt(v),\alpha)$. From \cite{Jan96}*{8.6.(7)} we now see that $R(T)^* = T$. 
\end{proof}

Recall that $W$ denotes the Weyl group of $\mfg$. For $r\in I$ we write $s_r$ for the simple root reflections generating $W$. Let $w_0$ be the longest element in $W$, and choose a specific reduced expression for $w_0$,
\begin{equation}\label{EqRedExpr}
w_0 = s_{r_1}\ldots s_{r_N}.
\end{equation}
We then write
\[
T_{w_0} = T_{r_1}\ldots T_{r_N} \in \mcU_q(\mfu),
\]
which is independent of the choice of reduced expression for $w_0$. Further write 
\begin{equation}\label{EqDefS}
\mcS_0 = e^{2\pi i \rho^{\vee}}\in T,
\end{equation}
where
\[
\rho^{\vee} = \frac{1}{2}\sum_{\alpha \in \Delta^{+}} \alpha^{\vee}.
\]

\begin{Prop}\label{PropStarT}
We have
\[
T_{w_0}^{*} = R(T_{w_0}) =  T_{w_0} \mcS_0 = \mcS_0 T_{w_0}.
\]
\end{Prop} 
\begin{proof}
An easy consideration of weight spaces shows that 
\[
T_{r} e^{\pi i \alpha^{\vee}} T_r^{-1} = e^{\pi i s_r(\alpha^{\vee})},\qquad \alpha \in \Delta. 
\]
The elements in $\Delta^+$ can be enumerated as
\[
\beta_1 = \alpha_{r_1}, \quad \beta_2 = s_{r_1}(\alpha_{r_2}), \quad \ldots, \quad \beta_N = s_{r_1}\ldots s_{r_{N-1}}(\alpha_{r_N}). 
\]
Then it follows from \eqref{EqForm} and the above observation, together with the fact that $T_{w_0}$ is independent of the choice of reduced decomposition of $w_0$, that
\[
T_{w_0}^* = T_{r_N}^*\ldots T_{r_1}^* =T_{w_0} e^{2\pi i \rho^{\vee}}  = e^{2\pi i \rho^{\vee}}  T_{w_0}. 
\]
The identity for $R$ follows immediately from the fact that $R(-)^*$ leaves each factor of $T_{w_0}$ invariant, and hence the whole of $T_{w_0}$. 
\end{proof}

Denote now by $\Ad(T_r)$ the Lusztig algebra automorphism of $U_q(\mfg)$, uniquely determined by 
\[
\Ad(T_r)(X) = T_rXT_r^{-1},\qquad X \in U_q(\mfg).
\]
From \cite{Kol14}*{Lemma 3.4}, we obtain 
\begin{equation}\label{EqIdAdT_0}
\Ad(T_{w_0})(E_r)^* = -E_{\tau_0(r)},\qquad \Ad(T_{w_0})(F_r)^* = -F_{\tau_0(r)},\qquad \Ad(T_{w_0})(K_{\omega})^* = K_{-\tau_0(\omega)}, 
\end{equation}
where $\tau_0$ is the automorphism induced on the Dynkin diagram by the action of $-w_0$. Together with the definition of $T_{r}$ and \eqref{EqForm}, we obtain the following lemma, which also follows from the fact that the $T_r$ satisfy the braid relations.

\begin{Lem}\label{LemCommTw0Tr}
The following identity holds in $\mcU_q(\mfu)$:
\[
T_{w_0}T_{r} T_{w_0}^{-1}  = T_{\tau_0(r)}.
\]
\end{Lem}

\section{Twist-braided Hopf algebras and associated characters}\label{SecTwistBraid}

We resume the notation from Sections \ref{SubsecQUE} and \ref{SubsecQFA}.

\subsection{Endomorphisms of $U_q(\mfb)$}

Let $\tau$ be an involutive Dynkin diagram automorphism. We also write $\tau$ for the corresponding linear automorphism of the weight lattice $P$ determined by 
\[
\tau(\varpi_r) = \varpi_{\tau(r)}.
\]
We can extend $\tau$ to a Hopf algebra isomorphism of $U_q(\mfg)$, compatible with the compact $*$-structure, such that
\[
\tau(E_r) = E_{\tau(r)},\qquad \tau(F_r) = F_{\tau(r)},\qquad \tau(K_{\omega}) = K_{\tau(\omega)}. 
\]
By duality, we obtain a Hopf algebra automorphism $\tau$ of $\mcO_q(G)$.

Let $\epsilon: I \rightarrow \R$ be a real-valued $\tau$-invariant function on $I$. We can extend $\epsilon$ to a semigroup homomorphism 
\[
\epsilon: (Q^+,+) \rightarrow (\R,\cdot).
\]

From the couple $(\tau,\epsilon)$ we construct a Hopf algebra endomorphism 
\[
\nu: U_q(\mfb) \rightarrow U_q(\mfb),\qquad K_{\omega}\mapsto K_{\tau(\omega)},\quad E_r \mapsto \epsilon_r E_{\tau(r)}.
\]
Denote also by $\nu$ the corresponding Hopf algebra endomorphism of $U_q(\mfb^-)$ determined by 
\[
\nu(X) = \nu(X^*)^*,\qquad X\in U_q(\mfb^-),
\] 
where we restrict $*$ from $U_q(\mfu)$ to a conjugate-linear algebra anti-homomorphism from $U_q(\mfb^{\pm})$ to $U_q(\mfb^{\mp})$. Concretely,
\[
\nu: U_q(\mfb^-) \rightarrow U_q(\mfb^-),\qquad K_{\omega}\mapsto K_{\tau(\omega)},\quad F_r \mapsto \epsilon_r F_{\tau(r)}.
\]
Note that because of the involutivity of $\tau$ and $\tau$-invariance of $\epsilon$, we have 
\begin{equation}\label{EqInvNu}
(\nu\otimes \id)\msR = (\id\otimes \nu)\msR.
\end{equation}
We will write $\End_*(U_q(\mfb))$ for the class of all homomorphisms $\nu$ of $U_q(\mfb)$ of the above form. Then $\nu$ completely determines $\tau$ and $\epsilon$, and we write 
\[
\nu = \nu_{\tau,\epsilon},\qquad \tau = \tau_{\nu},\qquad \epsilon = \epsilon_{\nu}.
\] 
When $\tau = \id$, we will also write $\nu = \nu_{\epsilon}$. On the other hand, we have $\nu_{\tau,\epsilon} = \tau$ in case $\epsilon_r = 1$ for all $r$. 

\begin{Def}
We say that $\nu$ is of \emph{symmetric type} if $\epsilon_r^2 = 1$ for all $r$. We say $\nu$ is of \emph{flag type} if $\epsilon_r^2 = \epsilon_r$ for all $r$ and $\tau =\id$. 
\end{Def}

\begin{Rem}
It is not hard to show that a general endomorphism of $U_q(\mfb)$ satisfying \eqref{EqInvNu} must be of the form 
\[
E_r \mapsto \epsilon_r E_{\tau(r)},\qquad K_{\omega} \mapsto K_{\sigma(\omega)}
\]
where $\tau$ is an involution of $I$, $\epsilon_r \in \C$ satisfies $\epsilon_{\tau(r)} = \overline{\epsilon_r}$, and $\sigma$ is an endomorphism of $P$ with $\sigma^{\top} = \sigma$ and $\sigma(\alpha_r) = \alpha_{\tau(r)}$ for all $r\in I$ with $\epsilon_r\neq 0$. 
\end{Rem}

\begin{Rem}
Note that for $\nu \in \End_*(U_q(\mfb))$, the endomorphisms $\nu$ of $U_q(\mfb^{\pm})$ glue together to an algebra $*$-endomorphism of $U_q(\mfu)$ if and only if $\nu$ is of symmetric type, in which case it defines a Hopf $*$-algebra automorphism $\nu$ of $U_q(\mfu)$. By duality, we obtain in this case a Hopf $*$-algebra automorphism $\nu$ of $\mcO_q(U)$. 
\end{Rem}

\begin{Rem}
By rescaling, one can always reduce to the case $\epsilon_r \in \{-1,0,1\}$, but the extra flexibility of an arbitrary $\epsilon_r$ can be convenient with respect to the \emph{contraction method} \cite{IW53,DeC13}.
\end{Rem}

For $\nu \in \End_*(U_q(\mfb))$ we write
\[
\msR_{\nu} = (\nu\otimes \id)\msR, \qquad \mbr_{\nu}(f,g) = (\msR_{\nu},f\otimes g), \qquad f,g\in\mcO_q(G).
\]
From \eqref{EqInvNu}, $(S\otimes S)\msR = \msR$ and $\msR^* = \msR_{21}$ we find
\begin{equation}\label{EqrnuStar}
\mbr_{\nu}(f,g^*) = \overline{\mbr_{\nu}(g,f^*)}. 
\end{equation}
Consider the following bilinear functional
\[
\omega_{\nu}: \mcO_q^{\com}(G_{\R}) \times \mcO_q^{\com}(G_{\R}) \rightarrow \C,\quad \omega_{\nu}(fg^{\dag},hk^{\dag}) = \varepsilon(f) \mbr_{\nu}(h,g^*)\overline{\varepsilon(k)}.
\]
 Then $\omega_{\nu}$ is a convolution invertible real $2$-cocycle functional with 
\begin{equation}\label{EqomInv}
\omega_{\nu}^{-1}(fg^{\dag},hk^{\dag}) = \varepsilon(f) \mbr_{\nu}(h,S(g)^*)\overline{\varepsilon(k)},\qquad f,g,h,k\in \mcO_q(G).
\end{equation}
By \eqref{EqrnuStar} we also have
\begin{equation}\label{EqomStar}
\omega_{\nu}(f^{\dag},g^{\dag}) = \overline{\omega_{\nu}(g,f)},\qquad f,g\in \mcO_q^{\com}(G_{\R}).
\end{equation}
The following definition extends the usual `complexification' of a Hopf $*$-algebra \cite{Maj93b}, see also \cite[Section 7.3]{Maj95}.
\begin{Def}
For  $\nu,\mu\in \End_*(U_q(\mfb))$ the \emph{$(\nu,\mu)$-Drinfeld double} $\mcO_q^{\nu,\mu}(G_{\R})$ is defined as the vector space $\mcO_q^{\com}(G_{\R})$ endowed with the new multiplication
\[
m_{\nu,\mu}(f,g) = \omega_{\nu}(f_{(1)},g_{(1)}) f_{(2)}g_{(2)} \omega_{\mu}^{-1}(f_{(3)},g_{(3)}),\qquad f,g\in \mcO_q^{\com}(G_{\R}),
\]
and the original $*$-structure. 
\end{Def}

As the $\omega_{\nu}$ are $2$-cocycle functionals for $\mcO_q^{\com}(G_{\R})$ satisfying \eqref{EqomStar}, it follows that the $\mcO_q^{\nu,\mu}(G_{\R})$ are associative $*$-algebras, where by the particular form of $\omega_{\nu}$ one has
\[
m_{\nu,\mu}(f,g^{\dag}) = fg^{\dag},\qquad f,g\in \mcO_q(G).
\]
The $\mcO_q^{\nu,\mu}(G_{\R})$ form a connected cogroupoid with compatible $*$-structure \cite{Bic14}*{Definition 2.4 and Definition 3.14} by 
\[
\Delta_{\nu\mu}^{\kappa}: \mcO_q^{\nu,\mu}(G_{\R}) \rightarrow \mcO_q^{\nu,\kappa}(G_{\R}) \otimes \mcO_q^{\kappa,\mu}(G_{\R}),\quad fg^{\dag}\mapsto f_{(1)}g_{(1)}^{\dag}\otimes f_{(2)}g_{(2)}^{\dag},\qquad f,g\in \mcO_q(G),
\]
where we use Sweedler notation for $(\mcO_q(G),\Delta)$. In particular,
\[
(\mcO_q^{\nu}(G_{\R}),\Delta_{\nu}) = (\mcO_q^{\nu,\nu}(G_{\R}),\Delta_{\nu,\nu}^{\nu})
\] 
are Hopf $*$-algebras. We also write
\[
\mcO_q(G_{\R}) = \mcO_q^{\id}(G_{\R}).
\]
In the following, we will sometimes in general simply write $\Delta = \Delta_{\nu\mu}^{\kappa}$ as this will not lead to confusion.

\begin{Rem}
The above interchange relations lead to a family of Poisson structures $\{-,-\}_{\nu,\mu}$ on $G$ considered as a real manifold. They can be viewed as manifolds with an \emph{affine Poisson structure} \cites{Wei90,Lu90}.
\end{Rem}

For $\pi,\pi'$ representations of $U_q(\mfg)$, let us write
\[
\msR^{\pi,\pi'} = (\pi\otimes \pi')\msR.
\]

\begin{Lem}\label{LemFundInt}
Let $\nu,\mu \in \End_*(U_q(\mfb))$. In $\mcO_q^{\nu,\mu}(G_{\R})$ we have the following defining commutation relations between the holomorphic and the antiholomorphic part: with $Y = Y_{\pi}$ and $Y' = Y_{\pi'}$,  
\begin{equation}\label{EqDefHolAntihol}
Y_{13}' \msR_{\mu,12}^{\pi',\pi}Y_{23}^{\dag} = Y_{23}^{\dag} \msR_{\nu,12}^{\pi',\pi}Y_{13}'
\end{equation}
as identities in $\End(V_{\pi'})\otimes \End(V_{\pi})\otimes \mcO_q^{\nu,\mu}(G_{\R})$.
\end{Lem}
\begin{proof}
We can rewrite \eqref{EqDefHolAntihol} as
\[
(Y_{23}^{\dag})^{-1}Y_{13}' = \msR_{\nu,12}^{\pi',\pi} Y_{13}' (Y_{23}^{\dag})^{-1} (\msR_{\mu,12}^{\pi',\pi})^{-1}.
\]
Then this relation follows straightforwardly from the definition of the product in $\mcO_q^{\nu,\mu}(G_{\R})$, using that 
\[
(Y^{\dag})^{-1} = (Y^{-1})^{\dag} = ((\id\otimes S)Y)^{\dag},
\]
together with the fact that $(\id\otimes S^{-1})\msR = \msR^{-1}$ and $S(Y(\xi,\eta))^* = Y(\eta,\xi)$.
\end{proof}

As $\mbr_{\nu}$ factors over a skew pairing between $\mcO_q(B^+)$ and $\mcO_q(B^-)$, it follows that similarly as above we can form the $(\nu,\mu)$-Drinfeld double $\mcO_q^{\nu,\mu}(B_{\R})$. Using \eqref{EqIdentBorel} and \eqref{EqMsP+}, we then have a unique surjective $*$-homomorphism
\begin{equation}\label{EqMapP}
\msP: \mcO_q^{\nu,\mu}(G_{\R}) \rightarrow \mcO_q^{\nu,\mu}(B_{\R}),\qquad fg^{\dag} \mapsto \msP_+(f)\msP_-(g^*)
\end{equation}
extending the homomorphism $\mcO_q(G) \rightarrow \mcO_q(B)$. This map preserves also the comultiplications $\Delta_{\nu,\mu}^{\kappa}$. In particular, we have a surjective map of Hopf $*$-algebras $\mcO_q^{\nu}(G_{\R}) \rightarrow \mcO_q^{\nu}(B_{\R})$, where $\mcO_q^{\nu}(B_{\R}) = \mcO_q^{\nu,\nu}(B_{\R})$. 

Recall now the notation from \eqref{EqNameImIota}. Then we have for $\nu = (\tau_{\nu},\epsilon)$ and $\mu = (\tau_{\mu},\eta)$ that
\[
\omega_{\nu}(L_{\omega}^+,L_{\chi}^-) = q^{(\omega,\tau_{\nu}(\chi))},\quad \omega_{\nu}(X_r^+,X_s^-) = \frac{\epsilon_r \delta_{r,\tau_{\nu}(s)}}{q_r-q_r^{-1}},\quad \omega_{\nu}(L_{\omega}^+,X_s^-) = \omega_{\nu}(X_r^+,L_{\omega}^-)= 0,
\]
\[
\omega_{\mu}^{-1}(L_{\omega}^+,L_{\chi}^-) = q^{-(\omega,\tau_{\mu}(\chi))},\quad \omega_{\mu}^{-1}(X_r^+,X_s^-) = \frac{\eta_r \delta_{r,\tau_{\mu}(s)}}{q_r^{-1}-q_r},\quad \omega_{\mu}^{-1}(L_{\omega}^+,X_s^-) = \omega_{\mu}^{-1}(X_r^+,L_{\omega}^-)= 0.
\]
Moreover,
\[
\Delta(L_{\varpi}^{\pm}) = L_{\varpi}^{\pm}\otimes L_{\varpi}^{\pm},\qquad \Delta(X_r^+) = X_r^+ \otimes L_{r}^+ + 1\otimes X_r^+,\qquad \Delta(X_r^-) = X_r^- \otimes 1+ (L_r^-)^{-1}\otimes X_r^-.
\]
Hence in $\mcO_q^{\nu,\mu}(B_{\R})$  the  following defining interchange rules hold, see also \cite{DCNTY17}*{Section 3}:
\[
L_{\omega}^+ L_{\chi}^- = q^{(\omega,\tau_{\nu}(\chi)-\tau_{\mu}(\chi))}L_{\chi}^-L_{\omega}^+, 
\]
\begin{equation}\label{EqCommRelBR}
L_{\omega}^- X_r^+ = q^{(\alpha_r,\tau_{\mu}(\omega))}X_r^+L_{\omega}^-,\qquad L_{\omega}^+X_r^- =  q^{-(\alpha_r,\tau_{\nu}(\omega))} X_r^- L_{\omega}^+,
\end{equation}
\[
\lbrack X_r^+,X_s^-\rbrack = \frac{\delta_{r,\tau_{\nu}(s)} \epsilon_r L_{r}^+ - \delta_{r,\tau_{\mu}(s)} \eta_s (L_{s}^-)^{-1}}{q_r-q_r^{-1}}.
\]

\subsection{The quantized enveloping Lie $*$-algebra $U_q(\mfg_{\nu})$ and its dual $\mcO_q(G_{\nu})$}\label{SubsecNuDef}

For $\epsilon$ as above a real-valued function on $I$, let $U_q^{\epsilon}(\mfg)$ be the Hopf algebra obtained by changing in $U_q(\mfg)$ the commutation relation between $E_r,F_s$ to
\[
\lbrack E_r,F_s\rbrack = \delta_{rs} \epsilon_r \frac{K_r - K_r^{-1}}{q_r-q_r^{-1}}. 
\]
Then with $\tau$ an involutive automorphism of the Dynkin diagram preserving $\epsilon$, we have on $U_q^{\epsilon}(\mfg)$ the Hopf $*$-algebra structure
\[
K_{\omega}^\dag = K_{\tau(\omega)},\qquad E_r^{\dag} = F_{\tau(r)}K_{\tau(r)},\qquad F_r^{\dag} = K_{\tau(r)}^{-1}E_{\tau(r)}. 
\]
\begin{Def}
For $\nu = \nu_{\tau,\epsilon}$, we denote by $U_q(\mfg_{\nu})$ the Hopf $*$-algebra obtained by endowing $U_q^{\epsilon}(\mfg)$ with the $*$-structure $\dag$.
\end{Def}
 From \eqref{EqCommRelBR}, we see immediately that we  have a surjective Hopf $*$-algebra morphism 
\begin{equation}\label{EqDefiota}
\iota_{\nu}: \mcO_q^{\nu}(B_{\R}) \rightarrow U_q(\mfg_{\nu})^{\cop},\qquad X_r^+\mapsto E_{\tau(r)},\quad X_r^- \mapsto F_{r},\quad L_{\omega}^+ \mapsto K_{\tau(\omega)},\quad L_{\omega}^- \mapsto K_{\omega},
\end{equation}
with kernel generated by the central elements $L_{-\tau(\omega)}^+ L_{\omega}^- - 1$. 

\begin{Rem}
As in \cite{DeC13}*{Appendix B}, one sees that $U_q(\mfg_{\nu})$ is a quantization of the enveloping $*$-algebra of the real Lie algebra $\mfg_{\nu} \subseteq \mfg$ spanned (over the reals) by the compact Cartan algebra $\mft = \R[ih_r] \subseteq \mfh$ and the $f_{\alpha} - \epsilon_{\alpha}e_{\tau(\alpha)}$ and $i(f_{\alpha} + \epsilon_{\alpha}e_{\tau(\alpha)})$, where the $f_{\alpha}$ run through the negative root vectors.  In the flag case, we obtain the real Lie subalgebra $\mfk_S \oplus \mfn_S^-$ of the negative parabolic subalgebra $\mfp_S^-$ of $\mfg$ at $S = \{r\in I \mid \epsilon_r = 1\}$, consisting of the nilpotent part $\mfn_S^-$, generated by the $f_{r}$ with $r\notin S$, and the compact part $\mfk_S$ of the Levi factor.  In the symmetric case this gives the real semisimple Lie algebra associated to the involution $\nu$, consisting of all elements $X\in \mfg$ with $\nu(X)^* = -X$.
\end{Rem}

\begin{Prop}\label{PropExtPairing}
There is a unique pairing $(-,-)_{\epsilon}$ of Hopf algebras between $U_q^{\epsilon}(\mfg)$ and $\mcO_q(G)$ such that
\begin{equation}\label{EqRestrHol}
(K_{\omega},f)_{\epsilon} = (K_{\omega},f),\quad (E_r,f)_{\epsilon} = \epsilon_r (E_r,f),\quad (F_r,f)_{\epsilon} = (F_r,f). 
\end{equation}
Moreover, there is a unique pairing $(-,-)_{\nu}$ of Hopf $*$-algebras between $U_q(\mfg_{\nu})$ and $\mcO_q^{\nu}(G_{\R})$ extending the above pairing $(-,-)_{\epsilon}$.
\end{Prop}
\begin{proof}
It is easily seen by a rescaling argument that there exists a unique pairing of Hopf algebras between $U_q^{\epsilon}(\mfg)$ and $\mcO_q(G)$ satisfying $\eqref{EqRestrHol}$. It follows that there can be at most one extension to a pairing of Hopf $*$-algebras between $U_q(\mfg_{\nu})$ and  $\mcO_q^{\nu}(G_{\R})$, defined by
\begin{equation}\label{EqDefLeft}
(X,fg^{\dag})_{\nu} = (X_{(1)},f)_{\epsilon}\overline{(S(X_{(2)})^{\dag},g)_{\epsilon}},\qquad f,g\in \mcO_q(G).
\end{equation}
To see that this is indeed a pairing of Hopf $*$-algebras, the only non-trivial relation to verify is that also
\begin{equation}\label{EqDefRight}
(X,f^{\dag}g)_{\nu} =\overline{(S(X_{(1)})^{\dag},f)_{\epsilon}}(X_{(2)},g)_{\epsilon}.
\end{equation}
Now the left hand side equals
\begin{align*}
(X,f^{\dag}g)_{\nu} &= \omega_{\nu}(f_{(1)}^{\dag},g_{(1)})(X,g_{(2)}f_{(2)}^{\dag})_{\nu}\omega_{\nu}^{-1}(f_{(3)}^{\dag},g_{(3)})\\
&= \mbr_{\nu}(g_{(1)},f_{(1)}^*)(X_{(1)},g_{(2)})_{\epsilon}\overline{(S(X_{(2)})^{\dag},f_{(2)})_{\epsilon}} \mbr_{\nu}(g_{(3)},S(f_{(3)})^*)
\end{align*}
It is then sufficient to verify that this equals \eqref{EqDefRight} for $X \in U_q(\mfb) \cup U_q(\mfb^-)$. Now for $X \in U_q(\mfb)$, we have
\begin{align*}
(X,f^{\dag}g)_{\nu} &= \mbr_{\nu}(g_{(1)},f_{(1)}^*)(\nu_{\epsilon}(X_{(1)}),g_{(2)})(\tau(X_{(2)}),f_{(2)}^*) \mbr_{\nu}(g_{(3)},S(f_{(3)})^*)\\
&= (\msR_{\nu}(\nu_{\epsilon}\otimes \tau)\Delta(X) \msR_{\nu}^{-1},g\otimes f^*) \\
&= ((\nu_{\epsilon}\otimes \tau)(\msR\Delta(X) \msR^{-1}),g\otimes f^*)\\
&= ((\nu_{\epsilon}\otimes \tau)\Delta^{\op}(X),g\otimes f^*)\\
&= (\tau(X_{(1)}),f^*)(\nu_{\epsilon}(X_{(2)}),g)\\
&= \overline{(S(X_{(1)})^{\dag},f)_{\epsilon}} (X_{(2)},g)_{\epsilon}.
\end{align*}
Similarly, for $X\in U_q(\mfb^-)$ we have 
\begin{align*}
(X,f^{\dag}g)_{\nu} & = \mbr_{\nu}(g_{(1)},f_{(1)}^*)(X_{(1)},g_{(2)})(\nu(X_{(2)}),f_{(2)}^*) \mbr_{\nu}(g_{(3)},S(f_{(3)})^*)\\
& = (\msR_{\nu}(\id\otimes \nu)\Delta(X) \msR_{\nu}^{-1},g\otimes f^*) \\
& = ((\id\otimes \nu)(\msR\Delta(X) \msR^{-1}),g\otimes f^*)\\
& = ((\id\otimes \nu)\Delta^{\op}(X),g\otimes f^*)\\
& = (\nu(X_{(1)}),f^*)(X_{(2)},g)\\
& = \overline{(S(X_{(1)})^{\dag},f)_{\epsilon}} (X_{(2)},g)_{\epsilon}. \qedhere
\end{align*}
\end{proof} 

The above pairing will of course not be non-degenerate. 

\begin{Def}
We define $\mcO_q(G_{\nu})$ to be the coimage Hopf $*$-algebra of $\mcO_q^{\nu}(G_{\R})$ under the $*$-homomorphism
\[
f \mapsto (f,-)_{\nu}
\]
into the dual of $U_q(\mfg_{\nu})$. We denote by $\pi_{\nu}$ the quotient map
\[
\pi_{\nu}: \mcO_q^{\nu}(G_{\R}) \rightarrow \mcO_q(G_{\nu})
\]
\end{Def}

To see which extra relations one is quotienting out by, we introduce the following definition. 

\begin{Def}\label{DefmsE}
Fixing  $\epsilon$ as above, we let $\msE = \msE_{\epsilon}$ be the unique element in $\mcU_q(\mfg) = \prod_{\varpi} \End(V_{\varpi})$ determined by 
\[
\msE \xi = \epsilon_{\varpi - \wt(\xi)}\xi,\qquad \xi \in V_{\varpi}.
\]
\end{Def}

We then let $\msE_{\pi}$ be the corresponding action of $\msE$ in a general representation, and $\msE_{\varpi}$ the action in a particular irreducible representation of highest weight $\varpi$. 

\begin{Def}\label{DefQuot}
We denote $\mcO_q(\widetilde{G}_{\nu})$ for the Hopf $*$-algebra obtained by quotienting out $\mcO_q^{\nu}(G_{\R})$ by the extra ($*$-compatible) relations
\begin{equation}\label{EqDefIdGnu}
\tau(Y_{\varpi}^{\dag})\msE_{\varpi} Y_{\varpi} = \msE_{\varpi} 
\end{equation}
for all $\varpi \in P^+$, where we use the shorthand $\tau(Y) = (\id\otimes \tau)(Y)$. 
\end{Def}
Note that we can rewrite the defining relations \eqref{EqDefIdGnu} of $\mcO_q(\widetilde{G}_{\nu})$ as
\begin{equation}\label{EqCommX}
\tau(Y_{\varpi}^{\dag})\msE_{\varpi}  = \msE_{\varpi}S(Y_{\varpi}),
\end{equation}
which makes it clear that we are dividing out by a Hopf ideal,
\[
Y_{13}^{\dag}Y_{12}'^{\dag} \msE  - \msE S(Y)_{13}S(Y')_{12} = Y_{13}^{\dag}(Y_{12}'^{\dag} \msE - \msE S(Y')_{12}) + (Y_{13}^{\dag}\msE - \msE S(Y)_{13})S(Y')_{12}.
\]

\begin{Rem}
As $\msE_{\varpi}\otimes \msE_{\varpi'}$ restricts to and equals $\Delta(\msE)$ on the irreducible module spanned by the tensor product of the highest weight vectors, it follows that we only need to impose the relations \eqref{EqCommX} on the highest weight representations for the fundamental weights. 
\end{Rem}

\begin{Lem}\label{LemFactt}
The pairing $(-,-)_{\nu}$ factors through $\mcO_q(\widetilde{G}_{\nu})$.
\end{Lem}
\begin{proof}
We have to check that \eqref{EqCommX} is satisfied when applying the pairing $(-,-)_{\nu}$ to this equation with respect to the generators $E_r,F_r$ and $K_{\omega}$. Now for $f\in \mcO_q^{\nu}(G_{\R})$, we have 
\[
(K_{\omega},f^{\dag})_{\nu} = \overline{(K_{\tau(\omega)}^{-1},f)},\quad (E_r,f^{\dag})_{\nu} = -\overline{(F_{\tau(r)},f)},\quad (F_r,f^{\dag})_{\nu} = - \epsilon_r \overline{(E_{\nu(r)},f)}.
\]
Hence \eqref{EqCommX} translates to 
\[
\pi_{\varpi}(K_{\omega})\msE_{\varpi} = \msE_{\varpi} \pi_{\varpi}(K_{\omega}),\quad \pi_{\varpi}(F_r)^*\msE_{\varpi} = \epsilon_r\msE_{\varpi}\pi_{\varpi}(K_r^{-1}E_r),\qquad \epsilon_r\pi_{\varpi}(E_r)^*\msE_{\varpi} = \msE_{\varpi}\pi_{\varpi}(F_rK_r).
\]
The first relation is obviously true, while using $\pi_{\varpi}(X)^* = \pi_{\varpi}(X^*)$ and self-adjointness of $\msE_{\varpi}$ the last two relations reduce to 
\[
\pi_{\varpi}(E_r)\msE_{\varpi} = \epsilon_r \msE_{\varpi}\pi_{\varpi}(E_r). 
\]
This however follows from the fact that if $E_rV_{\varpi}(\varpi - \alpha) \neq 0$ for some $\alpha \in Q^+$, where $V_{\varpi}(\omega)$ is the weight space at weight $\omega$, then $\alpha -\alpha_r \in Q^+$ and hence $\epsilon_{\alpha} = \epsilon_{\alpha-\alpha_r}\epsilon_r$.
\end{proof}

\begin{Rem}
It follows that we obtain a surjective Hopf $*$-algebra homomorphism
\[
\mcO_q(\widetilde{G}_{\nu}) \twoheadrightarrow \mcO_q(G_{\nu}).
\]
It is easy to see that this will be an isomorphism when $\epsilon_r\neq 0$ for all $r$, as the defining relations for $\mcO_q(\widetilde{G}_{\nu})$ can then be written 
\[
\tau(Y_{\varpi}^{\dag}) = \msE_{\varpi}S(Y_{\varpi})\msE_{\varpi}^{-1},
\]
from which it follows that $\mcO_q(\widetilde{G}_{\nu})$ will be non-degenerately paired with $U_q(\mfg_{\nu})$. We suspect that this will be true in general, but were not able to prove this. 
\end{Rem}

\begin{Rem}
As we are only interested in the classical limit for motivational reasons, the classical limit $\mcO(G_{\nu})$ for $q \rightarrow 1$ will be interpreted without further justification as the algebra of regular functions on the real affine group $G_{\nu} = \Spec_{*}(\mcO(G_{\nu}))\subseteq G$ with Lie algebra $\mfg_{\nu}$.   For example, in the flag case we have that the group $G_{\nu}$ of $*$-preserving characters of $\mcO(G_{\nu})$ equals $K_SN_S^-$, with $N_S^-$ the unipotent part of the negative parabolic subgroup $P_S^-$ associated to the simple roots $S = \{r\in I\mid \epsilon_r = 1\}$, and with $K_S = U\cap P_S^-$. In the symmetric case, we have that the space of $*$-characters of $\mcO(G_{\nu})$ is the subgroup $G_{\nu} = \{g\in G\mid \nu(g)^* = g^{-1}\}$, which has Lie algebra $\mfg_{\nu}$ but is not necessarily connected or simply connected. 
\end{Rem}

\subsection{$\nu$-braided deformation of $\mcO_q(G)$}

Consider the left, resp.~ right coactions (= $*$-preserving comodule algebra structures)
\[
\lambda_{\nu,\mu} = (\pi_{\nu}\otimes \id)\Delta_{\nu,\mu}^{\nu}: \mcO_q^{\nu,\mu}(G_{\R}) \rightarrow  \mcO_q(G_{\nu}) \otimes \mcO_q^{\nu,\mu}(G_{\R}),
\]
\[
\rho_{\nu,\mu} = (\id\otimes \pi_{\mu})\Delta_{\nu,\mu}^{\mu}: \mcO_q^{\nu,\mu}(G_{\R}) \rightarrow \mcO_q^{\nu,\mu}(G_{\R}) \otimes  \mcO_q(G_{\mu}).
\]

In the following, we will be interested  in characterizing the fixed point $*$-subalgebra of $\lambda_{\nu,\mu}$ in case $\mu = \id$, which we will denote 
\[
\mcO_q(G_{\nu}\dbbackslash G_{\R}) = \mcO_q^{\nu,\id}(G_{\R})^{\lambda_{\nu,\id}} = \{f\in \mcO_q^{\nu,\id}(G_{\R}) \mid \lambda_{\nu,\id}(f) = 1\otimes f\}.
\]
We then put $\lambda_{\nu} = \lambda_{\nu,\id}$ and $\rho_{\nu} = \rho_{\nu,\id}$. 

\begin{Rem}
We are mimicking notation from geometric invariant theory (GIT): for $q=1$ we will not necessarily have that the ordinary quotient $G_{\nu}\backslash G_{\R}$ is an affine variety, and in particular will not equal the (real) spectrum $G_{\nu}\dbbackslash G_{\R}$ of $\mcO(G_{\R})^{\lambda_{\nu,\id}}$. However, $G_{\nu}\backslash G_{\R}$ will be embedded in this spectrum as a Zariski dense open subset, see Remark \ref{RemSpectrumZ}.
\end{Rem} 

Let us write
\[
W_{\pi} = (\id\otimes \pi_{\nu})Y_{\pi} \in \End(V_{\pi}) \otimes \mcO_q(G_{\nu}).
\]
Then the left and right coactions $\lambda_{\nu}$ and $\rho_{\nu}$ of resp. $\mcO_q(G_{\nu})$ and $\mcO_q(U)$ on $\mcO_q^{\nu,\id}(G_{\R})$ are determined by 
\begin{equation}\label{EqMapCorep}
(\id\otimes \lambda_{\nu})Y_{\pi} = W_{\pi,12}Y_{\pi,13},\qquad (\id\otimes \rho_{\nu})Y_{\pi} = Y_{\pi,12}U_{\pi,13}.
\end{equation}

Recall the element $\msE$ introduced in Definition \ref{DefmsE}.

\begin{Lem}\label{LemZQuot}
The elements 
\[
Z_{\pi} = \tau(Y_{\pi}^{\dag})(\msE_{\pi}\otimes 1) Y_{\pi} \in \End(V_{\pi}) \otimes \mcO_q^{\nu,\id}(G_{\R})
\]
lie in $\End(V_{\pi})\otimes \mcO_q(G_{\nu}\dbbackslash G_{\R})$. 
\end{Lem}

\begin{proof}
By Lemma \ref{LemFactt} we can use the relation \eqref{EqDefIdGnu}, so that using \eqref{EqMapCorep} we find
\[
(\id\otimes\lambda_{\nu})(Z_{\pi}) = \tau(Y_{\pi,13}^{\dag})\tau(W_{\pi,12}^{\dag}) (\msE_{\pi}\otimes 1) W_{\pi,12}Y_{\pi,13} = \tau(Y_{\pi,13}^{\dag})(\msE_{\pi}\otimes 1)Y_{\pi,13},
\]
so $Z_{\pi}$ has entries in $\mcO_q(G_{\nu}\dbbackslash G_{\R})$.
\end{proof}

\begin{Lem}
The matrices $Z_{\pi}$ satisfy the \emph{$\tau$-modified reflection equation}: with $Z = Z_{\pi}$ and $Z' = Z_{\pi'}$, we have
\[
\msR_{21}^{\pi,\pi'}Z_{13}\msR_{\tau,12}^{\pi,\pi'}Z_{23}' = Z_{23}' \msR_{\tau,21}^{\pi,\pi'}Z_{13}\msR_{12}^{\pi,\pi'},
\]
where we write $\msR^{\pi,\pi'} = (\pi\otimes \pi')\msR$ and $\msR_{21}^{\pi,\pi'} = (\pi\otimes \pi')\msR_{21}$. 
\end{Lem}
\begin{proof}
Writing also $Y_{\pi} = Y$ etc., we compute using \eqref{EqDefHolAntihol} that
\begin{align*}
\msR_{21}^{\pi,\pi'} Z_{13} \msR_{\tau,12}^{\pi,\pi'} Z_{23} 
& = \msR_{21}^{\pi,\pi'}\tau(Y)^{\dag}_{13} \msE_{1} Y_{13} \msR_{\tau,12}^{\pi,\pi'}\tau(Y')^{\dag}_{23} \msE_{2}' Y_{23}'\\ 
& = \msR_{21}^{\pi,\pi'}\tau(Y)^{\dag}_{13} \msE_{1} \tau(Y')^{\dag}_{23}\msR_{\epsilon,12}^{\pi,\pi'} Y_{13} \msE_{2}' Y_{23}' \\ 
& = \msR_{21}^{\pi,\pi'}\tau(Y)^{\dag}_{13}\tau(Y')^{\dag}_{23} \msE_{1} \msR_{\epsilon,12}^{\pi,\pi'}\msE_{2}' Y_{13}  Y_{23}' \\ 
& = \tau(Y')^{\dag}_{23} \tau(Y)^{\dag}_{13}  \msR_{21}^{\pi,\pi'}\msE_{1} \msR_{\epsilon,12}^{\pi,\pi'}\msE_{2}' Y_{13}  Y_{23}'
\end{align*}
Now we use that
\[
\msR_{21}(\msE\otimes 1)\msR_{\epsilon}(1\otimes \msE) = \msR_{21}\msR (\msE\otimes \msE) =  (1\otimes \msE)\msR_{\epsilon,21}(\msE\otimes 1)\msR,
\]
which follows from an easy weight argument. Hence
\begin{align*}
\msR_{21}^{\pi,\pi'} Z_{13} \msR_{\tau,12}^{\pi,\pi'} Z_{23}' 
& = \tau(Y')^{\dag}_{23} \tau(Y)^{\dag}_{13} \msE_{2}'  \msR_{\epsilon,21}^{\pi,\pi'}\msE_{1} \msR_{12}^{\pi,\pi'}Y_{13}  Y_{23}'\\ 
& = \tau(Y')^{\dag}_{23} \tau(Y)^{\dag}_{13} \msE_{2}'  \msR_{\epsilon,21}^{\pi,\pi'}\msE_{1}  Y_{23}' Y_{13}\msR_{12}^{\pi,\pi'}\\ 
& = \tau(Y')^{\dag}_{23} \msE_{2}' \tau(Y)^{\dag}_{13}  \msR_{\epsilon,21}^{\pi,\pi'}Y_{23}' \msE_{1}  Y_{13}\msR_{12}^{\pi,\pi'}\\ 
& = \tau(Y')^{\dag}_{23} \msE_{2}'  Y_{23}' \msR_{\tau,21}^{\pi,\pi'}\tau(Y)^{\dag}_{13} \msE_{1}  Y_{13}\msR_{12}^{\pi,\pi'}\\
& = Z_{23}' \msR_{\tau,21}^{\pi,\pi'} Z_{13} \msR_{12}^{\pi,\pi'}. \qedhere
\end{align*}
\end{proof}

Note also that
\[
Z_{\pi}^{\dag} =\tau(Z_{\pi}).
\]

We want to view $\mcO_q(G_{\nu}\dbbackslash G_{\R})$ as a deformation of $\mcO_q(G)$. We will need some preparation. For $\xi,\eta\in V_{\pi}$, denote
\[
Z(\xi,\eta) = (\xi^*\otimes 1)Z_{\pi}(\eta \otimes 1) \in \mcO_q(G_{\nu}\dbbackslash G_{\R})
\]
for the associated matrix coefficient. Recall that we view $V^*$ as the dual of $V$ with the contragredient representation \eqref{EqContragredient}. By the Peter-Weyl-decomposition, we have a vector space isomorphism
\[
\oplus_{\varpi \in P^+} V_{\varpi}^* \otimes V_{\varpi} \rightarrow \mcO_q(G),\quad \xi^* \otimes \eta \mapsto Y(\xi,\eta) = (\xi^*\otimes 1)Y_{\varpi}(\eta\otimes 1)
\]

\begin{Prop}\label{PropLinBij}
The map 
\begin{equation}\label{EqIsoj}
j_{\nu}: \mcO_q(G) \rightarrow \mcO_q(G_{\nu}\dbbackslash G_{\R}),\quad Y(\xi,\eta) \mapsto Z(\xi,\eta)
\end{equation}
is a linear bijection.
\end{Prop}
\begin{proof}
Let us identify the space of operators $\Hom(V,W)$ with $W\otimes V^*$ in the natural way,
\[
W\otimes V^* \cong \Hom(V,W),\qquad \xi\otimes \eta^* \mapsto \xi\eta^*.
\]
Since the multiplication map $\mcO_q(\overline{G}) \otimes \mcO_q(G) \rightarrow \mcO_q^{\nu,\id}(G_{\R})$ is bijective, it follows from the Peter-Weyl-decomposition that we have a bijective linear map
\begin{equation}\label{EqIsoDTen}
\mathrm{PW}: \oplus_{\varpi,\varpi' \in P^+} V_{\varpi}^* \otimes \Hom(V_{\varpi'},V_{\varpi}) \otimes V_{\varpi'} \rightarrow \mcO_q^{\nu,\id}(G_{\R}),
\end{equation}
\[
 \xi^*\otimes (\eta\otimes \xi'^*) \otimes \eta' \mapsto \tau(Y_{\varpi}(\eta,\xi)^{\dag})Y_{\varpi'}(\xi',\eta').
\]
Since
\[
j_{\nu}(Y_{\varpi}(\xi,\eta)) = \mathrm{PW}(\xi^* \otimes \msE_{\varpi} \otimes \eta)
\]
and none of the $\msE_{\varpi}$ are zero, this proves that the map $j_{\nu}$ is injective. 

To see that the map is surjective, consider on $\mcO_q^{\nu,\id}(G_{\R})$ the infinitesimal right action of $U_q(\mfg_{\nu})$ via
\[
f\lhd X = ((X,-)_{\nu}\otimes\id)\lambda_{\nu}(f),\qquad X\in U_q(\mfg_{\nu}),f\in \mcO_q^{\nu,\id}(G_{\R}). 
\]
It is easy to see that under the isomorphism \eqref{EqIsoDTen}, the action $\lhd$ restricts to each of the components $V_{\varpi}^* \otimes \Hom(V_{\varpi'},V_{\varpi}) \otimes V_{\varpi'}$, on which it is given by 
\[
(\xi^*\otimes T \otimes \eta) \lhd X = \xi ^*\otimes (T\lhd X)\otimes \eta
\]
for an action of $U_q(\mfg_{\nu})$ on $\Hom(V_{\varpi'},V_{\varpi})$. We are to show that $T \lhd X = \varepsilon(X)T$ for all $X\in U_q(\mfg_{\nu})$ implies $\varpi = \varpi'$ and $T \in \C \msE_{\varpi}$. However, it is easily seen that 
\[
T \lhd E_r = -K_{r}^{-1}E_{r}T + \epsilon_r K_{r}^{-1}TE_r,\quad T \lhd F_r = -\epsilon_r F_{r}K_{r}T K_r^{-1}+ TF_r,\qquad T \lhd K_{\chi} = K_{\chi}^{-1}TK_{\chi}.
\]
Hence if $T \lhd X = \varepsilon(X)T$ for all $X\in U_q(\mfg_{\nu})$, it follows that 
\[
E_r T = \epsilon_r TE_{r},\qquad TF_r = \epsilon_r F_{r}T,\qquad K_{\chi}T = TK_{\chi}.
\]
The first identity implies that $T$ preserves the vector space spanned by highest weight vectors, while the second identity implies that the action of $T$ on the highest weight vector completely determines the action. Combined with the third identity, it then follows  that $T=0$ unless $\varpi' = \varpi$, in which case the space of $T$'s is one-dimensional, consisting of multiples of $\msE_{\varpi}$. 
\end{proof}

We want to determine explicitly the resulting $*$-algebra structure that $\mcO_q(G)$ inherits through the map $j_{\nu}$. 

\begin{Def}
Let $\epsilon: I \rightarrow \R$, and extend $\epsilon$ as before to a semigroup homomorphism $(Q^+,+) \rightarrow (\R,\cdot)$. We define 
\[
\Omega_{\epsilon} \in \mcU_q(\mfg) \hat{\otimes}\mcU_q(\mfg)
\]
as the unique element such that 
\[
\Omega_{\epsilon}\iota =  \epsilon_{\varpi + \varpi' - \varpi''}\iota,\qquad \forall \iota \in \Hom_{U_q(\mfu)}(V_{\varpi''},V_{\varpi}\otimes V_{\varpi'}).
\]
\end{Def}
It is easily seen that $\Omega_{\epsilon}$ is well-defined, as non-zero $\iota$ as above exist only when $\varpi + \varpi' - \varpi'' \in Q^+$. Moreover, when $\epsilon$ has no zero values, it is easily seen that 
\[
\Omega_{\epsilon} = (\msE\otimes \msE)\Delta(\msE^{-1}).
\]
It follows by continuity that, in general, $\Omega_{\epsilon}$ satisfies the $2$-cocycle identity
\[
(\Omega_\epsilon\otimes 1)(\Delta\otimes \id)(\Omega_{\epsilon}) = (1\otimes \Omega_{\epsilon})(\id\otimes \Delta)(\Omega_{\epsilon}).
\]
Moreover, as $\Omega_{\epsilon}$ assumes constant values on isotypical components in the tensor product, we have
\begin{equation}\label{EqCommDelt}
\Omega_{\epsilon}\Delta(X) = \Delta(X) \Omega_{\epsilon},\qquad X\in U_q(\mfu). 
\end{equation}

Finally, we note the following behaviour with respect to the universal $R$-matrix. 

\begin{Lem}
The following identity holds:
\begin{equation}\label{EqCommROmega}
\msR \Omega_{\epsilon} = \Omega_{\epsilon,21}\msR. 
\end{equation}
\end{Lem}
\begin{proof}
This follows immediately from the fact that, with $\Sigma$ the flip map, the elements $\Sigma (\pi\otimes \pi')\msR$ preserve spectral subspaces, while the $(\pi\otimes \pi')\Omega_{\epsilon}$ are constant on spectral subspaces.
\end{proof}

The following modifies the construction of braided Hopf algebras  as in \cite{Maj93a}*{Theorem 4.1}, see also \cite{KS97}*{Proposition 10.3.30} and the references in \cite{Maj95}.

\begin{Def}
We define $\mcO_q(Z_{\nu}) = \mcO_q^{\nu-\braid}(G)$ to be the vector space $\mcO_q(G)$ with the product 
\begin{equation}\label{EqDefProdBraid}
f * g = (g_{(2)} \otimes f_{(1)},\Omega_{\epsilon})\mbr(f_{(2)},g_{(3)})f_{(3)}g_{(4)}\mbr(f_{(4)},\tau(S(g_{(1)}))) 
\end{equation}
and the $*$-structure 
\begin{equation}\label{EqDefStarBraid}
f^{\sharp} = \tau(S(f)^*),
\end{equation}
where on the right hand sides one uses the Hopf $*$-algebra structure of $\mcO_q(U)$. 
\end{Def}

We will show in a moment, by an indirect argument, that this defines a unital $*$-algebra structure, but for now we just view the above as a binary and unary operation. 

\begin{Def}
We will call $\mcO_q(Z_{\nu})$ the \emph{$\nu$-modified braided Hopf algebra}. 
\end{Def}

Note that the unit $1\in \mcO_q(G)$ is still the unit of $\mcO_q(Z_{\nu})$. We record the following `inverse formula' to \eqref{EqDefProdBraid}.

\begin{Lem} 
For all $f,g\in \mcO_q(Z_{\nu})$ we have 
\begin{equation}\label{EqAlternateK} 
(f_{(1)} \otimes g_{(1)},\Omega_{\epsilon})f_{(2)}g_{(2)} =  \mbr(S(f_{(1)}),g_{(1)})f_{(2)} * g_{(3)}\mbr(f_{(3)},\tau(g_{(2)})).\end{equation} 
\end{Lem}
\begin{proof}
From \eqref{EqCommROmega}, we see that
\[
f * g =\mbr(f_{(1)},g_{(2)}) (f_{(2)} \otimes g_{(3)},\Omega_{\epsilon})f_{(3)}g_{(4)}\mbr(f_{(4)},\tau(S(g_{(1)}))).
\]
The result then follows from the fact that $(S\otimes \id)\msR = \msR^{-1}$, while $(\id\otimes S)\msR$ is the inverse of $\msR$ with respect to $\mcU_q(\mfg)\hat{\otimes} \mcU_q(\mfg)^{\opp}$, where $\mcU_q(\mfg)^{\opp}$ has the opposite product. 
\end{proof}
\begin{Rem}
In case $\msE$ is invertible, we can further reduce \eqref{EqAlternateK} by inverting $\Omega_{\epsilon}$ and using \eqref{EqCommROmega},
\begin{equation}\label{EqAlternateKInv}
fg = \mbr(S(f_{(1)}),g_{(1)})(g_{(2)}\otimes f_{(2)},\Omega_{\epsilon}^{-1})f_{(3)}*g_{(4)} \mbr(f_{(4)},\tau(g_{(3)})).
\end{equation}
\end{Rem}
\begin{Theorem}\label{TheoIsoZCoset}
The map $j_{\nu}$ introduced in \eqref{EqIsoj} induces an isomorphism of $*$-algebras 
\[
j_{\nu}: \mcO_q(Z_{\nu}) \cong \mcO_q(G_{\nu}\dbbackslash G_{\R}).
\]
\end{Theorem}
Remark that, borrowing again the coproduct from $\mcO_q(U)$, the map $j_{\nu}$ can be written more intrinsically as 
\begin{equation}\label{EqDefjnu}
j_{\nu}: \mcO_q(Z_{\nu}) \rightarrow \mcO_q(G_{\nu}\dbbackslash G_{\R}),\quad f \mapsto f_{(2)}(\msE) \tau(S(f_{(1)}))^{*\dag} f_{(3)}. 
\end{equation}
\begin{proof}
By Proposition \ref{PropLinBij}, it only remains to show that $j_{\nu}$ is a $*$-algebra map. Now the preservation of $*$-structures follows immediately from the definitions and \eqref{EqIsoj}.  On the other hand, to show that $j_{\nu}$ is multiplicative we may restrict to the case where  none of the $\epsilon_r$ are zero, as the structure coefficients of our algebras depend continuously on the $\epsilon_r$. Let us fix elements $f, g \in \mcO_q(Z_{\nu})$. By the bijectivity of the map $j_{\nu}$, there exists $h\in \mcO_q(Z_{\nu})$ such that
\[
j_{\nu}(f)j_{\nu}(g) = j_{\nu}(h).
\] 
We are to show that $h = f * g$. We have
\begin{align*}
j_{\nu}(f)j_{\nu}(g) &= f_{(2)}(\msE)g_{(2)}(\msE) \tau(S(f_{(1)}))^{*\dag} f_{(3)}   \tau(S(g_{(1)}))^{*\dag} g_{(3)} \\ 
&= f_{(2)}(\msE)g_{(4)}(\msE)  \omega_{\nu}^{-1}(\tau(S(g_{(3)}))^{*\dag},f_{(3)})
\tau(S(f_{(1)}))^{*\dag} \tau(S(g_{(2)}))^{*\dag} f_{(4)}g_{(5)}\omega_{\id}(\tau(S(g_{(1)}))^{*\dag},f_{(5)})\\
&= f_{(2)}(\msE)g_{(4)}(\msE) \mbr_{\nu}(f_{(3)},\tau(g_{(3)})) \tau(S(f_{(1)}))^{*\dag} \tau(S(g_{(2)}))^{*\dag} f_{(4)} g_{(5)}\mbr(f_{(5)},\tau(S(g_{(1)})))\\
&= h_{(2)}(\msE)\tau(S(h_{(1)}))^{*\dag} h_{(3)}.
\end{align*} 
Applying the counit to the $*\dag$-parts and bringing $\msE$ to the other side, this reads
\[
h = f_{(1)}(\msE)g_{(3)}(\msE)  \mbr_{\nu}(f_{(2)},\tau(g_{(2)}))(f_{(3)}g_{(4)}(\msE^{-1}))f_{(4)} g_{(5)}\mbr(f_{(5)},\tau(S(g_{(1)}))).
\]

But since 
\[
\mbr_{\nu}(f,g) = g_{(1)}(\msE) \mbr(f,\tau(g_{(2)}))g_{(3)}(\msE^{-1}),
\]
we then have for $X\in U_q(\mfg)$ that 
\begin{align*}
h(X) &= f_{(1)}(\msE)g_{(2)}(\msE) \mbr(f_{(2)},g_{(3)}) (f_{(3)} g_{(4)}(\msE^{-1}))(f_{(4)}g_{(5)}(X))\mbr(f_{(5)},\tau(S(g_{(1)})))\\  
&=  (f \otimes g, \msE  \msR_1\msE_{(1)}^{-1} X_{(1)} \msR_{1'} \otimes \tau(S(\msR_{2'}))\msE \msR_2\msE_{(2)}^{-1} X_{(2)}). \\
&= (f \otimes g, \msE  \msE_{(2)}^{-1}\msR_1 X_{(1)} \msR_{1'} \otimes \tau(S(\msR_{2'}))\msE \msE_{(1)}^{-1}\msR_2 X_{(2)}).
\end{align*}
It follows that $h = f*g$. 
\end{proof}

Now as the left coaction $\lambda_{\nu}$ commutes with the comultiplication $\Delta_{\nu,\id}^{\id}$, viewed as a right coaction on $\mcO_q^{\nu,\id}(G_{\R})$, it follows that the latter descends to a right coaction of $\mcO_q(G_{\R}) = \mcO_q^{\id}(G_{\R})$ on $\mcO_q(G_{\nu}\dbbackslash G_{\R})$. From the bijectivity of $j_{\nu}$, we get that $\mcO_q(Z_{\nu})$ inherits this coaction, given concretely via 
\begin{equation}\label{Eqdeltanu}
\delta_{\nu}: \mcO_q(Z_{\nu}) \rightarrow \mcO_q(Z_{\nu}) \otimes \mcO_q(G_{\R}): Z_{\pi} \mapsto \tau(Y_{\pi})_{13}^{\dag}Z_{\pi,12} Y_{\pi,13},
\end{equation}
where we transport the matrices $Z_{\pi}$ to $\mcO_q(Z_{\nu})$.  This in particular descends to a coaction of $\mcO_q(U)$ as
\begin{equation}\label{Eqrhonu}
\rho_{\nu}: Z_{\pi} \mapsto \tau(U_{\pi})_{13}^{*}Z_{\pi,12} U_{\pi,13},
\end{equation}
which we can also write as a twisted coadjoint coaction
\begin{equation}\label{EqCoactrho}
\rho_{\nu}(f) = f_{(2)}\otimes S(\tau(f_{(1)}))f_{(3)},\qquad f\in \mcO_q(Z_{\nu}).
\end{equation}
On the other hand, we can also descend to a coaction of $\mcO_q(B_{\R})$ by 
\[
\beta_{\nu}: Z_{\pi} \mapsto \tau(T_{\pi}^-)_{13}^{-1}Z_{\pi,12} T_{\pi,13}^+,
\]
which through the natural Hopf $*$-algebra surjection $\iota:\mcO_q(B_{\R}) \rightarrow U_q(\mfu)^{\cop}$  introduced in \eqref{EqDefiota} descends to a left coaction 
\begin{equation}\label{Eqgammanu}
\gamma_{\nu}: \mcO_q(Z_{\nu}) \rightarrow U_q(\mfu) \otimes \mcO_q(Z_{\nu}). 
\end{equation}
Since $(\id\otimes \iota)T_{\pi}^+ = (\pi\otimes \id)\msR$ and $(\id\otimes \iota)T_{\pi}^- = (\pi\otimes \id)\msR_{21}^{-1}$, we can write 
\begin{equation}\label{EqCoactgamm}
(\id\otimes \gamma_{\nu})Z_{\pi} = (\pi\otimes \id)(\msR_{\tau,21}) Z_{\pi,13}(\pi\otimes \id)(\msR)_{12}.
\end{equation}

\begin{Rem}\label{RemSpectrumZ}
The notation $Z_{\nu}$ refers to an isomorphic copy of the real spectrum of $\mcO(G_{\nu}\dbbackslash G_{\R})$. Classically, it corresponds to the real variety of $Y \in \prod_{\varpi}\End(V_{\varpi})$ satisfying the identities
\begin{equation}\label{EqClassY}
Y\otimes Y = \Omega_{\epsilon}\Delta(Y),\qquad \tau(Y)^* = Y.
\end{equation}

We have
\[
G_{\nu}\backslash G_{\R} \rightarrow Z_{\nu},\quad G_{\nu}g \mapsto \tau(g)^*\msE g.
\]

When $\msE$ is invertible, we can rewrite \eqref{EqClassY} as 
\[
\msE^{-1}Y\otimes \msE^{-1}Y = \Delta(\msE^{-1}Y),
\] 
so that we can naturally identify 
\[
Z_{\nu} \cong H_{\nu} := \{g\in G \mid \tau(g)^* = \msE g\msE^{-1}\} \subseteq G \subseteq \prod_{\varpi} \End(V_{\varpi}),\quad Y \mapsto \msE^{-1}Y.
\]
If moreover $\nu$ is of symmetric type, so that we can view $\nu$ as an involutive automorphism of $G$, we obtain that 
\begin{equation}\label{EqHnu}
H_{\nu} = \{g\in G\mid \nu(g)^* = g\},
\end{equation}
and the map $ G_{\nu}\backslash G_{\R} \rightarrow Z_{\nu}\cong H_{\nu}$ is in this case given by $G_{\nu}g \mapsto  \nu(g)^*g$, which is in general not surjective but has Zariski dense image. 
\end{Rem}

\subsection{Embedding of $\mcO_q(Z_{\nu})$ inside $\mcO_q^{\nu,\id}(B_{\R})$}

We want to relate now $\mcO_q(Z_{\nu})$ with $\mcO_q^{\nu,\id}(B_{\R})$. Recall the $*$-homomorphism $\msP$ introduced in \eqref{EqMapP}.

\begin{Prop}\label{PropInu}
The natural $*$-homomorphism 
\[
I_{\nu} = \msP\circ j_{\nu}: \mcO_q(Z_{\nu}) \rightarrow \mcO_q^{\nu,\id}(B_{\R})
\]
is injective.
\end{Prop}
\begin{proof}
Using the expression \eqref{EqDefjnu} for $j_{\nu}$, we have
\begin{equation}\label{EqAltInu}
I_{\nu}(f) = f_{(2)}(\msE) \msP_-(\tau(S(f_{(1)}))) \msP_+(f_{(3)}). 
\end{equation}
Assume now that $\msP(j_{\nu}(f)) = 0$. Using the natural vector space pairing of $\mcO_q^{\nu,\id}(B_{\R}) \cong \mcO_q(B^-) \otimes \mcO_q(B^+)$ with $U_q(\mfb^-)\otimes U_q(\mfb^+)$, together with the stability of $U_q(\mfb^-)$ under $\tau \circ S$, we see that 
\[
f(X\msE Y) = 0,\qquad \forall X\in U_q(\mfb^-),Y\in U_q(\mfb^+).
\]
Now since 
\begin{equation}\label{EqCommE}
Y \msE = \msE \nu_{\epsilon}(Y),\qquad \msE X = \nu_{\epsilon}(X)\msE,\qquad Y \in U_q(\mfb^+), X\in U_q(\mfb^-),
\end{equation}
it follows from $U_q(\mfg) = U_q(\mfb^+)U_q(\mfb^-) = U_q(\mfb^-)U_q(\mfb^+)$ that in fact
\[
f(X\msE Y) = 0,\qquad \forall X,Y\in U_q(\mfg).
\]
Since $\msE$ is non-zero in each irreducible representation, it follows by an easy argument, using the Peter-Weyl decomposition and the fact that the center of $U_q(\mfg)$ separates representations of $U_q(\mfg)$, that $f = 0$. 
\end{proof}

The map $I_{\nu}$ has an important equivariance property. Note first that the right $\mcO_q(U)$-coaction $\rho_{\nu}$ endows $\mcO_q(Z_{\nu})$ with an infinitesimal left $U_q(\mfu)$-module $*$-algebra structure
\begin{equation}\label{EqInvAct}
X \rhd f = (\id\otimes (X,-))\rho_{\nu}(f),\qquad X \rhd Z(\xi,\eta) = Z(S(\tau(X_{(1)}))^*\xi,X_{(2)}\eta),
\end{equation}
where compatibility with the $*$-structure means that
\[
(X\rhd f)^{\#} = S(X)^* \rhd f^{\#}.
\]
On the other hand, as $\mcO_q^{\nu,\id}(B_{\R})$ forms part of a connected cogroupoid, it is in particular a right Galois object for $\mcO_q(B_{\R}) = \mcO_q^{\id}(B_{\R})$. We thus have on $\mcO_q^{\nu,\id}(B_{\R})$ the adjoint (or Miyashita-Ulbrich) action of $\mcO_q(B_{\R})$ \cite{Sch04}*{Definition 2.1.8}, which is a right $\mcO_q(B_{\R})$-module $*$-algebra structure determined explicitly in our case by 
\[
X \lhdb Y = S(Y_{(1)})X Y_{(2)},\qquad X\in \mcO_q^{\nu,\id}(B_{\R}), Y \in \mcO_q(B) \cup \mcO_q(B^-),
\]
using the usual Hopf algebra structure of $\mcO_q(B^{\pm})$. Recall now again the Hopf $*$-algebra homomorphism $\iota =\iota_{\id}: \mcO_q(B_{\R}) \rightarrow U_q(\mfu)^{\cop}$   introduced in \eqref{EqDefiota}.

\begin{Prop}\label{PropEquiInu}
The following equivariance property holds: for all $f \in \mcO_q(Z_{\nu})$ and $g \in \mcO_q(B_{\R})$ we have
\[
I_{\nu}(\iota(S(g)) \rhd f) = I_{\nu}(f)\lhdb g. 
\]
\end{Prop}
\begin{proof}
It is enough to verify this for $g \in \mcO_q(B)$, as both sides are module $*$-algebras and $\iota S = S^{-1}\iota$. Fix now $\pi,\pi'$, and note that 
\[
(\id\otimes \pi \iota)T_{\pi'}^+ = (\pi'\otimes \pi)\msR = \msR_{\pi',\pi}. 
\]
As $\iota$ flips the coproduct, it follows that
\begin{equation}\label{EqEasyExpr}
(\id\otimes \iota S)(T_{\pi'}^+)_{13} \rhd Z_{\pi,23} = (\msR_{\pi',\pi\circ \tau}\otimes 1)(1\otimes Z_{\pi}) (\msR_{\pi',\pi}^{-1}\otimes 1).
\end{equation}
On the other hand, as $(\id \otimes I_{\nu})Z_{\pi} = \tau(T_{\pi}^{+,*})(\msE_{\pi}\otimes 1)T^{+}_{\pi}$, we have by the fundamental interchange relation \eqref{LemFundInt} and \eqref{EqStarT} that
\begin{align*}
((\id\otimes I_{\nu})Z_{\pi})_{23} \lhdb T_{\pi',13}^+ 
&=   (T_{\pi'}^+)^{-1}_{13}\tau(T_{\pi}^{-})_{23}^{-1} \msE_{\pi,2}T^{+}_{\pi,23}T_{\pi',13}^+\\
&=  (T_{\pi'}^+)^{-1}_{13}\tau(T_{\pi}^{-})_{23}^{-1} \msE_{\pi,2}\msR_{\pi',\pi,12}T^{+}_{\pi',13}T_{\pi,23}^+ \msR_{\pi',\pi,12}^{-1}\\
&= (T_{\pi'}^+)^{-1}_{13}(T_{\pi\circ \tau}^{-})_{23}^{-1}\msR_{\nu,\pi',\pi\circ \tau,12}T^{+}_{\pi',13} \msE_{\pi,2}T_{\pi,23}^+ \msR_{\pi',\pi,12}^{-1}\\
&=\msR_{\pi',\pi\circ \tau,12} (T_{\pi\circ \tau}^{-})_{23}^{-1} \msE_{\pi,2}T_{\pi,23}^+ \msR_{\pi',\pi,12}^{-1}\\
&= \msR_{\pi',\pi\circ \tau,12}((\id\otimes I_{\nu}) Z_{\pi})_{23} \msR_{\pi',\pi,12}^{-1}. \qedhere
\end{align*}
Comparing this with \eqref{EqEasyExpr} finishes the proof.
\end{proof}

\begin{Rem}
In the case $\nu = \id$ this result is well-known, see e.g. \cite{Bau00}*{Theorem 3} and the references loc. cit.
\end{Rem}

We want to characterize the image of $I_{\nu}$. Note first that from the proof of Proposition \ref{PropInu}, we see that for $\xi_{\varpi}$ a unit highest weight vector and $X\in U_q(\mfb^-),Y \in U_q(\mfb^+)$
\begin{align*}
(I_{\nu}(Z_{\varpi}(\xi_{\varpi},\xi_{\varpi})), X \otimes Y) 
&= \langle \xi_{\varpi},\tau(S(X))\msE_{\varpi} Y \xi_{\varpi}\rangle \\
&= (L_{-\varpi}^-,Y)(L_{-\tau(\varpi)}^+,X) \langle \xi_{\varpi},\msE_{\varpi}\xi_{\varpi}\rangle \\
&= (L_{-\tau(\varpi)}^+L_{-\varpi}^-, X \otimes Y).
\end{align*}
As the above pairing of $\mcO_q(B_{\R})$ with $U_q(\mfb^-)\otimes U_q(\mfb^+)$ is non-degenerate, we deduce that
\begin{equation}\label{EqImZHigh}
I_{\nu}(Z(\xi_{\varpi},\xi_{\varpi})) = L_{-\tau(\varpi)}^+L_{-\varpi}^-.
\end{equation}
Let us denote in the following
\begin{equation}\label{EqDefa}
a_{\varpi} = Z(\xi_{\varpi},\xi_{\varpi}).
\end{equation}
\begin{Lem}
We have the following relations
\[
a_{\varpi}^{\sharp} = a_{\tau(\varpi)},\qquad a_{\varpi}Z(\xi,\eta) = q^{((\id+\tau)\varpi,\wt(\xi) - \wt(\eta))}Z(\xi,\eta) a_{\varpi}.
\]
\end{Lem}
\begin{proof}
The identity for $a_{\varpi}^{\sharp}$ follows immediately upon applying the $*$-homomorphism $I_{\nu}$ and using \eqref{EqImZHigh}. For the second identity, we have by Proposition \ref{PropEquiInu} and \eqref{EqImZHigh} that
\[
a_{\varpi}Z(\xi,\eta)  = (K_{-(\id+\tau)\varpi}\rhd Z(\xi,\eta))a_{\varpi} =  q^{(\varpi,(\id+\tau)\wt(\xi) - (\id+\tau)\wt(\eta))}Z(\xi,\eta) a_\varpi. \qedhere
\]
\end{proof}

It follows from the above that we can consider the $*$-algebra 
\[
\mcO_q^{\loc}(Z_{\nu}) = \mcO_q(Z_{\nu})[a_{\varpi}^{-1}]
\] 
obtained by localising the $a_{\varpi}$. We can extend $\varpi \mapsto a_{\varpi}$ to the whole of $P$ by requiring the relations
\begin{equation}\label{EqComma}
a_{\omega + \chi}  =  q^{((\id-\tau)\omega,\chi)}a_{\omega}a_{\chi},\qquad \omega,\chi\in P.
\end{equation}
The map $I_{\nu}$ then extends to $\mcO^{\loc}_q(Z_{\nu})$.

\begin{Prop}
The elements $X_r^+$ lie in the image of $\mcO^{\loc}_q(Z_{\nu})$ under $I_{\nu}$. Moreover, as a $*$-algebra $\mcO^{\loc}_q(Z_{\nu})$ is generated by the $a_{\omega}$ and the 
\[
x_r  = I_{\nu}^{-1}(X_r^+).
\] 
\end{Prop} 
\begin{proof}
Since $I_{\nu}(\mcO_q(Z_{\nu}))$ is closed under the right action of $\mcO_q(B_{\R})$, it follows by \eqref{EqImZHigh} that the range of $I_{\nu}$ contains
\[
L_{-\tau(\varpi)}^+L_{-\varpi}^- \lhdb X_r^+ = (1-q^{((\id + \tau)\varpi,\alpha_r)})L_{-\tau(\varpi)}^+L_{-\varpi}^-  X_r^+,
\]
and so $X_r^+ \in I_{\nu}(\mcO_q^{\loc}(Z_{\nu}))$. Now by \eqref{EqInvAct} and the definition \eqref{EqDefa} of the $a_{\varpi}$, we have that $\mcO_q(Z_{\nu})$ is generated as a left $U_q(\mfu)$-module by the $a_{\varpi}$. It follows that $I_{\nu}(\mcO_q(Z_{\nu}))$ is the smallest subspace containing the $a_{\varpi}$ and stable under the $\lhdb X_r^+$ and $\lhdb (X_r^+)^*$. As the latter operations can be implemented by left and right multiplication with the $X_r^+$ and $(X_r^+)^*$, it follows that $I_{\nu}(\mcO_q^{\loc}(Z_{\nu}))$ is contained in the $*$-algebra generated by the $a_{\omega}$ and $x_r$, and must hence coincide with it. 
\end{proof}

It follows that we can present $\mcO_q^{\loc}(Z_{\nu})$ directly by generators and relations: using \eqref{EqCommRelBR}, we see that it is generated by elements $a_{\omega},x_r,y_r = x_r^\sharp$, with $a_{\varpi}^{\sharp} = a_{\tau(\varpi)}$, such that the $x_r$ and $y_r$ satisfy the quantum Serre relations for the Dynkin diagram under consideration, and such that \eqref{EqComma} holds together with $a_0= 1$ and
\[
a_{\omega} x_r = q^{-((\id + \tau)\omega,\alpha_r)}x_ra_{\omega} ,\qquad a_{\omega}y_r = q^{((\id + \tau)\omega,\alpha_r)}y_ra_{\omega} 
\]
\[
x_r y_s  - q^{-(\alpha_r,\alpha_s)}     y_sx_r = \frac{\delta_{r,\tau(s)} \epsilon_s a_{-\alpha_s} - \delta_{r,s}}{q_r-q_r^{-1}}. 
\]

\begin{Rem}
For $\nu =\id$, one can characterize the image of $I_{\nu}(\mcO_q(Z_{\nu}))$ into $U_q(\mfg)$ by means of $\iota$ as the locally finite part of $U_q(\mfg)$ with respect to the adjoint action \cite{JL92}. The precise connection with the locally finite part of $\mcO_q^{\nu,\id}(B_{\R})$, or a quotient $*$-algebra thereof, becomes more muddy in the general case, particularly when $\tau \neq \id$, but will not be needed in what follows. 
\end{Rem}

\begin{Rem}\label{RemSpectral}
The embedding $I_{\nu}$ puts some `spectral conditions' on the $*$-algebra $\mcO_q(Z_{\nu}) \cong \mcO_q(G_{\nu}\dbbackslash G_{\R})$. For example if $\varpi = \tau(\varpi)$, then $I_{\nu}(a_{\varpi}) = (L_{-\varpi}^+)^*L_{-\varpi}^+$ is a positive element. This might allow one to define $\mcO_q(G_{\nu}\backslash G_{\R})$ as the finer structure of $\mcO_q(G_{\nu}\dbbackslash G_{\R})$ together with such spectral conditions, putting a restriction on its $*$-representation theory. We will however not dive deeper into these matters here. 
\end{Rem} 

\subsection{Characters of $\mcO_q(Z_{\nu})$}

\begin{Lem}\label{LemCharK}
The unital $*$-characters of $\mcO_q(Z_{\nu})$ are in one-to-one correspondence with elements 
\[
\msK \in \mcU_q(\mfg)= \prod_{\varpi} \End(V_{\varpi})
\] 
such that $\varepsilon(\msK) = 1$,
\begin{equation}\label{EqDefModKStar}
\msK^* = \tau(\msK)
\end{equation} 
and  
\begin{equation}\label{EqDefModK}
\Omega_{\epsilon}\Delta(\msK) =\msR^{-1}(\msK\otimes 1)\msR_{\tau}(1\otimes \msK). 
\end{equation}
\end{Lem}
\begin{proof}
This is a direct consequence of \eqref{EqAlternateK}.
\end{proof}
Note that in case $\msE$ is invertible, \eqref{EqDefModK} can be rewritten as 
\begin{equation}\label{EqDeltaKEInv} 
\Delta (\msE^{-1}\msK) = \msR^{-1} (\msE^{-1}\msK\otimes 1)\msR_{\nu}(1\otimes \msE^{-1} \msK).
\end{equation}
Another way of writing this is 
\begin{equation}\label{EqDeltaKEInvAlt}
\Delta(\msE^{-1}\msK) = (1\otimes \msE^{-1}\msK)\msR_{\nu,21}(\msE^{-1}\msK\otimes 1)\msR_{21}^{-1},
\end{equation}
but it is not clear what the corresponding limit would be in the case of $\msE$ not invertible. However, if \eqref{EqDefModKStar} holds, we get using \eqref{EqCommDelt} that, upon applying $*$ to \eqref{EqDefModK},
\begin{equation}\label{EqDefModKAlt}
\Omega_{\epsilon}\Delta(\msK) =  (1\otimes \msK)\msR_{\tau,21}(\msK\otimes 1)\msR_{21}^{-1}, 
\end{equation}
so that in particular we have the $\tau$-modified reflection equation
\begin{equation}\label{EqModRefl}
(1\otimes \msK)\msR_{\tau,21}(\msK\otimes 1)\msR_{21}^{-1} =  \msR^{-1}(\msK\otimes 1)\msR_{\tau}(1\otimes \msK).
\end{equation}
Finally, note that the counitality assumption $\varepsilon(\msK) = 1$ is automatic once $\msK \neq 0$.

\begin{Def}
A non-zero element $\msK \in \mcU_q(\mfg)$ satisfying \eqref{EqDefModK} will be called a \emph{$\nu$-modified universal $K$-matrix}. If also \eqref{EqDefModKStar} holds, we call $\msK$ \emph{$*$-compatible}.
\end{Def}

\begin{Theorem}\label{TheoOneToOneCorr}
There is a one-to-one correspondence between 
\begin{enumerate}
\item $*$-compatible $\nu$-modified universal $K$-matrices $\msK \in \mcU_q(\mfu)$,
\item unital $*$-characters $\chi:\mcO_q(Z_{\nu}) \rightarrow \C$,
\item unital $*$-homomorphisms $\phi: \mcO_q(Z_{\nu}) \rightarrow \mcO_q(U)$ intertwining $\rho_{\nu}$ with $\Delta$,
\item unital $*$-homomorphisms $\hat{\phi}: \mcO_q(Z_{\nu}) \rightarrow U_q(\mfu)$ intertwining $\gamma_{\nu}$ with $\Delta$,
\item unital $*$-homomorphisms $\Phi: \mcO_q(Z_{\nu})  \rightarrow \mcO_q(G_{\R})$ intertwining $\delta_{\nu}$ with $\Delta$.
\end{enumerate}
The correspondence is determined by 
\[
\chi_{\msK}(f) = f(\msK),\qquad \phi_{\chi}(f) = (\chi\otimes \id)\rho_{\nu}(f),\qquad \hat{\phi}_\chi(f)= (\id\otimes \chi)\gamma_{\nu}(f),\qquad \Phi_{\chi}(f)= (\chi\otimes \id)\delta_{\nu}(f).
\]
\end{Theorem}
\begin{proof}
The equivalence between the first two items is the content of Lemma \ref{LemCharK}.

If $(L,\Delta)$ is any Hopf $*$-algebra, and $(A,\alpha)$ a right $L$-comodule $*$-algebra, it is well-known (see e.g.~ \cite{DM03b} for a discussion) that there is a one-to-one correspondence between $*$-characters on $A$ and $*$-algebra maps $\pi:A\rightarrow L$ intertwining $\alpha$ and $\Delta$, given by the correspondence
\[
f \mapsto \pi_f = (f\otimes \id)\alpha,\qquad \pi \mapsto f_{\pi} = \varepsilon \circ \pi,
\]
where $\varepsilon$ is the counit of $L$. A similar correspondence holds for left coactions. This gives the correspondence between the last four items.
\end{proof}

Note that by \eqref{Eqdeltanu} and \eqref{Eqrhonu}, we have  
\begin{equation}\label{EqImZPhiMatrix}
(\id\otimes \Phi)Z_{\varpi} = \tau(Y_{\pi})^{\dag}_{12}(\pi(\msK)\otimes 1)Y_{\pi},
\end{equation}
\begin{equation}\label{EqImZphiMatrix}
(\id\otimes \phi)Z_{\varpi} = \tau(U_{\pi})^{*}_{12}(\pi(\msK)\otimes 1)U_{\pi}. 
\end{equation}

In the following, we fix a unital $*$-character 
\[
\chi: \mcO_q(Z_{\nu}) \rightarrow \C,
\]
and we let $\msK$ be the associated $*$-compatible $\nu$-modified universal $K$-matrix. Then $\phi$, $\hat{\phi}$ and $\Phi$ are the associated equivariant maps into respectively $\mcO_q(U)$, $U_q(\mfu)$ and $\mcO_q(G_{\R})$. We write the images of $\phi$, $\hat{\phi}$ and $\Phi$ respectively as
\[
\begin{split}
\mcO_q(K \backslash U) &= \phi(\mcO_q(Z_{\nu})) \subseteq \mcO_q(U), \\
U_q^{\fin}(\mfk') &= \hat{\phi}(\mcO_q(Z_{\nu})) \subseteq U_q(\mfu), \\
\mcO_q(L\dbbackslash G_{\R}) &= \Phi(\mcO_q(Z_{\nu})) \subseteq \mcO_q(G_{\R}).
\end{split}
\]
Then $\mcO_q(K\backslash U)$ and $\mcO_q(L\dbbackslash G_{\R})$ are right coideal $*$-subalgebras in their respective Hopf $*$-algebras, while $U_q^{\fin}(\mfk')$ is a left coideal $*$-subalgebra in $U_q(\mfu)$, which (slightly deviating from \cite{KoSt09}) we call a \emph{Noumi-Sugitani coideal subalgebra}. We may view $\mcO_q(L\dbbackslash G_{\R})$ as the \emph{Drinfeld codouble} of $\mcO_q(K\backslash U)$ and $U_q^{\fin}(\mfk')$. 

\begin{Rem}
Interpreting $\msK$ classically as an element in the spectrum of $Z_{\nu}$, the groups $L$ and $K$ are its stabilizers under the respective actions of $G$ and $U$ on $Z_{\nu}$. The symbol `\,$\fin$' in $U_q^{\fin}(\mfk')$ should be seen as indicating that it corresponds to some `locally finite part' of a quantized enveloping coideal subalgebra. Finally, as we will justify in Proposition \ref{PropInclusion}, $\mfk'$ should be seen as a Lie subalgebra of the Lie algebra $\mfk$ of $K$, and will coincide with it in many cases of interest. In the setting of Poisson homogeneous spaces, $L$ corresponds to the Lagrangian in the Drinfeld double $G$ of $U$ associated to the Poisson homogeneous space $K\backslash U$ \cite{Dri93}.
\end{Rem}

\begin{Prop}
The $*$-homomorphism $\Phi: \mcO_q(Z_{\nu}) \rightarrow \mcO_q(L\dbbackslash G_{\R})$ is a $*$-isomorphism.
\end{Prop} 
\begin{proof}
We claim that $\pi(\msK)\neq 0$ for all representations $\pi$. Indeed, assume that $\pi_{\varpi}(\msK)= 0$ for some $\varpi \in P^+$. Then \eqref{EqDefModK} and \eqref{EqCommDelt} imply that
\[
(\pi_{\varpi} \otimes \id)\Delta(\msK) (\pi_{\varpi}\otimes \id)\Omega_{\epsilon} =0. 
\]
However, let $\iota$ be a non-zero $U_q(\mfu)$-intertwiner $\C = V_{0} \rightarrow V_{\varpi}\otimes V_{\tau_0(\varpi)}$. Then it follows that 
\[
(\pi_{\varpi} \otimes \pi_{\tau_0(\varpi)})\Delta(\msK) (\pi_{\varpi}\otimes \pi_{\tau_0(\varpi)})\Omega_{\epsilon} \iota = (\pi_{\varpi} \otimes \pi_{\tau_0(\varpi)})\Delta(\msK) \iota  = \varepsilon(\msK)\iota = 0,
\]
contradicting $\varepsilon(\msK) = 1$. 

Let now $\{e_k\}$ be an orthonormal basis of $V_{\pi}$. Then by \eqref{EqImZPhiMatrix} the map $\Phi$ is given by 
\begin{equation}\label{EqImZPhi}
Z_{\pi}(e_k,e_l) \mapsto \sum_{ij} \pi(\msK)_{ij} \tau(Y_{\pi}(e_i,e_k))^{\dag}Y_{\pi}(e_j,e_l).
\end{equation}
As the $\tau(Y(e_i,e_k))^{\dag}Y(e_j,e_l)$ are all linearly independent when $\pi= \pi_{\varpi}, i,j,k,l$ vary, it follows that  that the kernel is trivial unless $\pi(\msK)=0$ for some $\pi$, which we have shown is impossible.
\end{proof}

Let us end by showing a duality relation between $\mcO_q(K\backslash U)$ and $U_q^{\fin}(\mfk')$. We start with a general result, see \cite{KS14} for similar results in the operator algebraic framework and \cite{LVD07,FS09} for the framework of algebraic quantum groups. In the following proposition, we will use again the unitary antipode $R$ \eqref{EqUnitaryAntipode}, acting by duality also on $\mcO_q(U)$ as an involutive Hopf $*$-algebra anti-automorphism.

\begin{Prop}\label{PropDualCoideal}
Let $I_{r/l}$ be a right/left coideal $*$-subalgebra of $\mcO_q(U)$, and let $J_{l/r}$ be a left/right coideal $*$-subalgebra of $\mcU_q(\mfu)$. Then  
\[
\Hat{I}_r = \{X \in \mcU_q(\mfu) \mid \forall f\in I_r: X(-f) = \varepsilon(f)X\} \subseteq \mcU_q(\mfu),
\]
\[
\Hat{I}_l = \{X \in \mcU_q(\mfu) \mid \forall f\in I_l: X(f-) = \varepsilon(f)X\} \subseteq \mcU_q(\mfu)
\]
are respectively a left/right coideal $*$-subalgebra of $\mcU_q(\mfu)$, while
\[
\Hat{J}_l = \{f \in \mcO_q(U) \mid \forall X\in  J_l: f(X-) = \varepsilon(X)f\} \subseteq \mcO_q(U),
\]
\[
\Hat{J}_r = \{f \in \mcO_q(U) \mid \forall X\in  J_r: f(-X) = \varepsilon(X)f\} \subseteq \mcO_q(U)
\]
are respectively a right/left coideal $*$-subalgebra of $\mcO_q(U)$. Moreover,
\begin{enumerate}
\item $\Hat{\Hat{I}}_{r/l} = I_{r/l}$,
\item $\Hat{\Hat{J}}_{l/r}$ is the weak closure of $J_{l/r}$,
\item if $I_{r/l,1} \subseteq I_{r/l,2}$, resp.~ $J_{l/r,1} \subseteq J_{l/r,2}$, then $\Hat{I}_{r/l,1}\supseteq \Hat{I}_{r/l,2}$, resp.~ $\Hat{J}_{l/r,1}\supseteq \Hat{J}_{l/r,2}$.
\end{enumerate}
\end{Prop}
\begin{proof}
By means of the unitary antipode $R$, we can switch between left and right coideal $*$-subalgebras, in a way which is compatible with the above dualities. It is thus sufficient to show that the above proposition holds for a left coideal $*$-subalgebra $I \subseteq \mcO_q(U)$ and a right coideal $*$-subalgebra $J \subseteq \mcU_q(\mfu)$. Moreover, item $(3)$ is immediately clear from the definitions. 

It is further clear that $\Hat{I}$ and $\Hat{J}$ are respectively right and left coideals. To see that they are $*$-algebras, we consider the Heisenberg $*$-algebra $\mcH = \mcU_q(\mfu) \otimes \mcO_q(U)$ consisting of  $\mcU_q(\mfu)$ and $\mcO_q(U)$ as $*$-subalgebras with the interchange relation 
\[
X \cdot f = f_{(1)}\cdot X(f_{(2)}-).
\]
Put 
\begin{equation}\label{EqCommIsDual}
\widetilde{I} = \{X \in \mcU_q(\mfu)\mid \forall f \in I: X\cdot f = f\cdot X\},\qquad \widetilde{J} = \{f \in \mcO_q(U) \mid \forall X\in J: X\cdot f = f\cdot X\}.
\end{equation}
Then clearly $\widetilde{I}$ and $\widetilde{J}$ are $*$-subalgebras. As $\mcO_q(U)$ and $\mcU_q(\mfu)$ are independent within $\mcH$, we have that
\begin{equation}\label{EqDefHeisComm}
X\in \widetilde{I} \quad \Leftrightarrow \quad \forall f\in I: f_{(1)}\otimes X(f_{(2)}-) = f\otimes X.
\end{equation}
Applying the counit to the first leg, we find $\widetilde{I} \subseteq \Hat{I}$. On the other hand, since $I$ is a left coideal, it follows that \eqref{EqDefHeisComm} holds for all $X \in \Hat{I}$, whence $\widetilde{I} = \Hat{I}$ and $\Hat{I}$ a $*$-algebra. On the other hand, 
\begin{equation}\label{EqDefHeisComm2}
f\in \widetilde{J} \quad \Leftrightarrow \quad \forall X\in J: f_{(1)}\otimes X(f_{(2)}-) = f\otimes X.
\end{equation}
Applying the counit to the second leg, we find $\widetilde{J}\subseteq \Hat{J}$, and the right coideal property of $J$ ensures that $X(-g) \in J$ for all $g\in \mcO_q(U)$, from which the equality $\widetilde{J}= \Hat{J}$ follows. In particular $\Hat{J}$ is a $*$-algebra.

Let us now prove item (2). It is clear that $J \subseteq \Hat{\Hat{J}}$, and that $\Hat{\Hat{J}}$ is weakly closed. To see that $\Hat{\Hat{J}}$ equals the weak closure of $J$, note first that since $J \subseteq \mcU_q(\mfu) \cong \prod_{\varpi} \End(V_{\varpi})$ is a unital $*$-algebra, the weak closure $\overline{J}$ will be isomorphic to $\prod_{\pi} \End(W_{\pi})$ for certain finite dimensional Hilbert spaces $W_{\pi}$. Moreover, since $\varepsilon$ is a non-trivial character on $J$, it corresponds to a particular one-dimensional $W_0$. Let $p$ be the projection on $W_0$. Considering $p$ as an element of $\mcU_q(\mfg)$, we have that $pV_{\varpi}$ consists of vectors on which $J$ acts by the counit, and hence $\Hat{J}$ is spanned by all $U(\xi,p\eta)$. In turn, this implies that $\Hat{\Hat{J}}$ consists of all $X\in \mcU_q(\mfu)$ with 
\begin{equation}\label{EqGrouplike}
\Delta(X)(p\otimes 1) = p\otimes X.
\end{equation}
Let now $X$ be an element satisfying \eqref{EqGrouplike}. Applying the comultiplication to the first leg and multiplying the second leg to the left with $S^{-1}$ of the third leg reveals that 
\begin{equation}\label{EqGrouplike2}
(X\otimes 1)\Delta(p) = (1\otimes S^{-1}(X))\Delta(p).
\end{equation}
In particular, since $p \in \overline{J}$ and $\overline{J}$ a right coideal, we find that 
\[
X (\id\otimes f)(\Delta(p)) \in \overline{J},\qquad \forall f\in \mcO_q(U). 
\]
Since all $Y\in J$ satisfy \eqref{EqGrouplike2}, we then obtain that also 
\[
X Y (\id\otimes f)(\Delta(p))Z \in \overline{J},\qquad \forall f\in \mcO_q(U),Y,Z\in J. 
\]
In particular, let $B$ be the weak closure of the algebra generated by the $Y (\id\otimes f)(\Delta(p))Z$ with $f\in \mcO_q(U),Y,Z\in J$. Then $B$ is a weakly closed, $*$-closed two-sided ideal in $\overline{J}$, with 
\[
X B \subseteq \overline{J}.
\]
To finish the proof of (2), we need to show that $B=\overline{J}$ (and hence contains the unit). Suppose however this were not the case. Then as $B$ is a weakly closed, $*$-closed two-sided ideal in $\overline{J} \cong \prod_{\pi} \End(W_{\pi})$, there would exist a finite dimensional representation $V$ of $\mcU_q(\mfu)$ and a non-zero vector $\xi$ with $B\xi = 0$. This implies however $\Delta(p)(1\otimes \xi) = 0$, and hence $0=S(p_{(1)})p_{(2)}\xi = \varepsilon(p)\xi = \xi$, a contradiction. 

Finally, let us prove item (1). Again, the inclusion $I \subseteq \Hat{\Hat{I}}$ is immediate. As $I$ is a left coideal, it is clear that $I$ will be spanned by elements $U_{\varpi}(\xi,\eta)$ where $\xi$ ranges over $V_{\varpi}$ and where $\eta \in W_{\varpi}$ for a certain subspace $W_{\varpi}\subseteq V_{\varpi}$. Let $q_{\varpi}$ be the projection of $V_{\varpi}$ onto $W_{\varpi}$, and let $q = \prod q_{\varpi} \in \mcU_q(\mfu)$. Let $p \in \Hat{I}$ be the projection onto the orthogonal complement of the kernel projection of the restriction of $\varepsilon$ to $\Hat{I}$. Since $I \subseteq \Hat{\Hat{I}}$, we have $p \leq q$. If we can show that $p \in \Hat{I}$, then also $q\leq p$ (since $\varepsilon(p)=1$), and hence $I = \Hat{\Hat{I}}$. 

To see that $p \in \Hat{I}$, we will first faithfully represent $\mcH$ on $\mcO_q(U)$. 
Consider on $\mcO_q(U)$ the \emph{Haar state} $\varphi: \mcO_q(U) \rightarrow \C$, uniquely characterized by the conditions $\varphi(1) = 1$ and $\varphi(U_{\varpi}(\xi,\eta))= 0$ for $\varpi \neq 0$. Then $\varphi$ is faithful, and $\mcO_q(U)$ becomes a pre-Hilbert space for the inner product
\[
\langle f,g\rangle = \varphi(f^*g).
\]
Concretely, we have
\[
\langle U_{\varpi}(\xi,\eta),U_{\varpi'}(\xi',\eta')\rangle = \delta_{\varpi,\varpi'} \frac{\langle \xi',K_{2\rho}\xi\rangle \langle \eta,\eta'\rangle}{\Tr(\pi_{\varpi}(K_{2\rho}))},
\]
see e.g.~ \cite{NT13}*{Theorem 1.4.3 and Proposition 2.4.10}. Consider the following representations:
\[
\pi: \mcO_q(U) \rightarrow \End(\mcO_q(U)),\quad \pi(f)g = fg,
\]
\[
\hat{\pi}: \mcU_q(\mfu) \rightarrow  \End(\mcO_q(U)),\quad \hat{\pi}(X)f = f(-X),
\]
so in particular $\hat{\pi}(X)U(\xi,\eta) = U(\xi,X\eta)$. Then we obtain in particular a $*$-representation
\[
\pi_{\mcH}: \mcH \rightarrow \End_*(\mcO_q(U)),\quad X\mapsto \hat{\pi}(X),\quad f \mapsto \pi(f),
\]
where $\End_*(\mcO_q(U))$ denotes the $*$-algebra of adjointable endomorphisms. We claim that $\pi_{\mcH}$ is faithful. To see this, consider also the $*$-representation
\[
\hat{\pi}': U_q(\mfu) \rightarrow \End_*(\mcO_q(U)),\quad \hat{\pi}'(X)f = f(S^{-1}(X)-),
\]
so that $\hat{\pi}'(X)U(\xi,\eta) = U(S^{-1}(X)^*\xi,\eta)$. Then $\End_*(\mcO_q(U))$ has the right $\hat{\pi}'$-adjoint $U_q(\mfu)$-module $*$-algebra structure
\[
y \lhd X = \hat{\pi}'(S(X_{(1)}))y\hat{\pi}'(X_{(2)}).
\]
Let $\End_*^{\fin}(\mcO_q(U))$ be the locally finite part of $\End_*(\mcO_q(U))$ with respect to $\lhd$, i.e.~ the $*$-algebra of elements which span a finite-dimensional subspace under $\lhd$. Then we can consider the projection map 
\[
E: \End_*^{\fin}(\mcO_q(U)) \rightarrow \End_*^{\fin}(\mcO_q(U))_0
\]
onto the $\lhd$-trivial subspace. Since $x \in \End_*(\mcO_q(U))$ is $\lhd$-trivial if and only if $x$ commutes with all $\hat{\pi}'(X)$ for $X \in U_q(\mfu)$, and since the $\pi(Y)$ commute with the $\hat{\pi}'(X)$ for $X\in U_q(\mfu)$ and $Y\in \mcU_q(\mfu)$, it follows that in fact $\End_*^{\fin}(\mcO_q(U))_0 = \hat{\pi}(\mcU_q(\mfu))$. On the other hand, an easy computation reveals
\[
\pi(f) \lhd X= \pi(f(X-)), 
\]
so in particular $\pi(U(\xi,\eta)) \lhd X= \pi(U(X^*\xi,\eta))$. It follows that
\[
E(\hat{\pi}(X)\pi(f)) = \varphi(f) \hat{\pi}(X),\qquad f\in \mcO_q(U), X\in \mcU_q(\mfu),
\] 
which implies $\pi_{\mcH}$ is faithful. 

We are now ready to show that $p\in \Hat{I}$. Indeed, by the characterisation of $\Hat{I}$ in \eqref{EqCommIsDual} and the faithfulness of $\pi_{\mcH}$, it is sufficient to show that $\hat{\pi}(p)\pi(f) = \pi(f)\hat{\pi}(p)$. However, this follows immediately from the fact that $I$ is a $*$-algebra, for then we have
\[
\pi(f)\hat{\pi}(p) U(\xi,\eta) = f U(\xi,p\eta) = \hat{\pi}(p) (fU(\xi,p\eta)) = \hat{\pi}(p)\pi(f)\hat{\pi}(p) U(\xi,\eta),\qquad \forall f\in I,\forall \xi,\eta. \qedhere
\]
\end{proof}

\begin{Def}\label{DefDualCoid}
We define
\[
\begin{split}
\widehat{U}_q^{\fin}(\mfk') &= \{f \in \mcO_q(U) \mid \forall X\in U_q^{\fin}(\mfk'): f(X-) = \varepsilon(X)f\} \subseteq \mcO_q(U), \\
\widehat{\mcO}_q(K\backslash U) &= \{X \in \mcU_q(\mfu) \mid \forall f\in \mcO_q(K\backslash U): X(-f) = \varepsilon(f)X\} \subseteq \mcU_q(\mfu).
\end{split}
\]
\end{Def}

It follows from Proposition \ref{PropDualCoideal} that $\widehat{U}_q^{\fin}(\mfk')$ is a right coideal $*$-subalgebra in $\mcO_q(U)$, while $\widehat{\mcO}_q(K\backslash U)$ is a left coideal $*$-subalgebra in $\mcU_q(\mfu)$ in the sense that 
\begin{equation}\label{EqRightCoidSub}
(f\otimes \id)\Delta(X) \in \widehat{\mcO}_q(K\backslash U),\qquad X\in \widehat{\mcO}_q(K\backslash U),f\in \mcO_q(U).
\end{equation}
Moreover, if we denote by $\mcU_q(\mfk')$ the weak closure of $U_q^{\fin}(\mfk')$, we also have 
\[
\dbwidehat{U}_q{}^{\!\!\!\fin}(\mfk') = \mcU_q(\mfk'),\qquad \vardbwidehat{\mcO}_q(K\backslash U) = \mcO_q(K\backslash U).
\]

Part of the following proposition can be found (in the untwisted case) in \cite[Corollary 4.5]{KoSt09}.

\begin{Prop}\label{PropInclusion}
We have 
\[
\mcO_q(K\backslash U) \subseteq \widehat{U}_q^{\fin}(\mfk'),\qquad U_q^{\fin}(\mfk')  \subseteq \widehat{\mcO}_q(K\backslash U).
\]
Furthermore, 
\begin{equation}\label{EqCharDualO}
\widehat{\mcO}_q(K\backslash U) = \{X \in \mcU_q(\mfu)\mid (1\otimes \msK)\Delta(X) = (\id\otimes \tau)(\Delta(X)) (1\otimes \msK)\}.
\end{equation}
\end{Prop}

\begin{proof}
From \eqref{EqCoactrho}, we have that
\begin{equation}\label{EqImphi}
\phi(f)(Y)= f(S(\tau(Y_{(1)}))\msK Y_{(2)}), 
\end{equation}
hence $X\in \widehat{\mcO}_q(K\backslash U)$ if and only if
\[
X_{(1)} \otimes S(\tau(X_{(2)}))\msK X_{(3)} = X\otimes \msK,
\]
which is equivalent to 
\[
(1\otimes \msK)\Delta(X) = (\id\otimes \tau)(\Delta(X)) (1\otimes \msK),
\]
proving \eqref{EqCharDualO}.

On the other hand, by \eqref{EqCoactgamm} we see that 
\begin{equation}\label{EqHatPhi}
\hat{\phi}(f) = (f\otimes \id)(\msR_{\tau,21}(\msK\otimes 1)\msR),\qquad f\in \mcO_q(Z_{\nu}).
\end{equation}
We then compute, using that $\msR^*\msR$ commutes with $\Delta(X)$ for $X\in U_q(\mfu)$ and with $\Omega_{\epsilon}$, that
\begin{eqnarray*}
(1 \otimes 1 \otimes \msK)(\id \otimes \Delta)(\msR_{\tau,21}(\msK \otimes 1)\msR)
&=& \msK_3 \msR_{\tau,21}\msR_{\tau,31} \msK_1 \msR_{13}\msR_{12}\\ 
&=& \msR_{\tau,21}\msK_3 \msR_{\tau,31} \msK_1 \msR_{13}\msR_{12}\\ 
&=& \msR_{\tau,21}(\msK_3 \msR_{\tau,31} \msK_1 \msR_{31}^{-1}) (\msR^* \msR)_{13}\msR_{12} \\ 
&\underset{\eqref{EqDefModK}}{=}&  \msR_{\tau,21}\Omega_{\epsilon,13}\Delta(K)_{13} (\msR^* \msR)_{13}\msR_{12} \\
&=& \msR_{\tau,21}(\msR^* \msR)_{13}\Omega_{\epsilon,13}\Delta(K)_{13} \msR_{12} \\
&\underset{\eqref{EqDefModKAlt}}{=}& \msR_{\tau,21}(\msR^* \msR)_{13}(\msR^{-1}_{13}\msK_1 \msR_{\tau,13}\msK_3) \msR_{12}\\ 
&=& \msR_{\tau,21} \msR_{31} \msK_1 \msR_{\tau,13}\msK_3 \msR_{12}\\
&=& \msR_{\tau,21} \msR_{31} \msK_1 \msR_{\tau,13}\msR_{12}\msK_3 \\
&= & (\id \otimes (\id \otimes \tau)\Delta)(\msR_{\tau,21}(\msK \otimes 1)\msR)(1 \otimes 1 \otimes \msK).
\end{eqnarray*}
From \eqref{EqCharDualO}, it now follows that $\hat{\phi}(f) \in  \widehat{\mcO}_q(K\backslash U)$ for all $f\in \mcO_q(Z_{\nu})$, i.e.~ $U_q^{\fin}(\mfk')  \subseteq \widehat{\mcO}_q(K\backslash U)$.

We then have also $\mcO_q(K\backslash U)  = \dbwidehat{\mcO}_q(K\backslash U) \subseteq \widehat{U}_q^{\fin}(\mfk')$.
\end{proof}

\begin{Rem}
It is possible that the equality $\mcO_q(K\backslash U) =\widehat{U}_q^{\fin}(\mfk')$ holds in full generality, but we were not able to prove this - conditions for equality of these algebras were already asked for in the non-modified case in \cite[Remark 4.6]{KoSt09}. We will however verify this property directly in many cases, sometimes by quite ad hoc computations. 
\end{Rem}

In the next sections, we will construct $*$-compatible $\nu$-modified universal $K$-matrices in the case where $\nu$ is either of flag type or of symmetric type. 

\section{Quantum flag varieties}\label{SecFlag}

In this section, we will examine in more detail the case where $\nu = (\id,\epsilon) \in \End_*(U_q(\mfb))$ is of flag type, so $\epsilon_r \in \{0,1\}$ for all $r$. We resume further notation from Section \ref{SecPrelim} and Section \ref{SecTwistBraid}.

\subsection{Construction of a $*$-compatible $\nu$-modified universal $K$-matrix}

Let $S_0 = \{r\mid \epsilon_r = 1\}$, and let $S = \{\tau_0(r)\mid r\in S_0\}$, where we recall that $\tau_0$ is the Dynkin diagram automorphism defined by the longest word $w_0$ of $W$, that is $\tau_0(\omega) = - w_0(\omega)$ for $\omega \in P$. 

\begin{Def}\label{DefUniKFlag}
We define $\Elw \in \mcU_q(\mfu)$
by 
\[
\Elw \xi := \epsilon_{\varpi - w_0(\mathrm{wt}(\xi))} \xi, \quad \xi \in V_\varpi.
\]
\end{Def}

Note that this is meaningful: if $\xi$ is a non-zero weight vector at $\omega$, then $T_{w_0}\xi$ is a non-zero weight vector at $w_0\omega$, hence $w_0\omega = \varpi - \alpha$ for some $\alpha \in Q^+$. 

We will in the following also write $\nu_0 := \tau_0 \circ \nu \circ \tau_0 \in \End_*(U_q(\mfb))$. Explicitly we have
\[
\nu_0(K_\omega) = K_\omega, \quad \nu_0(E_r) = \epsilon_{\tau_0(r)} E_r, \quad \nu_0(F_r) = \epsilon_{\tau_0(r)} F_r.
\]
Observe that $\nu_0(E_r) = E_r$ for $r \in S$ and $\nu_0(E_r) = 0$ for $r \notin S$. Similarly for $F_r$.

We denote by $Q_S^+$ the positive span of the simple roots $\alpha_r$ with $r \in S$. 

\begin{Lem}\label{LemCommRel}
Let $X \in U_q(\mfb)$ and $Y \in U_q(\mfb^-)$. Then
\[
\Elw X = \nu_0(X) \Elw, \quad Y \Elw = \Elw \nu_0(Y).
\]
Moreover, if $X \in U_q(\mfb)_{\alpha}$ and $Y \in U_q(\mfb^-)_{-\beta}$ for $\alpha, \beta \in Q_S^+$ then $\Elw X = X \Elw$ and $\Elw Y = Y \Elw$.
\end{Lem}

\begin{proof}
Let $\xi$ be of weight $\omega$. First we compute
\[
\Elw E_r \xi = \epsilon_{\varpi - w_0(\omega + \alpha_r)} E_r \xi = \epsilon_{\tau_0(r)} E_r \epsilon_{\varpi - w_0(\omega)} \xi = \nu_0(E_r) \Elw \xi.
\]
Next we compute
\[
F_r \Elw \xi = \epsilon_{\varpi - w_0(\omega)} F_r \xi = \epsilon_{\varpi - w_0(\omega - \alpha_r)} \epsilon_{\tau_0(r)} F_r \xi = \Elw \nu_0(F_r) \xi.
\]
The final statement in the lemma follows immediately from the fact that $\epsilon_{\tau_0(r)} = 1$ for $r \in S$.
\end{proof}

\begin{Lem}
We have $\Omega_\epsilon \Delta(\Elw) = \Elw \otimes \Elw$.
\end{Lem}

\begin{proof}
Given $V_\varpi \subseteq V_{\varpi^\prime} \otimes V_{\varpi^{\prime \prime}}$ we have $\Omega_\epsilon |_{V_\varpi} = \epsilon_{\varpi^\prime + \varpi^{\prime \prime} - \varpi}$. For $\xi \in V_\varpi$ of weight $\omega$ we get
\[
\Omega_\epsilon \Delta(\Elw) \xi = \epsilon_{\varpi^\prime + \varpi^{\prime \prime} - \varpi} \epsilon_{\varpi - w_0(\omega)} \xi = \epsilon_{\varpi^\prime + \varpi^{\prime \prime} - w_0(\omega)}  \xi.
\]
Now consider $\xi^\prime \otimes \xi^{\prime\prime} \in V_{\varpi^\prime} \otimes V_{\varpi^{\prime \prime}}$ with $\mathrm{wt}(\xi^\prime) = \omega^\prime$ and $\mathrm{wt}(\xi^{\prime\prime}) = \omega^{\prime \prime}$. Then we have
\[
\Elw \xi^\prime \otimes \Elw \xi^{\prime\prime} = \epsilon_{\varpi^\prime - w_0(\omega^\prime)} \xi^\prime \otimes \epsilon_{\varpi^{\prime\prime} - w_0(\omega^{\prime\prime})} \xi^{\prime \prime} = \epsilon_{\varpi^{\prime} + \varpi^{\prime \prime} - w_0(\omega^\prime + \omega^{\prime \prime})} \xi^\prime \otimes \xi^{\prime \prime}.
\]
From this we conclude that $(\Elw \otimes \Elw) \xi = \epsilon_{\varpi^\prime + \varpi^{\prime \prime} - w_0(\omega)} \xi$.
Comparing the two expressions we obtain the equality $\Omega_\epsilon \Delta(\Elw) = \Elw \otimes \Elw$.
\end{proof}

\begin{Theorem}\label{TheoFlagK}
The element $\Elw$ is a $*$-compatible $\nu$-modified universal K-matrix.
\end{Theorem}

\begin{proof}
Since $\msR \in \mcU_q(\mfb^+) \hat{\otimes} \mcU_q(\mfb^-)$, we obtain by Lemma \ref{LemCommRel} 
\[
(1 \otimes \Elw) \mathscr{R}_{21} (\Elw \otimes 1) \mathscr{R}_{21}^{-1} = \mathscr{R}_{\nu_0, 21} (\Elw \otimes \Elw) \mathscr{R}_{21}^{-1} = \mathscr{R}_{\nu_0, 21} (\Elw \otimes 1) \mathscr{R}_{\nu_0, 21}^{-1} (1 \otimes \Elw).
\]
We have $\nu_0(F_r) = 0$ for $r \notin S$, hence the first leg of the element $\mathscr{R}_{\nu_0, 21}$ only contains expressions in the generators $F_r$ with $r \in S$, and similarly for its inverse. Then again by Lemma \ref{LemCommRel} we have $(\Elw\otimes 1) \mathscr{R}_{\nu_0, 21}^{-1} = \mathscr{R}_{\nu_0, 21}^{-1} (\Elw\otimes 1)$ and we get
\[
(1 \otimes \Elw) \mathscr{R}_{21} (\Elw \otimes 1) \mathscr{R}_{21}^{-1} = \mathscr{R}_{\nu_0, 21} \mathscr{R}_{\nu_0, 21}^{-1} (\Elw \otimes \Elw) = \Elw \otimes \Elw.
\]
Since $\Omega_\epsilon \Delta(\Elw) = \Elw \otimes \Elw$ by the previous lemma, we see that $\Elw$ is a $\nu$-modified universal $K$-matrix. The $*$-compatibility is immediate, since ${\Elw}^* = \Elw$.
\end{proof}

\subsection{Comparison of the coideal subalgebras $U_q^{\fin}(\mfk')$ and $U_q(\mfk_S)$}

Fix $S$, $S_0$ as in the previous section. Recall the notations introduced following Theorem \ref{TheoOneToOneCorr}, using the $*$-compatible $\nu$-modified universal $K$-matrix $\msK$ constucted in Definition \ref{DefUniKFlag}. Let further $\mfk_S$ be the compact form of the Levi factor of the parabolic subalgebra associated to $S$, so that the complexification $\mfk_S^{\C}\subseteq \mfg$ is generated by $\mfh$ and the $e_r$, $f_r$ with $r \in S$. Let $U_q(\mfk_S)\subseteq U_q(\mfu)$ be the quantized enveloping $*$-algebra of $\mfk_S$, generated by the $K_{\omega}$ for $\omega \in P$ and the $E_r$, $F_r$ for $r \in S$, and let $\mcU_q(\mfk_S)$ be its weak completion. Finally, let $\mcO_q(K_S\backslash U) = \hat{U}_q(\mfk_S)$  be the dual right coideal $*$-subalgebra of $\mcO_q(U)$.

We denote by $W_S\subseteq W$ the subgroup generated by the $s_r$ with $r\in S$, and $w_{S,0}$ the longest element in $W_S$ with respect to the natural word length on $W_S$. We let $w_S = w_{S,0}w_0$. 

\begin{Theorem}\label{TheoIdUnivFlag}
The equality $\mcU_q(\mfk') = \mcU_q(\mfk_S)$ holds. 
\end{Theorem}
\begin{proof}
We recall from \eqref{EqHatPhi} that 
\[
\hat{\phi}(f) = (f\otimes \id)(\msR_{21}(\msK\otimes 1)\msR). 
\]
Since by Lemma \ref{LemCommRel}
\[
\msR_{21}(\msK \otimes 1) = (\msK\otimes 1)\msR_{\nu_0,21},\quad (\msK\otimes 1)\msR = \msR_{\nu_0}(\msK\otimes 1),
\]
we obtain from the fact that $\msK$ is a self-adjoint projection that
\[
\hat{\phi}(Z(\xi,\eta)) = (U(\msK\xi,\msK\eta)\otimes \id)(\msR_{\nu_0,21}\msR_{\nu_0}).
\]
This shows that $U_q^f(\mfk') \subseteq U_q(\mfk_S)$, and hence $\mcU_q(\mfk') \subseteq \mcU_q(\mfk_S)$.

On the other hand, let $\eta_{w_0\varpi}\in V_{\varpi}$ be a lowest weight vector for $U_q(\mfu)$. Then $U_q(\mfk_S)\eta_{w_0\varpi}$ is an irreducible $U_q(\mfk_S)$-module of highest weight $w_S(\varpi) = w_{S,0}w_0(\varpi)$. Then with $\xi_{w_S(\varpi)}$ a corresponding highest weight vector of unit norm, we find
\[
\hat{\phi}(Z(\xi_{w_S(\varpi)},\xi_{w_S(\varpi)})) = (U(\msK\xi_{w_S(\varpi)},\msK\xi_{w_S(\varpi)})\otimes \id)(\msR_{\nu_0,21}\msR_{\nu_0}).
\] 
Now 
\[
\msK \xi_{w_S(\varpi)} = \epsilon_{\varpi - w_0w_S(\varpi)}\xi_{w_S(\varpi)} =  \epsilon_{\varpi - w_{S,0}(\varpi)}\xi_{w_S(\varpi)} =  \xi_{w_S(\varpi)},
\]
since $\varpi - w_{S,0}(\varpi) \in Q_S^+$. It follows that
\[
\hat{\phi}(Z(\xi_{w_S(\varpi)},\xi_{w_S(\varpi)})) =  (U(\xi_{w_S(\varpi)},\xi_{w_S(\varpi)})\otimes \id)(\msR_{\nu_0,21}\msR_{\nu_0}) =   (U(\xi_{w_S(\varpi)},\xi_{w_S(\varpi)})\otimes \id)(\msQ^2) = K_{-2 w_S(\varpi)}.  
\]
Similarly, since $F_r\xi_{w_S(\varpi)}$ still lies in the range of $\msK$ for $r\in S$, it follows that
\[
\hat{\phi}(Z(\xi_{w_S(\varpi)},F_r\xi_{w_S(\varpi)})) =  (U(\xi_{w_S(\varpi)},F_r\xi_{w_S(\varpi)})\otimes \id)(\msR_{\nu_0,21}\msR_{\nu_0}) 
\]
will be a scalar multiple of $K_{\alpha_r-2w_S(\varpi)}F_r$. Since $\mcU_q(\mfk')$ can be identified with the bicommutant of the $*$-algebra $U_q^{\fin}(\mfk')$, it now follows immediately that $\mcU_q(\mfk') = \mcU_q(\mfk_S)$. 
\end{proof}

\subsection{Comparison of the coideal subalgebras $\mcO_q(K\backslash U)$ and $\mcO_q(K_S\backslash U)$}

Fix again the setting as in the previous subsection. We then also have the following dual result.

\begin{Theorem}\label{TheoMainFlag2}
The equality $\mcO_q(K\backslash U) = \mcO_q(K_S\backslash U)$ holds.
\end{Theorem}
\begin{proof}
We already know by Proposition \ref{PropInclusion} and Theorem \ref{TheoIdUnivFlag} that $\mcO_q(K\backslash U)  \subseteq \mcO_q(K_S\backslash U)$. On the other hand, consider the elements $a_{\varpi} = Z(\xi_{\varpi},\xi_{\varpi})$ introduced in \eqref{EqDefa}. Then 
\[
\phi(a_{\varpi}) = \sum_i U_{\varpi}(e_i,\xi_{\varpi})^* U_{\varpi}(\msK e_i,\xi_{\varpi}),
\]
for $e_i$ an orthonormal basis of $V_{\varpi}$. However, choosing the $e_i$ to be eigenvectors of $\msK$, the non-zero $\msK e_i$ then form an orthonormal basis of the $U_q(\mfk_S)$-module spanned by the $U_q(\mfu)$-lowest weight vector $\eta_{w_0\varpi}$ in $V_{\varpi}$. Following the reasoning as in \cite{DCN15}*{Proposition 2.3}, we see that the $\phi(a_{\varpi})$ generate $\mcO_q(K_S\backslash U)$ as a $U_q(\mfu)$-module. Since $\phi$ is $U_q(\mfu)$-equivariant, we must hence have $\mcO_q(K\backslash U) = \mcO_q(K_S\backslash U)$.  
\end{proof}

\section{Quantum symmetric spaces}\label{SecSymm}

In this section we fix a semisimple Lie algebra $\mfg$ with Dynkin diagram $\Gamma$ and underlying set $I$, and use also further notation as in Section \ref{SecPrelim} and Section \ref{SecTwistBraid}. We will further write $\mfa \subseteq \mfg$ for the real span of the $h_r$, and $T = e^{i \mfa} \subseteq U$ for the maximal torus of $U$ associated to $\mfh$. 
 
\subsection{Involutions in Satake form}

We recall some preliminaries on involutions of $\mfg$ in maximally split form. We follow mainly the exposition in \cite{Kol14}*{Section 2}.

For $X\subseteq I$, we write $\mfg_X$ for the semisimple Lie algebra generated by the $\{e_r,f_r,h_r\mid r\in X\}$, and we write $\Delta_X\subseteq \Delta$ for the associated root system. We denote by $Q_X \subseteq Q$ the root lattice spanned by $\Delta_X$, equipped with the restriction of the bilinear form $(-,-)$. We write $W_X \subseteq W$ for the Weyl group of $\mfg_X$, and $w_X$ for the longest element in $W_X$. We further write
\[
\rho_X^{\vee} = \frac{1}{2}\sum_{\alpha \in \Delta_X^{+}} \alpha^{\vee}\in P. 
\]

\begin{Def}\label{DefEnhSat}
We call \emph{concrete Satake diagram\footnote{In what follows, we will exclude the trivial Satake diagram $X = I$ and $\tau = \id$}} on the Dynkin diagram $\Gamma$ the datum of a subset $X\subseteq I$ and a Dynkin diagram involution $\tau$ of $I$ such that the following two conditions are satisfied:
\begin{itemize}
\item $\tau$ preserves $X$ and coincides on it with the action of $-w_X$,
\item $(\alpha_r,\rho_X^{\vee})  \in  \Z$ for all $r\in I \setminus X$ with $\tau(r) = r$. 
\end{itemize}

We call \emph{enhanced Satake diagram} a concrete Satake diagram that is also equipped with a function
\[
z: I \rightarrow \{\pm 1\}
\] 
such that 
\[
z_r = 1\quad \textrm{ when }(\alpha_r,\rho_X^{\vee})\in \Z,\qquad z_rz_{\tau(r)} = -1\quad \textrm{ when }(\alpha_r,\rho_X^{\vee})\notin \Z.
\]
\end{Def}

Note that such a $z$ always exists, since $(\alpha_r,\rho_X^{\vee})\notin \Z$ implies $\tau(r)\neq r$. Moreover, $z_r = 1$ for $r\in X$ as then $(\alpha_r,\rho_X^{\vee}) = 1$. A function $z$ as above satisfies in particular 
\begin{equation}\label{EqProdzz}
z_rz_{\tau(r)} (-1)^{(\alpha_r,2\rho_X^{\vee})}= 1,\qquad \forall r\in I. 
\end{equation}

By direct diagram checking, one verifies that the resulting Satake diagrams correspond to the (unions of) Satake diagrams as in \cite{Ar62}, from which we also borrow the nomenclature, together with the diagrams consisting of two copies of the same Dynkin diagram and the involution interchanging the two copies. Note however that as we are fixing the Dynkin diagram beforehand, we also need to treat as separate the Satake diagrams obtained by applying Dynkin diagram automorphisms. This means:
\begin{enumerate}
\item In the non-simple case, the ordering of the different components is taken into account. 
\item In the $DIII$-case $\mfu^*_{2p}(\mathbb{H}) = \mfso^*(4p)$ we include also the Satake diagram with the coloring of the fork endpoints interchanged, see Table \ref{Tablesostar}.
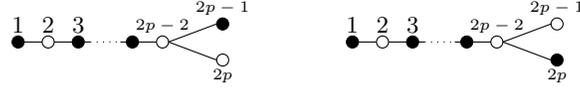
\begin{table}[ht]
\caption{Concrete Satake diagrams for $\mfso^*(4p)$}\label{Tablesostar}
\begin{center}
\bgroup
\def\arraystretch{3}
{\setlength{\tabcolsep}{1.5em}
\begin{tabular}{cc}
  \begin{tikzpicture}[scale=.4,baseline]
\draw[fill = black] (0cm,0) circle (.2cm) node[above]{\small $1$} ;
\draw (0.2cm,0) -- +(0.6cm,0);
\draw (1cm,0) circle (.2cm)node[above]{\small $2$};
\draw (1.2cm,0) -- +(0.6cm,0);
\draw[fill = black] (2cm,0) circle (.2cm)node[above]{\small $3$};
\draw (2.2cm,0) -- +(0.2cm,0);
\draw[dotted] (2.4cm,0) --+ (1cm,0);
\draw (3.4cm,0) --+ (0.2cm,0);
\draw[fill = black] (3.8cm,0) circle (.2cm); 
\draw (4cm,0) --+ (0.6cm,0);
\draw (4.8cm,0) circle (.2cm)  node[above]{\tiny $2p-2$};
\draw (5cm,0) --+ (1.6,0.6);
\draw (5cm,0) --+ (1.6,-0.6);
\draw[fill = black] (6.8cm,0.6) circle (.2cm) node[above]{\tiny $2p-1$} ;
\draw (6.8cm,-0.6) circle (.2cm) node[below]{\tiny $2p$} ;
\end{tikzpicture}
& 
\begin{tikzpicture}[scale=.4,baseline]
\draw[fill = black] (0cm,0) circle (.2cm) node[above]{\small $1$} ;
\draw (0.2cm,0) -- +(0.6cm,0);
\draw (1cm,0) circle (.2cm)node[above]{\small $2$};
\draw (1.2cm,0) -- +(0.6cm,0);
\draw[fill = black] (2cm,0) circle (.2cm)node[above]{\small $3$};
\draw (2.2cm,0) -- +(0.2cm,0);
\draw[dotted] (2.4cm,0) --+ (1cm,0);
\draw (3.4cm,0) --+ (0.2cm,0);
\draw[fill = black] (3.8cm,0) circle (.2cm); 
\draw (4cm,0) --+ (0.6cm,0);
\draw (4.8cm,0) circle (.2cm)  node[above]{\tiny $2p-2$};
\draw (5cm,0) --+ (1.6,0.6);
\draw (5cm,0) --+ (1.6,-0.6);
\draw (6.8cm,0.6) circle (.2cm) node[above]{\tiny $2p-1$} ;
\draw[fill = black] (6.8cm,-0.6) circle (.2cm) node[below]{\tiny $2p$} ;
\end{tikzpicture}
\end{tabular}
}
\egroup
\end{center}
\end{table}

\item In the $D$-cases $\mfso(1,7)$, $\mfso(2,6)$ and $\mfso(3,5)$ we include the Satake diagrams obtained by rotation, see Table \ref{Tableso8}. Note that one of these establishes the isomorphism $\mfso(2,6)\cong \mfso^*(8)$.  

\begin{table}[ht]
\caption{Concrete Satake diagrams for $\mfso(p,8-p)$ with $1\leq p \leq 3$}\label{Tableso8}
\begin{center}
\bgroup
\def\arraystretch{4}
{\setlength{\tabcolsep}{1.5em}
\begin{tabular}{cccc}
$\mfso(1,7) \;(DII)$
& 
\begin{tikzpicture}[scale=.4,baseline]
\draw (0cm,0) circle (.2cm) node[above]{\small $1$} ;
\draw (0.2cm,0) -- +(0.6cm,0);
\draw[fill = black] (1cm,0) circle (.2cm) node[above]{\small $2$} ;
\draw (1.2cm,0) --+ (0.6cm,0.6cm);
\draw (1.2cm,0) --+ (0.6cm,-0.6cm);
\draw[fill = black] (1.95cm,0.7) circle (.2cm) node[above] {\small $3$}; 
\draw[fill = black] (1.95cm,-0.7) circle (.2cm)  node[below]{\small $4$};
\end{tikzpicture}
& 
\begin{tikzpicture}[scale=.4,baseline]
\draw[fill = black] (0cm,0) circle (.2cm) node[above]{\small $1$} ;
\draw (0.2cm,0) -- +(0.6cm,0);
\draw[fill = black] (1cm,0) circle (.2cm) node[above]{\small $2$} ;
\draw (1.2cm,0) --+ (0.6cm,0.6cm);
\draw (1.2cm,0) --+ (0.6cm,-0.6cm);
\draw (1.95cm,0.7) circle (.2cm) node[above] {\small $3$}; 
\draw[fill = black] (1.95cm,-0.7) circle (.2cm)  node[below]{\small $4$};
\end{tikzpicture}
& 
\begin{tikzpicture}[scale=.4,baseline]
\draw[fill = black] (0cm,0) circle (.2cm) node[above]{\small $1$} ;
\draw (0.2cm,0) -- +(0.6cm,0);
\draw[fill = black] (1cm,0) circle (.2cm) node[above]{\small $2$} ;
\draw (1.2cm,0) --+ (0.6cm,0.6cm);
\draw (1.2cm,0) --+ (0.6cm,-0.6cm);
\draw[fill = black] (1.95cm,0.7) circle (.2cm) node[above] {\small $3$}; 
\draw (1.95cm,-0.7) circle (.2cm)  node[below]{\small $4$};
\end{tikzpicture}\\ 
$\mfso(2,6)\; (DI)$  
& 
\begin{tikzpicture}[scale=.4,baseline]
\draw (0cm,0) circle (.2cm) node[above]{\small $1$} ;
\draw (0.2cm,0) -- +(0.6cm,0);
\draw (1cm,0) circle (.2cm) node[above]{\small $2$} ;
\draw (1.2cm,0) --+ (0.6cm,0.6cm);
\draw (1.2cm,0) --+ (0.6cm,-0.6cm);
\draw[fill = black] (1.95cm,0.7) circle (.2cm) node[above] {\small $3$}; 
\draw[fill = black] (1.95cm,-0.7) circle (.2cm)  node[below]{\small $4$};
\end{tikzpicture}
& 
\begin{tikzpicture}[scale=.4,baseline]
\draw[fill = black] (0cm,0) circle (.2cm) node[above]{\small $1$} ;
\draw (0.2cm,0) -- +(0.6cm,0);
\draw (1cm,0) circle (.2cm) node[above]{\small $2$} ;
\draw (1.2cm,0) --+ (0.6cm,0.6cm);
\draw (1.2cm,0) --+ (0.6cm,-0.6cm);
\draw (1.95cm,0.7) circle (.2cm) node[above] {\small $3$}; 
\draw[fill = black] (1.95cm,-0.7) circle (.2cm)  node[below]{\small $4$};
\end{tikzpicture}
& 
\begin{tikzpicture}[scale=.4,baseline]
\draw[fill = black] (0cm,0) circle (.2cm) node[above]{\small $1$} ;
\draw (0.2cm,0) -- +(0.6cm,0);
\draw (1cm,0) circle (.2cm) node[above]{\small $2$} ;
\draw (1.2cm,0) --+ (0.6cm,0.6cm);
\draw (1.2cm,0) --+ (0.6cm,-0.6cm);
\draw[fill = black] (1.95cm,0.7) circle (.2cm) node[above] {\small $3$}; 
\draw (1.95cm,-0.7) circle (.2cm)  node[below]{\small $4$};
\end{tikzpicture}\\
$\mfso(3,5) \; (DI)$ 
&  
\begin{tikzpicture}[scale=.4,baseline]
\node (v1) at (2,0.8) {};
\node (v2) at (2,-0.8) {};
\draw (0cm,0) circle (.2cm) node[above]{\small $1$} ;
\draw (0.2cm,0) -- +(0.6cm,0);
\draw (1cm,0) circle (.2cm) node[above]{\small $2$} ;
\draw (1.2cm,0) --+ (0.6cm,0.6cm);
\draw (1.2cm,0) --+ (0.6cm,-0.6cm);
\draw (1.95cm,0.7) circle (.2cm) node[above] {\small $3$}; 
\draw (1.95cm,-0.7) circle (.2cm)  node[below]{\small $4$};
\draw[<->]
(v1) edge[bend left] (v2);
\end{tikzpicture}
& 
\begin{tikzpicture}[scale=.4,baseline]
\node (v1) at (0,0) {};
\node (v2) at (2,-0.8) {};
\draw (0cm,0) circle (.2cm) node[above]{\small $1$} ;
\draw (0.2cm,0) -- +(0.6cm,0);
\draw (1cm,0) circle (.2cm) node[above]{\small $2$} ;
\draw (1.2cm,0) --+ (0.6cm,0.6cm);
\draw (1.2cm,0) --+ (0.6cm,-0.6cm);
\draw (1.95cm,0.7) circle (.2cm) node[above] {\small $3$}; 
\draw (1.95cm,-0.7) circle (.2cm)  node[below]{\small $4$};
\draw[<->]
(v1) edge[bend right](v2);
\end{tikzpicture}
& 
\begin{tikzpicture}[scale=.4,baseline]
\node (v1) at (0,0) {};
\node (v2) at (2,0.8) {};
\draw (0cm,0) circle (.2cm) node[below]{\small $1$} ;
\draw (0.2cm,0) -- +(0.6cm,0);
\draw (1cm,0) circle (.2cm) node[below]{\small $2$} ;
\draw (1.2cm,0) --+ (0.6cm,0.6cm);
\draw (1.2cm,0) --+ (0.6cm,-0.6cm);
\draw (1.95cm,0.7) circle (.2cm) node[above] {\small $3$}; 
\draw (1.95cm,-0.7) circle (.2cm)  node[below]{\small $4$};
\draw[<->]
(v1) edge[bend left](v2);
\end{tikzpicture}\\
\end{tabular}
}
\egroup
\end{center}
\end{table}
\end{enumerate} 
When we do not care about the connection with the underlying Dynkin diagram, we will talk of an \emph{abstract Satake diagram}. More precisely, let us call concrete Satake diagrams on respective Dynkin diagrams $\Gamma,\Gamma'$ \emph{equivalent} if one is carried to the other by an isomorphism of Dynkin diagrams. Then we refer to abstract Satake diagram as an equivalence class under this relation.

Recall that $\tau_0$ is the Dynkin diagram automorphism induced by $-w_0$. 

\begin{Lem}\label{LemCommTau0Tau}
Let $(X,\tau,z)$ be an enhanced Satake diagram. Then 
\begin{itemize}
\item $\tau_0(X) = \tau(X) =  X$,
\item $\tau\tau_0 = \tau_0\tau$,
\item $z$ is $\tau\tau_0$-invariant,
\item $z$ is $w_X$-invariant.
\end{itemize}
\end{Lem} 
\begin{proof}
For the first three properties we refer to \cite{BK15b}*{Remark 7.2}. Note here that $\tau_0$ commutes in fact with any Dynkin diagram automorphism, as $\tau_0$ is trivial in the only case $D_4$ where there is more than one non-trivial Dynkin diagram automorphism. The last property follows from the fact that $z_r=1$ for $r\in X$. 
\end{proof}

Fix now an enhanced Satake diagram $(X,\tau,z)$. One constructs explicitly an involution $\theta = \theta(X,\tau,z)$ of $\mfg$ as follows.

Extend first again $\tau$ to an automorphism of $\mfg$ by 
\[
\tau(e_r) = e_{\tau(r)},\qquad \tau(f_r) = f_{\tau(r)},\qquad \tau(h_r) = h_{\tau(r)}. 
\]
Let $\omega$ be the Chevalley involution of $\mfg$, which is the complex Lie algebra automorphism determined by 
\[
\omega(e_r) = -f_r,\qquad \omega(f_r) = -e_r,\qquad \omega(h) = -h.
\]
Let
\begin{equation}\label{EqInvSimple}
m_r = \exp(e_r)\exp(-f_r)\exp(e_r) \in U
\end{equation}
and identify $s_r = \Ad(m_r)_{\mid \mfh} \in W$. Let 
\begin{equation}\label{EqRedExprX}
w_0 = s_{r_1}\ldots s_{r_N},\qquad w_X = s_{r_1'}\ldots s_{r_M'}
\end{equation}
be reduced expressions for the longest elements in respectively $W$ and $W_X$, and write the corresponding elements in $U$ as
\[
m_0 = m_{r_1}\ldots m_{r_N} \in U,\qquad m_X = m_{r_1'} \ldots m_{r_M'} \in U. 
\]
We note that $m_0,m_X$ are independent of the chosen reduced expressions.

Finally, note that $z$ may be extended uniquely to a unitary character on $Q$. The following lemma ensures that $z$ can be extended to a character on the weight lattice, i.e.~ an element of $T$, such that some of the symmetry properties of $z$ are preserved.

\begin{Lem}\label{LemChoiceInvExt}
Let $(X,\tau,z)$ be an enhanced Satake diagram. Then we can choose $\chi_0 \in \mfa$ with $\tau\tau_0(\chi_0) = \chi_0$ and such that $\widetilde{z} = e^{2\pi i \chi_0} \in T$ is an extension of $z$. 
\end{Lem}

\begin{proof}
For $X= \emptyset$ or $\tau = \id$ we have $z=1$, and we can hence take $\chi_0=0$. On the other hand, if $\tau\tau_0 = \id$ the existence of $\chi_0$ follows from $T = \exp(i\mfa)$.  We may thus assume that $\mfg$ is simple with $X\neq \emptyset$, $\tau\neq \id$ and $\tau\tau_0\neq \id$. By direct diagram checking (see again also \cite{BK15b}*{Remark 7.2}) it can be verified however that this can only happen for $\mfg$ of type $D_{l}$ for $l$ even and $\mfg_{\theta}\cong \mfso(p,2l-p)$ for $p$ odd, for which the Satake diagram is given by 
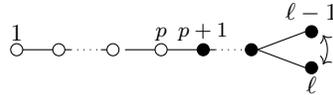
\[
 \begin{tikzpicture}[scale=.4,baseline=1cm]
\node (v1) at (10,0.8) {};
\node (v2) at (10,-0.8) {};
\draw (0cm,0) circle (.2cm) node[above]{\small $1$} ;
\draw (0.2cm,0) -- +(1cm,0);
\draw (1.4cm,0) circle (.2cm);
\draw (1.6cm,0) -- +(0.2cm,0);
\draw[dotted] (1.8cm,0) --+ (1cm,0);
\draw (2.8cm,0) --+ (0.2cm,0);
\draw (3.2cm,0) circle (.2cm); 
\draw (3.6cm,0) --+ (1cm,0);
\draw (4.8cm,0) circle (.2cm)  node[above]{\footnotesize $p$};
\draw (5cm,0) --+ (1cm,0);
\draw[fill = black] (6.2cm,0) circle (.2cm)  node[above]{\footnotesize $p+1$};
\draw (6.4cm,0) -- +(0.2cm,0);
\draw[dotted] (6.6cm,0) --+ (1cm,0);
\draw (7.6cm,0) --+ (0.2cm,0);
\draw[fill = black] (7.8cm,0) circle (.2cm);
\draw (8cm,0) --+ (1.6,0.6);
\draw (8cm,0) --+ (1.6,-0.6);
\draw[fill = black] (9.8cm,0.6) circle (.2cm) node[above]{\small $\ell-1$} ;
\draw[fill = black] (9.8cm,-0.6) circle (.2cm) node[below]{\small $\ell$} ;
\draw[<->]
(v1) edge[bend left] (v2);
\end{tikzpicture}
\]
(where contrary to custom we indicated also the action of $\tau$ on $X$ for clarity). It is clear that then $(\alpha_r,\rho_X^{\vee}) \in \Z$ for all $r\in I$, except possibly for $r=p$. However, using e.g.~ \cite{OV90}*{Reference Chapter, Section 2, Table 1}, we find that also $(\alpha_p,\rho_X^{\vee}) = l-p-1 \in \Z$. Hence $z=1$, so we can take $\chi_0=0$ in this case.
\end{proof}

In the following, we will fix $\chi_0$ and $\widetilde{z}$ as above.

\begin{Def}\label{DefConcSat}
Let $(X,\tau,z)$ be an enhanced Satake diagram for $\Gamma$. We define 
\begin{equation}\label{EqFormThet}
\theta = \theta(X,\tau,z) = \Ad(\widetilde{z}) \circ \tau \circ \omega \circ \Ad(m_X) \in \Aut(\mfg),
\end{equation}
and call it the \emph{Satake involution} of $\mfg$ associated to $(X,\tau,z)$. 
\end{Def}

\begin{Rem}
Note that the Satake involution indeed only depends on $(X,\tau,z)$. Moreover, it is easy to see that $\theta$ depends only on $(X,\tau)$ up to inner conjugacy, with $(X,\tau)$ corresponding to a unique inner conjugacy class, see Theorem \ref{TheoUniqueInnConj}.
\end{Rem}

\begin{Rem} 
It is not hard to check that $\theta$ is indeed an involution, and that $\theta$ commutes with $*$. In particular, $\theta$ restricts to a Lie algebra involution of $\mfu$. Associated to $\theta$ we then have the real Lie algebra
\[
\mfg_{\theta} = \{X\in \mfg \mid \theta(X)^* = -X\},
\]
and all real semisimple Lie algebras arise in this way, their isomorphism class uniquely determined by the associated abstract  Satake diagram. 
\end{Rem}

Note (for example by \cite{B-VB-PBMR95}*{Lemme 4.9}) that one can write $\omega = \tau_0 \circ \Ad(m_0) = \Ad(m_0) \circ \tau_0$, from which it follows that we can also write the Satake involution as
\begin{equation}\label{EqCompInnOut}
\theta = \Ad(\widetilde{z}) \circ \tau \circ \tau_0 \circ \Ad(m_0) \circ \Ad(m_X). 
\end{equation}

Let us write $\Theta$ for the dual of the restriction of $\theta$ to $\mfh$. From the Definition of $\theta$ and Lemma \ref{LemCommTau0Tau}, we immediately obtain the following. 

\begin{Lem}[\cite{Kol14}*{Equation (2.10)}]\label{LemCommThetaTau}
We have $\Theta(\alpha) = -w_X\tau(\alpha)$, and $\Theta$ commutes with $\tau_0$ and $\tau$. 
\end{Lem}

\subsection{Construction of a $*$-compatible $\nu$-modified universal $K$-matrix}

Let $(X,\tau,z)$ be an enhanced Satake diagram for $\Gamma$, and let $\theta = \theta(X,\tau,z)$ be the associated Satake involution. Fix an extension to the weight lattice
\[
\widetilde{z} = e^{2\pi i \chi_0} \in T
\]
with $\chi_0 \in \mfa$ a $\tau\tau_0$-invariant element as in Lemma \ref{LemChoiceInvExt}. We then put 
\begin{equation}\label{EqDefzTau}
\widetilde{z}_{\tau} = \widetilde{z}\circ \tau = \tau(\widetilde{z}).
\end{equation}
Let us write
\[
\rho = \frac{1}{2}\sum_{\alpha \in \Delta^+} \alpha,\qquad \rho_X = \frac{1}{2}\sum_{\alpha\in \Delta_X^+} \alpha,
\]
and let
\[
c_r = q^{\frac{1}{2}(\alpha_r,\Theta(\alpha_r)-2\rho_X)},\qquad r\in I,
\]
where we note that $c_r = 1$ for $r\in X$. 

Write 
\[
T_{w_X} = T_{r_1'}\ldots T_{r_{M}'} \in \mcU_q(\mfg)
\]
for the Lusztig braid operator associated to $w_X$. 

\begin{Lem}\label{CommTXT0}
The Lusztig braid operators $T_{w_0}$ and $T_{w_X}$ commute. 
\end{Lem}
\begin{proof}
This follows from Lemma \ref{LemCommTw0Tr}, Lemma \ref{LemCommTau0Tau} and the fact that $T_{w_X}$ is independent of the chosen reduced expression for $w_X$. 
\end{proof}

The following algebras were introduced by Letzter \cite{Let99}. We follow the conventions of \cite{Kol14,BK15b}. 

\begin{Def}\label{DefElBr}
We define $U_q'(\mfg_X)$ as the subalgebra of $U_q(\mfg)$ generated by the $K_r,E_r,F_r$ for $r\in X$. We define $U_q(\mfh^{\Theta})$ as the subalgebra of $U_q(\mfh)$ generated by the elements $K_{\omega}$ with $\Theta(\omega) = \omega$. We define $\msB \subseteq U_q(\mfg)$ as the subalgebra of $U_q(\mfg)$ generated by $U_q'(\mfg_X)$ and $U_q(\mfh^{\Theta})$ together with the elements
\[
B_r = F_r+ c_r X_r K_r^{-1},\qquad r\in I \setminus X,
\]
where
\[
X_r = - z_{\tau(r)} \Ad(T_{w_X})(E_{\tau(r)}).
\]
\end{Def}

One then has that $\msB$ is a right coideal subalgebra, 
\[
\Delta(\msB) \subseteq \msB \otimes U_q(\mfg). 
\]

\begin{Rem}\label{RemDiffConv}
In \cite{Kol14}*{Definition 4.3} a specific function $z = s(X,\tau)$ is used which does not satisfy our requirements as it takes values in $\{\pm1,\pm i\}$ in general. However, one can also use the current conventions for $z$ throughout the theory, see \cite{BK15b}*{Remark 5.2}. 
\end{Rem}

\begin{Rem}
In general $\msB$ depends more freely on the parameter $c$, and can have an additional parameter $s$. The specific choice for $c$ is needed to apply the results of \cite{BK15b}, while the choice $s=0$ simplifies some of the constructions.
\end{Rem}

\begin{Rem}\label{RemInvtautau0}
By Lemma \ref{LemCommTau0Tau} both $z$ and $c$ are $\tau\tau_0$-invariant, from which it follows that 
\[
\tau\tau_0(B_r) = B_{\tau\tau_0(r)}.
\]
In particular, $\msB$ is $\tau\tau_0$-invariant. 
\end{Rem}

In the following, we also write $c$ for the unique character
\[
c: P \rightarrow \R_{>0},\qquad c_{\alpha_r} = c_r.
\]
We then write $\gamma$ for the $\C^{\times}$-valued character on $P$ given by
\[
\gamma(\omega) = \widetilde{z}_{\tau}(\omega)c(\omega) ,\qquad \omega \in P.
\] 

We further write\footnote{We follow the notation in \cite{BK15b}, even though we also use $\xi$ for an arbitrary vector in a Hilbert space. This should however not lead to any confusion, as the r\^{o}les of the two different uses are quite different.} $\xi \in \mcU_q(\mfh)$ for the unique element such that for $\omega\in P$
\begin{equation}\label{EqXi}
\xi(\omega) = \gamma(\omega) q^{-(\omega^+,\omega^+)+ \sum_{r\in I} (\alpha_r^-,\alpha_r^-)(\omega,\varpi_r^{\vee})},
\end{equation}
where $\varpi_r^{\vee} = \frac{2\varpi_r}{(\alpha_r,\alpha_r)}$ are the fundamental coweights determined by $(\varpi_r^{\vee},\alpha_s) = \delta_{rs}$ and where
\begin{equation}\label{EqPM}
\omega^{\pm} = \frac{1}{2}(\omega \pm \Theta(\omega)).
\end{equation}
Finally, let $\quasiK \in  \mcU_q(\mfn^+)$ be the quasi-$K$-matrix introduced in \cite{BK15b}*{Theorem 6.10}, where we note that the construction performed there can be repeated ad verbatim for $q$ a positive scalar distinct from 1.

\begin{Theorem}[\cite{BK15b}*{Corollary 7.7 and Theorem 9.5}]
The element
\begin{equation}\label{EqDefOrK}
\mcK= \quasiK \xi T_{w_X}^{-1}  T_{w_0}^{-1} \in \mcU_q(\mfg)
\end{equation}
satisfies 
\begin{equation}\label{EqOrigK}
\Delta(\mcK) = (\mcK\otimes 1) \msR_{\tau\tau_0,21}(1\otimes \mcK)\msR = \msR_{21}(1\otimes \mcK)\msR_{\tau\tau_0}(\mcK\otimes1). 
\end{equation}
Moreover, for all  $X\in \msB$
\begin{equation}\label{EqCommOrK}
\mcK X = \tau\tau_0(X)\mcK.
\end{equation}
\end{Theorem}
\begin{Rem}
Note that when comparing conventions, the element $\hat{R}$ in \cite{BK15b}*{Theorem 9.5} coincides with our $\Sigma \circ\msR$, where $\Sigma$ is the flip map.
\end{Rem}

In the following, we will modify $\mcK$ so that it becomes a $*$-compatible $\nu$-modified universal $K$-matrix for a particular $\nu$. We will need some preliminary results concerning the behaviour of $\quasiK$ under the $*$-operation.  

First, we reinterpret the characterisation of $\quasiK$ in \cite{BK15b}*{Theorem 6.10}, so that we can use it for $q$ a concrete value as in our setting. We make first the following definition. 

\begin{Def}
For $r\in I \setminus X$ we define
\begin{equation}\label{EqBarB}
\overline{B_r} = F_r- (-1)^{(2\rho_X^\vee, \alpha_{\tau(r)})} q^{- (2\rho_X, \alpha_{\tau(r)})} c_r^{-1} z_{\tau(r)} \Ad(T_{w_X}^{-1}) (E_{\tau(r)}) K_r \in U_q(\mfg).
\end{equation}
\end{Def}

\begin{Prop}\label{PropUniqueK}
Let $\quasiK' = \sum_{\alpha \in Q^+} \quasiK_{\alpha}' \in \mcU_q(\mfn)$ with $\quasiK_{\alpha}' \in U_q(\mfn)_{\alpha}$, and assume that $\quasiK_{0}' =1$ and 
\[ 
B_r \quasiK' = \quasiK' \overline{B}_r,\qquad F_s \quasiK' = \quasiK' F_s,\qquad r\in I \setminus X,s\in X.
\]
Then $\quasiK' = \quasiK$.
\end{Prop} 
\begin{proof}
The uniqueness of $\quasiK$ will follow from \cite{BK15b}*{Proposition 6.3}, which also works for $q$ non-formal, once we have shown that for $q$ formal our value of $\overline{B_r}$ coincides with the bar involution applied to $B_r$. Here we note that in order to be able to specialize at some positive value distinct from 1, we can work over the smallest $\C$-algebra inside the algebraic closure of $\C(q)$ containing all formal roots of $q$ and $q^{-1}$, and containing $(q^n -q^{-n})^{-1}$ for all $n\in \Q$. Note that for $q$ formal, our definition of $B_r$ treats  $z_r$ as a complex scalar, but the $q$ in $c_r$ as a formal parameter.

Let now $\overline{\phantom{B}}$ be the bar involution of $U_q(\mfg)$ (for $q$ formal as above), determined as the algebra automorphism satisfying
\[
\overline{q^n} = q^{-n},\qquad \overline{E_r} = E_r,\qquad \overline{F_r} = F_r,\qquad \overline{K_{\omega}} = K_{\omega}^{-1}.
\]
Let us momentarily write the element in \eqref{EqBarB} as $\overline{B_r}'$, and use $\overline{B_r}$ for the value of the bar involution applied to $B_r$. First of all we have
\[
\overline{B_r} = F_r + \overline{c_r} \overline{X_r} K_r.
\]
Since $z_{\tau(r)} \in \mathbb{C}$, we have $\overline{z_{\tau(r)}} = z_{\tau(r)}$. On the other hand, we have in the notation of \cite{Lusz94}*{37.1} that $\Ad(T_r) = T_{r,1}^{\prime \prime}$ and $\Ad(T_r^{-1}) = T_{r,-1}^{\prime}$, with 
\[
\overline{T_{r,\pm 1}^{\prime \prime}(X)} = T_{r,\mp 1}^{\prime \prime}(\overline{X}), \quad
\overline{T_{r,\pm 1}^{\prime}(X)} = T_{r,\mp 1}^{\prime}(\overline{X}),\qquad X\in U_q(\mfg).
\]
Hence we have
\[
\overline{X_r} = - z_{\tau(r)} \overline{\Ad(T_{w_X})(E_{\tau(r)})} = - z_{\tau(r)} T_{w_X,-1}^{\prime \prime} (E_{\tau(r)}).
\]
In the proof of \cite{BK15a}*{Lemma 2.9} it is shown that 
\[
T_{w_X,-1}^{\prime \prime} (E_r) = (-1)^{(2\rho_X^\vee, \alpha_r)} q^{-(2\rho_X, \alpha_r)} T_{w_X,-1}^{\prime} (E_r).
\]
Using this we get
\[
\overline{X_r} = - (-1)^{(2\rho_X^\vee, \alpha_{\tau(r)})} q^{- (2\rho_X, \alpha_{\tau(r)})} z_{\tau(r)} \Ad(T_{w_X}^{-1}) (E_{\tau(r)}).
\]
Plugging this into $\overline{B_r}$ and using $\overline{c_r} = c_r^{-1}$ shows $\overline{B_r}' = \overline{B_r}$. 
\end{proof}

\begin{Cor}\label{CorInvTauTau_0}
The quasi-$K$-matrix $\quasiK$ satisfies $\tau\tau_0(\quasiK) = \quasiK$. Similarly, $\tau\tau_0(\mcK) = \mcK$.
\end{Cor}
\begin{proof}
As in Remark \ref{RemInvtautau0}, we have $\tau\tau_0(\overline{B_r}) = \overline{B}_{\tau\tau_0(r)}$ for $r\in I \setminus X$. The equality $\tau\tau_0(\quasiK) = \quasiK$ then follows immediately from the uniqueness in Proposition \ref{PropUniqueK}. The $\tau\tau_0$-invariance of $\mcK$ then follows immediately from its definition, Lemma \ref{LemCommTau0Tau} and the $\tau\tau_0$-invariance of $c$ and $\widetilde{z}$.
\end{proof}

Our first goal now will be to show the identities 
\[
\quasiK^* = \Ad(T_{w_0})(\tau_0(\quasiK)),\qquad \Ad(\widetilde{z})(\quasiK) = \tau(\quasiK),
\] 
see Theorem \ref{TheoStarX}. We use here the notation $\Ad(t)X = tXt^{-1}$ for $t$ a grouplike in $\mcU_q(\mfh)$ and $X\in U_q(\mfg)$. 

Let $\mcS_0$ be as in \eqref{EqDefS}, and let similarly 
\begin{equation}\label{EqDefSX}
\mcS_X = e^{2\pi i \rho_X^{\vee}} \in T.
\end{equation}
Let us further denote
\[
\mcS = e^{\pi i \rho^{\vee}}\in T,
\] 
so that $\mcS^2 = \mcS_0$. Note that since $(\rho^{\vee},\alpha_r) = 1$ for all $r\in I$, we have
\[
\Ad(\mcS)(E_r) = -E_r,\qquad \Ad(\mcS)(F_r) = -F_r,\qquad \Ad(\mcS)(K_{\omega}) = K_{\omega}. 
\]

\begin{Lem}\label{CommSTX}
The identity $\mcS T_{w_X}^{-1} = T_{w_X}^{-1}\mcS\mcS_X^{-1}$ holds.
\end{Lem}
\begin{proof}
We have
\[
\mcS T_r^{-1} \xi = e^{\pi i (\rho^{\vee},\wt(\xi) - (\wt(\xi),\alpha_r^{\vee})\alpha_r)}T_r^{-1}\xi = e^{-\pi i (\wt(\xi),\alpha_r^{\vee})}T_r^{-1}\mcS\xi. 
\]
Hence 
\[
\mcS T_{w_X}^{-1}\xi = e^{-2\pi i (\wt(\xi),\rho_X^{\vee})}T_{w_X}^{-1}\mcS \xi = T_{w_X}^{-1} \mcS \mcS_X^{-1}\xi. \qedhere
\]
\end{proof}

Let further $\sigma: U_q(\mfg) \rightarrow U_q(\mfg)$ be the unique algebra anti-isomorphism such that 
\[
\sigma(E_r) = E_r,\qquad \sigma(F_r) = F_r,\qquad \sigma(K_{\omega}) = K_{\omega}^{-1}.
\]
Note that by \cite{Jan96}*{8.(10)}, one has
\[
\sigma(\Ad(T_r)X) = \Ad(T_r^{-1})(\sigma(X)),\qquad X\in U_q(\mfg).
\]

\begin{Lem}
\label{lem:s-sigma}\label{LemActztauB}
For $r \in I \backslash X$ we have
\begin{enumerate}
\item\label{LemActztauB1} $\Ad(\mcS)(\sigma (B_r)) = - \overline{B_r}$,
\item\label{LemActztauB2} $\Ad(\widetilde{z})(\tau(B_r)) = z_{\tau(r)}B_{\tau(r)}$,
\item\label{LemActztauB3} $\Ad(\widetilde{z})(\tau(\overline{B_r})) = z_{\tau(r)}\overline{B_{\tau(r)}}$.
\end{enumerate}
\end{Lem}

\begin{proof}
Let us first prove \eqref{LemActztauB1}. We have
\[
\Ad(\mcS)(\sigma (B_r)) = - F_r + c_r K_r \Ad(\mcS)(\sigma (X_r)).
\]
Since $X_r \in U_q(\mfg)_{-\Theta(\alpha_r)}$, the same is true for $ \Ad(\mcS)(\sigma (X_r))$. Therefore we can write
\[
\Ad(\mcS)(\sigma (B_r)) = - F_r +q^{-(\alpha_r, \Theta(\alpha_r))} c_r \Ad(\mcS)(\sigma (X_r))K_r.
\]
Now from the identity $\sigma \circ \Ad(T_r) = \Ad(T_r^{-1}) \circ \sigma$ we obtain $\sigma \circ \Ad(T_{w_X}) = \Ad(T_{w_X}^{-1}) \circ \sigma$. Since  by Lemma \ref{CommSTX} we have 
\[
\Ad(\mcS) \Ad(T_{w_X}^{-1}) (E_{\tau(r)}) = - (-1)^{(2\rho_X^\vee, \alpha_{\tau(r)})} \Ad(T_{w_X}^{-1}) (E_{\tau(r)}),
\] 
we get
\[
\Ad(\mcS)(\sigma (X_r))= (-1)^{(2\rho_X^\vee, \alpha_{\tau(r)})} z_{\tau(r)} \Ad(T_{w_X}^{-1}) (E_{\tau(r)}).
\]
Plugging this into $\Ad(\mcS)(\sigma (B_r))$ and observing that 
\[
q^{-(\alpha_r, \Theta(\alpha_r))} c_r = q^{- \frac{1}{2} (\alpha_r , \Theta(\alpha_r) + 2\rho_X)} = q^{- (2\rho_X, \alpha_{\tau(r)})} c_r^{-1}
\]
gives the result.

Let us now prove \eqref{LemActztauB2}. Using that $\tau(X_r) \in U_q(\mfn)_{w_X\alpha_r}$, $z_r^2 =1$, $\widetilde{z}(w_X(\alpha_r)) = z_r$ and $c_{\tau(r)} = c_r$, we find 
\[
\Ad(\widetilde{z})(\tau(B_r)) = z_{\tau(r)}F_{\tau(r)} + \widetilde{z}(w_X\alpha_r) z_rz_{\tau(r)}c_r X_{\tau(r)}K_{\tau(r)}^{-1}  =  z_{\tau(r)}B_{\tau(r)}.
\]
The proof of \eqref{LemActztauB3} follows similarly. 
\end{proof}

\begin{Prop}
\label{prop:invariance-ssigma}
We have $\signaut(\sigma(\quasiK))  = \quasiK$.
\end{Prop}

\begin{proof}
We can write
\[
\signaut(\sigma(\quasiK)) = \sum_{\alpha \in Q^+} \signaut(\sigma(\quasiK_\alpha))
\] 
with $\signaut(\sigma(\quasiK_0)) = 1$ and $\signaut(\sigma(\quasiK_\alpha)) \in U_q(\mfn)_\alpha$.

Now since $\quasiK$ commutes with $F_r$ for $r\in X$, we also have that $\signaut(\sigma(\quasiK))$ commutes with $F_r$. On the other hand, applying $\signaut \circ \sigma$ to $B_r \quasiK = \quasiK \overline{B_r}$ with $r \in I \backslash X$, we get
\[
\signaut(\sigma(\quasiK)) \signaut(\sigma(B_r)) = \signaut(\sigma(\overline{B_r})) \signaut(\sigma(\quasiK)).
\]
But we have $\signaut (\sigma (B_r)) = - \overline{B_r}$ from Lemma \ref{lem:s-sigma}, which is also equivalent to $B_r= - \signaut (\sigma (\overline{B_r}))$. Using these we get 
\[
\signaut(\sigma(\quasiK)) \overline{B_r} = B_r \signaut(\sigma(\quasiK)).
\]
By Proposition \ref{PropUniqueK} we conclude $\signaut(\sigma(\quasiK)) = \quasiK$.
\end{proof}

Let now 
\[
\mathrm{cc}: U_q(\mfg) \rightarrow U_q(\mfg)
\]
be the unique anti-linear algebra homomorphism which is the identity on $E_r,F_r,K_{\omega}$. Then we have the identities
\begin{equation}\label{EqAdTw0star}
\Ad(T_{w_0}) = * \circ \mathrm{cc} \circ \Ad(\mcS) \circ \sigma \circ \tau_0, \quad
\Ad(T_{w_0}^{-1}) = \tau_0 \circ \sigma \circ \Ad(\mcS) \circ \mathrm{cc} \circ *,
\end{equation}
where the first identity  follows from \eqref{EqIdAdT_0} and where the second expression follows from the first, since all the various maps are involutions. Note further that the maps $\Ad(\mcS)$, $\tau_0$, $\sigma$ and $\mathrm{cc}$ all commute among themselves.

\begin{Theorem}\label{TheoStarX}
We have $\quasiK^* = \Ad(T_{w_0})(\tau_0(\quasiK))$. 
\end{Theorem}
\begin{proof}
We clearly have $\mathrm{cc}(\quasiK) = \quasiK$ by the defining property in Proposition \ref{PropUniqueK}, using that $z$ is real-valued and hence the $B_r$ and $\overline{B_r}$ are invariant under $\mathrm{cc}$. The result then follows from \eqref{EqAdTw0star} and Proposition \ref{prop:invariance-ssigma}.
\end{proof}

We now perform various small changes to $\msB$, first to make it $*$-invariant and then to make it rather a left coideal. Define
\begin{equation}\label{EqDefome0}
\omega_0  =-\frac{1}{2}(\rho-\rho_X).
\end{equation}
Using the obvious notation $K_{\omega_0} \in \mcU_q(\mfh)$, let us write 
\[
\widetilde{\msB} = \Ad(K_{\omega_0})(\msB),\qquad \widetilde{\quasiK} = \Ad(K_{\omega_0})(\quasiK),\qquad  \widetilde{\mcK} = \Ad(K_{\omega_0})(\mcK).
\] 
Clearly $\widetilde{\msB}$ is again a right coideal subalgebra. Note that $\widetilde{\msB}$ is generated by $U_q'(\mfg_X)$, $U_q(\mfh^{\Theta})$ and the elements
\[
\widetilde{B}_r = F_r+ q^{(\omega_0,\alpha_r -  \Theta(\alpha_r))}c_r X_rK_r^{-1},\qquad r\in I \setminus X.
\] 
Since $w_X \rho = \rho-2\rho_X$, and hence $w_X\omega_0 = \omega_0$, we can simplify this expression as
\[
\widetilde{B}_r = F_r+ q^{2(\omega_0,\alpha_r)}c_r X_rK_r^{-1} =  F_r+ q^{-(\alpha_r^-,\alpha_r^-)} X_rK_r^{-1},
\]
using once more the notation \eqref{EqPM} and the identity $(2\rho,\alpha_r) = (\alpha_r,\alpha_r)$. 
\begin{Lem}
The algebra $\widetilde{\msB}$ is $*$-invariant. 
\end{Lem} 
\begin{proof}
By \cite{DCNTY17}*{Theorem 3.11}, we need to check that 
\begin{equation}
q^{2(\omega_0,\alpha_r + \alpha_{\tau(r)})} c_r c_{\tau(r)} =  q^{(\Theta(\alpha_r) - \alpha_r,\alpha_{\tau(r)})}. 
\end{equation}
Note that this theorem was proven under a different assumption on $z$ mentioned in Remark \ref{RemDiffConv}, but it is easily verified that for the $*$-invariance of $\widetilde{B}$ the only feature of $z$ which was used was \eqref{EqProdzz} and the fact that $z_r = 1$ for $\alpha_r\perp X$, which is still valid in the current setup.  

Following the discussion under \cite{DCNTY17}*{Theorem 3.14}, it is sufficient to show that  $(\omega_0,\alpha_r) = 0$ for $r\in X$, and 
\begin{equation}\label{EqVerifDiffEasy}
(\omega_0,\alpha_r) = \frac{1}{4}(\Theta(\alpha_{\tau(r)}) - \alpha_{\tau(r)} -\Theta(\alpha_r) +2\rho_X,\alpha_r),\quad r\in I \setminus X.
\end{equation}
Now since $(\rho_X,\alpha_r) = (\rho,\alpha_r)$ for $r\in X$, we obtain $(\omega_0,\alpha_r)=0$ for $r\in X$. On the other hand, if $r\in I \setminus X$, we have by \cite{BK15a}*{Lemma 3.2} that
\[
\Theta(\alpha_{\tau(r)}) - \alpha_{\tau(r)} -\Theta(\alpha_r) = -\alpha_r. 
\] 
Since $2(\rho,\alpha_r) = (\alpha_r,\alpha_r)$ for all $r\in I$, it follows that \eqref{EqVerifDiffEasy} holds. 
\end{proof}

\begin{Rem}
Remark that $*$-invariance of Letzter coideals was also discussed in \cite[Proposition 4.6]{BW16}, and in a less concrete manner in \cite [discussion before Theorem 7.6]{Let02}.
\end{Rem}

We will now proceed to show that also $\widetilde{\mcK}$ has a nice behaviour with respect to $*$, see Proposition \ref{PropStarKMain}. We again need some preliminaries.

Let
\[
\xi' = \xi K_{2\omega_0} = K_{2\omega_0}\xi,
\]
with $\xi$ as in \eqref{EqXi}. Using that
\[
K_{\omega_0} T_{w_0} = T_{w_0}K_{\omega_0}^{-1},\qquad K_{\omega_0}T_{w_X} = T_{w_X}K_{\omega_0},
\] 
we find
\[
\widetilde{\mcK} = \widetilde{\quasiK} \xi' T_{w_X}^{-1}T_{w_0}^{-1}.
\]

Let $C_{\Theta}\in \mcU_q(\mfh)$ be the unique operator such that 
\[
C_{\Theta}\eta = q^{-(\wt(\xi)^+,\wt(\xi)^+)}\eta,\qquad \eta\in V_{\varpi},
\]
with $\omega^{\pm}$ defined as in \eqref{EqPM}.
 
\begin{Lem}
We have
\[
\xi' = \widetilde{z}_{\tau} C_{\Theta}. 
\]
\end{Lem}
\begin{proof}
We can write
\[
c_r = q^{-2(\omega_0,\alpha_r)}q^{-(\alpha_r^-,\alpha_r)},\qquad r\in I,
\]
and hence for $\omega = \sum_r k_r \alpha_r \in P$ with $k_r \in \Q$ we have 
\[
\gamma(\omega) = \widetilde{z}_{\tau}(\omega) \prod_{r\in I} c_r^{k_r} = \widetilde{z}_{\tau}(\omega) q^{-(2\omega_0,\omega)}\prod_{r\in I}  q^{-\sum_{r\in I\setminus X} k_r(\alpha_r^-,\alpha_r)}.
\]
On the other hand, since $(\alpha_r^-,\alpha_r^+) = 0$ by $\Theta$-invariance of $(-,-)$, we have
\[
\sum_{r\in I} (\alpha_r^-,\alpha_r^-)(\omega,\varpi_r^{\vee}) = \sum_{r\in I} k_r(\alpha_r^-,\alpha_r^-) = \sum_{r\in I} k_r (\alpha_r^-,\alpha_r),
\]
and hence
\[
\xi(\omega) =  \widetilde{z}_{\tau}(\omega) q^{-(2\omega_0,\omega)}\prod_{r\in I \setminus X}  q^{-\sum_{r\in I\setminus X} k_r(\alpha_r^-,\alpha_r^--\alpha_r)}q^{-(\omega^+,\omega^+)},
\]
from which the lemma follows.
\end{proof} 

\begin{Lem}\label{LemComCX}
The element $C_{\Theta}$ is invariant under $\tau$ and $\tau_0$, and commutes with $\quasiK$ and $K_{\omega_0}$.
\end{Lem}
\begin{proof}
The invariance of $C_{\Theta}$ under $\tau$ and $\tau_0$ follows immediately from Lemma \ref{LemCommThetaTau}. It is also immediate that $C_{\Theta}$ commutes with $K_{\omega_0}$. Finally, write again $\quasiK = \sum_{\alpha\in Q^+}\quasiK_{\alpha}$ with $\quasiK_{\alpha}\in U_q(\mfn)_{\alpha}$. Let $V$ be a representation of $U_q(\mfg)$, and $\xi\in V$. Then $C_\Theta \quasiK_{\alpha}C_{\Theta}^{-1}\xi = \frac{C_{\Theta}(\wt(\xi) + \alpha)}{C_{\Theta}(\wt(\xi))}\quasiK_{\alpha} \xi$. From \cite{BK15b}*{Proposition 6.1}, we know that $\quasiK_{\alpha} \neq0$ implies $\Theta(\alpha) = -\alpha$. As the latter implies in turn that $C_{\Theta}(\wt(\xi)+\alpha) = C_{\Theta}(\wt(\xi))$, the commutation of $\quasiK$ and $C_{\Theta}$ follows. 
\end{proof}

\begin{Lem}\label{LemComTw0z}\label{LemComXiTw0}
We have
\begin{equation}\label{EqAdzw0}
\Ad(T_{w_0})(\widetilde{z}) = \widetilde{z}_{\tau}^{-1},\qquad \Ad(T_{w_X})(\widetilde{z}) = \widetilde{z}
\end{equation}
and
\begin{equation}\label{EqAdxiw0}
\Ad(T_{w_0})(\xi') = \widetilde{z}^{-1} \widetilde{z}_{\tau}^{-1} \xi'.
\end{equation}
\end{Lem} 
\begin{proof}
As $z_r = 1$ for $r\in X$, it follows that $\widetilde{z}$ is $W_X$-invariant. The identities \eqref{EqAdzw0} then follow from the assumption that $\tau(\widetilde{z}) = \tau_0(\widetilde{z})$.

The identity \eqref{EqAdxiw0} follows from \eqref{EqAdzw0} by the computation  
\[
C_{\Theta}(w_0\omega) = C_{\Theta}(-w_0\omega) = C_{\Theta}(\tau_0(\omega)) = C_{\Theta}(\omega).
\]
\end{proof}

\begin{Lem}\label{LemCommQuasiKxi}
We have 
\begin{equation}\label{EqAdzX}
\Ad(\widetilde{z})(\quasiK) = \tau(\quasiK)
\end{equation}
and
\begin{equation}\label{EqAdXXi}
\quasiK \xi' = \xi'  \tau(\quasiK).
\end{equation}
\end{Lem}
\begin{proof}
For \eqref{EqAdzX}, it is by Proposition \ref{PropUniqueK} sufficient to show that 
\[
B_r \tau(\Ad(\widetilde{z})(\quasiK)) = \tau(\Ad(\widetilde{z})(\quasiK)) \overline{B}_r
, \qquad F_s \tau(\Ad(\widetilde{z})(\quasiK)) = \tau(\Ad(\widetilde{z})(\quasiK))F_s, \qquad r \in I \setminus X, \ s \in X.
\]
By Lemma \ref{LemActztauB} this is equivalent with the defining property of $\quasiK$.

For \eqref{EqAdXXi} we note that 
\[
\quasiK \xi' = \widetilde{z}_{\tau} \Ad(\widetilde{z}_{\tau}^{-1})(\quasiK) C_{\Theta} = \widetilde{z}_{\tau} \tau(\Ad(\widetilde{z}^{-1})(\tau(\quasiK)) C_{\Theta}. 
\]
From \eqref{EqAdzX} and Lemma \ref{LemComCX}, it then follows that
\[
\quasiK \xi' =  \widetilde{z}_{\tau} \tau(\quasiK) C_{\Theta} = \widetilde{z}_{\tau}C_{\Theta} \tau(\quasiK)  = \xi' \tau(\quasiK).
\]
\end{proof}

\begin{Prop}\label{PropStarKMain}
We have
\begin{equation}\label{EqKTildeBraid}
\Delta(\widetilde{\mcK}) = (\widetilde{\mcK}\otimes 1) \msR_{\tau\tau_0,21}(1\otimes \widetilde{\mcK})\msR = \msR_{21}(1\otimes \widetilde{\mcK})\msR_{\tau\tau_0}(\widetilde{\mcK}\otimes1) 
\end{equation}
and for all  $X\in \widetilde{\msB}$
\begin{equation}\label{ECommTildeK}
\widetilde{\mcK} X = \tau\tau_0(X)\widetilde{\mcK}.
\end{equation}
Moreover,
\begin{equation}\label{EqKStarFund} 
\tau\tau_0(\widetilde{\mcK}) = \widetilde{\mcK},\qquad \widetilde{\mcK}^* = \widetilde{\mcK}\mathcal{S}_X\mathcal{S}_0 \widetilde{z} \widetilde{z}_{\tau}^{-1}.
\end{equation}
\end{Prop}
\begin{proof}
Since $K_{\omega_0}$ is a $\tau\tau_0$-invariant grouplike element, the element $\widetilde{\mcK}$ satisfies \eqref{EqOrigK}, and \eqref{EqCommOrK} with respect to $\widetilde{\msB}$.

As $\xi',T_{w_0}$ and $T_{w_X}$ are $\tau\tau_0$-invariant, and as $\widetilde{\quasiK}$ is $\tau\tau_0$-invariant by Corollary \ref{CorInvTauTau_0} and $\tau\tau_0$-invariance of $\omega_0$, the identity $\tau\tau_0(\widetilde{\mcK}) = \widetilde{\mcK}$ follows. 

Note now that $T_{w_X}$ commutes with $\xi'$, and with $\quasiK$ by \eqref{EqCommOrK}. Using also Lemma \ref{LemCommQuasiKxi} and the fact that $K_{\omega_0}$ is $\tau$-invariant, we see that we can write
\[
\widetilde{\mcK} = T_{w_X}^{-1}\xi' \tau(\widetilde{\quasiK})T_{w_0}^{-1}.
\]
Since $\Ad(T_{w_0})$, $\tau$ and $\tau_0$ commute, since $\tau_0(\omega_0) = \omega_0$ and since $K_{\omega_0}T_{w_0}K_{\omega_0} = T_{w_0}$, we find from Theorem \ref{TheoStarX} and $\tau\tau_0$-invariance of $\widetilde{\quasiK}$ that
\[
\widetilde{\mcK} = T_{w_X}^{-1}\xi' T_{w_0}^{-1}\widetilde{\quasiK}^*,
\]
and thus from Lemma \ref{LemComXiTw0}
\[
\widetilde{\mcK} = T_{w_X}^{-1}T_{w_0}^{-1} \widetilde{z}^{-1}\widetilde{z}_{\tau}^{-1}\xi' \widetilde{\quasiK}^* =  \widetilde{z}\widetilde{z}_{\tau} T_{w_X}^{-1}T_{w_0}^{-1}\xi' \widetilde{\quasiK}^*.
\]
Using now Proposition \ref{PropStarT}, and the fact that $\mcS_X, \mcS_0$ assume values in $\{\pm 1\}$, we see that
\[
\widetilde{\mcK}^* =  \widetilde{\quasiK}\xi'^*T_{w_0}^{-1} \mcS_0T_{w_X}^{-1} \mcS_X\widetilde{z}^{-1}\widetilde{z}_{\tau}^{-1} = \widetilde{\quasiK}\xi'^*T_{w_0}^{-1} T_{w_X}^{-1}\mcS_0 \mcS_X\widetilde{z}^{-1}\widetilde{z}_{\tau}^{-1},
\]
where in the last step we used that $w_X\rho^{\vee} = \rho^{\vee} - 2\rho_X^{\vee}$ and $e^{4\pi i \rho_X^{\vee}} = 1$. As $\xi'^* = \xi' \widetilde{z}_{\tau}^{-2}$, and as $T_{w_X}$ and $T_{w_0}$ commute by Lemma \ref{CommTXT0}, this becomes the second identity in \eqref{EqKStarFund} by another application of \eqref{EqAdzw0}.
\end{proof}

Let us now move from right coideals to left coideals to have compatibility with the conventions in Section \ref{SecTwistBraid}. This can be achieved by means of the unitary antipode $R$ defined in \eqref{EqUnitaryAntipode}. We then write 
\begin{equation}\label{EqCoidSwitch}
U_q(\mfu^{\theta}) = R(\widetilde{\msB}),
\end{equation}
which is a left coideal $*$-subalgebra of $U_q(\mfu)$. Note that by \eqref{EqForm}, we have that $U_q(\mfu^{\theta})$ is generated by the $U_q'(\mfg_X)$, $U_q(\mfh^{\Theta})$ and the elements
\begin{equation}\label{DefNewGenCoid}
C_r = -q_r R(\widetilde{B}_r)^* =  E_r + q^{(\alpha_r^+,\alpha_r^+)}Y_rK_r,\qquad Y_r = -z_{\tau(r)} \Ad(T_{w_X})(F_{\tau(r)})
\end{equation}
for $r\in I \setminus X$. 

Let further $v\in \mcU_q(\mfu)$ be the ribbon element 
\begin{equation}\label{EqDefRibbon}
v\eta = q^{-(\varpi,\varpi + 2\rho)}\eta,\qquad \eta \in V_{\varpi},
\end{equation}
so that $v$ is central, self-adjoint and
\begin{equation}\label{EqPropRibbon}
\msR_{21}\msR = \Delta(v)(v^{-1}\otimes v^{-1}).
\end{equation}
Put 
\[
\widetilde{\msK} = R(\widetilde{\mcK})v^{-1}. 
\]
Then since $(R\otimes R)\msR = \msR$ and $R$ commutes with $\tau\tau_0$, we find by Proposition \ref{PropStarKMain} that
\begin{equation}\label{EqTildeMSK}
\Delta(\widetilde{\msK}) = (1\otimes \widetilde{\msK})\msR_{\tau\tau_0,21} (\widetilde{\msK}\otimes 1)\msR_{21}^{-1},\qquad \widetilde{\msK}^* = \mathcal{S}_X\mathcal{S}_0 \widetilde{z}_{\tau} \widetilde{z}^{-1}\tau\tau_0(\widetilde{\msK}).
\end{equation}
For a categorical motivation of passing between these different kinds of conditions for $K$-matrices, we refer to \cite{BZBJ16}. See also \cite{tDH-O98,tD98,Wee18} for further discussion on the categorical origin of $K$-matrices. 

Let now $\tau_{\nu} =\tau\tau_0$, and let $\epsilon = \epsilon_{\nu}$ be an $(X,\tau)$-admissible sign function on $I$ as in Definition \ref{DefTauAdmi}. Put $\nu = (\tau_{\nu},\epsilon_{\nu})\in \End_*(U_q(\mfb))$. The crucial property we will need for $\epsilon$ is that by Theorem \ref{TheoremSignEps} we can find an extension $\widetilde{\epsilon}\in T$ of $\epsilon$ such that 
\begin{equation}\label{EqCrucialProp}
\widetilde{\epsilon}\tau\tau_0(\widetilde{\epsilon}) = \mcS_0\mcS_X\widetilde{z}\widetilde{z}_{\tau}^{-1}. 
\end{equation}
Define $\msE$ as in Definition \ref{DefmsE}, and put
\begin{equation}\label{EqModKDef}
\msK = \msE \widetilde{\epsilon}^{-1}\widetilde{\msK} \in \mcU_q(\mfu).
\end{equation}

\begin{Theorem}\label{TheoSymmK}
The element $\msK$ is a $*$-compatible $\nu$-modified universal $K$-matrix.
\end{Theorem}
\begin{proof}
Since $\widetilde{\epsilon}$ is grouplike, it follows from \eqref{EqTildeMSK} that $\msE^{-1}\msK$ satisfies \eqref{EqDeltaKEInvAlt}, hence $\msK$ is a $\nu$-modified universal $K$-matrix. To see that it is $*$-compatible, note that by selfadjointness of $\msE$ we have $(\msE \widetilde{\epsilon}^{-1})^* = \msE\widetilde{\epsilon}$. Now $\msE\widetilde{\epsilon}$ is central, with 
\[
\msE\widetilde{\epsilon} \xi = \widetilde{\epsilon}_{\varpi}\xi,\qquad \xi \in V_{\varpi}. 
\]
Hence by \eqref{EqTildeMSK} and the defining property of $\widetilde{\epsilon}$, we find that
\[
\msK^* = \widetilde{\epsilon}^{-1}\tau\tau_0(\widetilde{\epsilon})^{-1} \tau\tau_0(\widetilde{\msK})\msE \widetilde{\epsilon} = 
\tau\tau_0(\widetilde{\epsilon})^{-1}\msE \tau\tau_0(\widetilde{\msK}) = \tau\tau_0(\msK),
\]
since $\epsilon$ and hence $\msE$ is $\tau\tau_0$-invariant. This proves $*$-compatibility.
\end{proof}

Note also that $\msK$ satisfies the following commutation relation, using \eqref{ECommTildeK} and the fact that $\msE\widetilde{\epsilon}^{-1}$ is central,
\begin{equation}\label{EqCommOrKNew}
 \msK X= \tau\tau_0(X)\msK,\qquad X\in U_q(\mfu^{\theta}). 
\end{equation}

\begin{Rem}
The only property of the $\tau\tau_0$-invariant sign function $\epsilon$ which is needed in the above construction is the existence of an extension $\widetilde{\epsilon}$ satisfying \eqref{EqCrucialProp}. This property is in general much weaker than being $(X,\tau)$-admissible (for example one could have $\epsilon = 1$). However, we believe that only in case of $(X,\tau)$-admissible $\epsilon$ will the associated $K$-matrix lead to a $*$-homomorphism $\phi:\mcO_q(Z_{\nu}) \rightarrow \mcO_q(U)$ with sufficiently nice spectral properties, cf.~ Remark \ref{RemSpectral}. Again, we will not deal here with this subtle phenomenon, which deserves further investigation. 
\end{Rem}

\subsection{Comparison of the coideal subalgebras $U_q^{\fin}(\mfk')$ and $U_q(\mfu^{\theta})$}

In this section, we clarify the connection between the left coideal $*$-subalgebra $U_q^{\fin}(\mfk') \subseteq U_q(\mfu)$ as constructed from the $*$-compatible $\nu$-modified universal $K$-matrix of \eqref{EqModKDef} by the map $\hat{\phi}$ in Theorem \ref{TheoOneToOneCorr}, and the left coideal $*$-subalgebra $U_q(\mfu^{\theta}) \subseteq U_q(\mfu)$ as defined by \eqref{EqCoidSwitch}.

We introduce first the following map
\[
\hat{\Phi}: \mcO_q(U) \rightarrow U_q(\mfg),\quad f\mapsto (\id\otimes f)(\msR_{\tau\tau_0,21}(1\otimes \mcK)\msR),
\]
where $\mcK$ was introduced in \eqref{EqDefOrK}. Recall further the unitary antipode $R$ defined in \eqref{EqUnitaryAntipode}and the element $\omega_0$ introduced in \eqref{EqDefome0}. 

\begin{Lem}\label{LemEqualitySwitch}
The equality $\hat{\phi}(\mcO_q(Z_{\nu})) = (R \circ \Ad(K_{\omega_0}) \circ \hat{\Phi}) (\mcO_q(U))$ holds. 
\end{Lem}
\begin{proof}
Recall from \eqref{EqHatPhi} that 
\[
\hat{\phi}(f) = (f\otimes \id)(\msR_{\tau\tau_0,21}(\msK \otimes 1)\msR_{\tau\tau_0}) = (\id\otimes f\circ \tau\tau_0)(\msR(1\otimes \msK)\msR_{\tau\tau_0,21}),
\]
where in the last step we used that $\mcK$ is $\tau\tau_0$-invariant by Corollary \ref{CorInvTauTau_0}. On the other hand, since $K_{\omega_0}$ is $\tau\tau_0$-invariant, we find that 
\[
R(\Ad(K_{\omega_0})(f)) = (\id\otimes f\circ \Ad(K_{\omega_0}^{-1})\circ R)(\msR(1\otimes R(\widetilde{\mcK}))\msR_{\tau\tau_0,21}).
\]
Since $R(\widetilde{\mcK})$ and $\msK$ differ only by multiplication with an invertible central element, this proves the lemma.
\end{proof}

\begin{Cor}\label{CorIncUSym}
We have $U_q^{\fin}(\mfk')\subseteq \mcU_q(\mfu^{\theta})$.
\end{Cor}
\begin{proof}
By \cite{Kol17}*{Theorem 3.11.(0)}, the image of $\hat{\Phi}$ is contained in $\msB$ (note that the element $\msR$ in \cite{Kol17} indeed coincides with our element $\msR$). The corollary then follows from Lemma \eqref{LemEqualitySwitch}.
\end{proof}

To obtain an inclusion in the opposite direction after completion, we need some preliminaries. Let us write 
\[
\widetilde{\msR} = \sum_A w_A E^A \otimes F^A
\]
where $E^A$ and $F^A$ denote the standard PBW-bases, $A$ being words in the positive roots and $w_A$ are non-zero scalars. We write $\wt(A) = \wt(E_A) \in Q^+$, where $E_A \in U_q(\mfn)_{\wt(E_A)}$. Write
\begin{equation}\label{EqDecompK}
\mcK = \sum_{\gamma\in Q^+} \mcK_{\gamma},\qquad \mcK_{\gamma} = \quasiK_{\gamma} \xi T_{w_X}^{-1}T_{w_0}^{-1}
\end{equation}
with $\quasiK = \sum_{\gamma} \quasiK_{\gamma}$ and $\quasiK_{\gamma} \in U_q(\mfn)_{\gamma}$ as before. Then writing $V(\omega)$ for the $\omega$-weight space of a representation $V$, we have
\[
\mcK_{\gamma}: V(\omega) \rightarrow V(w_Xw_0(\omega)+ \gamma). 
\]

\begin{Lem}
We have
\[
\hat{\Phi}(U_{\varpi}(\xi,\eta)) = \sum_{A,B}\sum_{\gamma \in Q^+} w_Aw_B \kappa_{A,B,\gamma} \tau\tau_0(F^A)K_{\Theta(\wt(\eta) - \wt(B))+\tau\tau_0(\gamma)}^{-1}E^BK_{\wt(\eta)}^{-1}
\]
where $\kappa_{A,B,\gamma} = \langle \xi, E^A\mcK_{\gamma}F^B \eta\rangle$.
\end{Lem}
\begin{proof}
This follows straightforwardly by writing out the left hand expression using the formulas \eqref{EqPropR3} and \eqref{EqDecompK}. 
\end{proof}

Note that the $(A,B,\gamma)$-term  in the sum for $\hat{\Phi}(U_{\varpi}(\xi,\eta))$ has weight $-\tau\tau_0(\wt(A)) + \wt(B)$. 

Choose now for each $\varpi\in P^+$ non-zero vectors $\eta_{w_0(\varpi)},\xi_{w_X(\varpi)} \in V_{\varpi}$ with respective weights $w_0(\varpi)$ and $w_X(\varpi)$. Put
\[
k_{\varpi} = U_{\varpi}(\eta_{w_0(\varpi)},\xi_{w_X(\varpi)})
\]

\begin{Lem}\label{LemCart}
We have 
\[
\hat{\Phi}(k_{\varpi}) = t_{\varpi} K_{\tau(\varpi) + \Theta(\tau(\varpi))}
\]
where $t_{\varpi} = \langle \eta_{w_0(\varpi)},\mcK_0 \xi_{w_X(\varpi)}\rangle$ is non-zero.
\end{Lem}
\begin{proof}
We need to analyze the inner products 
\[
\kappa_{A,B,\gamma} = \langle \eta_{w_0(\varpi)},E^A\mcK_{\gamma} F^B \xi_{w_X(\varpi)}\rangle.
\]
Since $\eta_{w_0(\varpi)}$ is a lowest weight vector, it follows that $\kappa_{A,B,\gamma}=0$ unless $A=\gamma=0$. On the other hand, the element $\mcK_0 F^B \xi_{w_X(\varpi)}$ has weight $w_Xw_0(w_X(\varpi)-\wt(B)) = w_0(\varpi) - w_Xw_0(\wt(B))$, hence $\kappa_{0,B,0} = \langle   \eta_{w_0(\varpi)},\mcK_{0} F^B \xi_{w_X(\varpi)}\rangle$ is zero unless $w_0(\varpi) = w_0(\varpi) - w_Xw_0(\wt(B)) = 0$, i.e. $\wt(B)= 0$. 

It follows that
\[
\hat{\Phi}(k_{\varpi}) = \kappa_{0,0,0} K_{\Theta(w_X(\varpi))}^{-1}K_{w_X(\varpi)}^{-1}.
\]
Using $w_X(\varpi) = -\Theta(\tau(\varpi))$, we can write $K_{\Theta(w_X(\varpi))}^{-1}K_{w_X(\varpi)}^{-1} =  K_{\tau(\varpi) + \Theta(\tau(\varpi))}$. Finally, $\kappa_{0,0,0}\neq 0$ since the weight spaces at $w_0(\varpi)$ and $w_X(\varpi)$ are one-dimensional and $\mcK_0 = \xi T_{w_X}^{-1}T_{w_0}^{-1}$.
\end{proof}

Write now
\[
f_{\varpi,r} = U_{\varpi}(E_r\eta_{w_0(\varpi)},\xi_{w_X(\varpi)})
\]

\begin{Lem}\label{LemBor1}
There exists $t_{\varpi,r}\in \C$ and $Y_{\varpi,r} \in U_q(\mfb)$ such that 
\begin{equation}\label{EqFormImHatPhi}
\hat{\Phi}(f_{\varpi,r}) = t_{\varpi,r}F_{\tau\tau_0(r)}K_{\tau(\varpi)+\Theta(\tau(\varpi))}+ Y_{\varpi,r}.
\end{equation}
Moreover, 
\begin{enumerate}
\item $t_{\varpi,r}\neq 0$ if $E_r \eta_{w_0(\varpi)}\neq 0$, 
\item $Y_{\varpi,r} = 0$ if $r\in X$,
\item $Y_{\varpi,r} \in U_q(\mfb)_{-\Theta \tau\tau_0(\alpha_r)}$ for $r\in I \setminus X$.
\end{enumerate}
\end{Lem}
\begin{proof}
Fix $\varpi,r$, where we assume that $E_r \eta_{w_0(\varpi)}\neq 0$. We have to analyze again the coefficients
\[
\kappa_{A,B,\gamma} = \langle E_r \eta_{w_0(\varpi)},E^A\mcK_{\gamma}F^B \xi_{w_X(\varpi)}\rangle.
\]
Since $\eta_{w_0(\varpi)}$ is a lowest weight vector, it is clear that $\kappa_{A,B,\gamma}=0$ unless $A$ is the simple root $\alpha_r$ or the empty word. 

If $A= \alpha_r$, we have that 
\[
\kappa_{\alpha_r,B,\gamma} =  \langle E_r \eta_{w_0(\varpi)},E_{r}\mcK_{\gamma}F^B \xi_{w_X(\varpi)}\rangle = \langle E_{r}^*E_r \eta_{w_0(\varpi)},\mcK_{\gamma}F^B \xi_{w_X(\varpi)}\rangle.
\]
Since $E_r^*E_r\eta_{w_0(\varpi)}$ is a non-zero multiple of $\eta_{w_0(\varpi)}$, it follows as before that $\kappa_{\alpha_r,B,\gamma} =0$ unless $\gamma = B = 0$, in which case $\kappa_{\alpha_r,B,\gamma}$ is a non-zero scalar. This already accounts for the general form \eqref{EqFormImHatPhi} and (1).

Assume now that $A$ is the empty word. We claim that $\kappa_{0,B,\gamma}= 0$ unless $\gamma =0$. Indeed, from the proof of \cite{BK15b}*{Proof of Theorem 6.10} it follows that $\mcK_{\alpha_s} = 0$ for all $s\in I$, which proves the claim. On the other hand, by weight arguments the element
\[
\kappa_{0,B,0} = \langle E_r\eta_{w_0(\varpi)},\mcK_{0}F^B \xi_{w_X(\varpi)}\rangle 
\]
will be zero unless 
\[
w_0(\varpi) + \alpha_r = w_0(\varpi) -w_Xw_0(\wt(B)).
\]
This forces $\wt(B)  = -\Theta \tau\tau_0(\alpha_r)$. As $\Theta(\alpha_r) = \alpha_r$ for $r\in X$, it follows that in this case no such $B$ exist, and hence $Y_{\varpi,r}=0$. On the other hand, this also shows that $Y_{\varpi,r} \in U(\mfb)_{-\Theta \tau\tau_0(\alpha_r)}$ for $r\in I \setminus X$.
\end{proof}

We will need some more detailed information in the case $r\in I \setminus X$. In the following, we write $Q_X\subseteq Q$ for the root lattice of $\mfg_X$, generated by the $\alpha_r$ with $r\in X$. 

\begin{Def}
For $r\in I \setminus X$, we define 
\[
\Lambda_r = \varpi_r - \Theta(\varpi_r) \in P^+.
\]
\end{Def}

\begin{Prop}\label{PropBor2}
The following identity holds for $r\in I \setminus X$: $\hat{\Phi}(f_{\Lambda_r,r}) = t_{\Lambda_r,r} B_{\tau\tau_0(r)}$ with $t_{\Lambda_r,r}\neq 0$.
\end{Prop}
\begin{proof}
Clearly $E_r \eta_{w_0(\Lambda_r)}\neq 0$ since $(\Lambda_r,\alpha_r)>0$. This shows that $t_{\Lambda_r,r} \neq 0$. Also note that in this case
\[
\hat{\Phi}(f_{\Lambda_r,r}) = t_{\Lambda_r,r}F_{\tau\tau_0(r)}+ Y_{\Lambda_r,r}
\]
since $\Theta(\Lambda_r) = - \Lambda_r$. Hence using notation as in Definition \eqref{DefElBr}, we can write $\hat{\Phi}(f_{\Lambda_r,r}) = t_{\Lambda_r,r}B_{\tau\tau_0(r)} + Y_r'$ where
\[
Y_r' = Y_{\Lambda_r,r} - t_{\Lambda_r,r} c_{\tau\tau_0(r)}X_{\tau\tau_0(r)} K_{\tau\tau_0(\alpha_r)}^{-1}.
\]
We clearly still have $Y_r' \in U(\mfb)_{\beta}$ with $\beta = -\Theta\tau\tau_0(\alpha_r)$. We claim that $\beta \notin Q_X^+$. Indeed, if $\beta$ were in $Q_X^+$, it would follow from $\Theta_{\mid Q_X^+} = \id$ and $\Theta^2=\id$ that $\beta = -\tau\tau_0(\alpha_r)$, which is impossible. Since $Y_r' \in \msB$, we then find $Y_r' = 0$ by the following lemma. 
\end{proof}

\begin{Lem}
Let $\beta \in Q^+ \setminus Q_X^+$. Then $\msB \cap U(\mfb)_{\beta} = \{0\}$.
\end{Lem}
\begin{proof}
As usual, let us put $B_r = F_r$ for $r \in X$. For $J = (j_1,\ldots,j_n)$ with $j_k \in I$, write $F_J = F_{j_1}\ldots F_{j_n}$ and $B_J = B_{j_1}\ldots B_{j_n}$, and put $|J|=n$. Let $\mathcal{J}$ be a collection of indices such that $\{F_J\mid J \in \mathcal{J}\}$ is a basis of $U(\mfn^-)$. Let $U_q(\mfn_X)$ be the unital algebra generated by the $E_r$ with $r\in X$, and let as before $U_q(\mfh^{\Theta})$ be the algebra generated by the $K_{\omega}$ with $\omega \in P$ and $\Theta(\omega) = \omega$. In \cite{Kol14}*{Proposition 6.2} it is shown that $\{B_J\mid J \in \mathcal{J}\}$ is a basis  for $\msB$ as a left $U_q(\mfn_X)U_q(\mfh^{\Theta})$-module. This uses the following fact: if we take $B_J$ with $|J|=n$, then $B_J-F_J \in U(\mfb)\mathcal{F}^{n-1}(U(\mfn^-))$, where $\mathcal{F}^{\bullet}$ is the filtration of $U(\mfn^-)$ defined by $\mathcal{F}^n(U(\mfn^-)) = \mathrm{span}\{F_J\mid |J|\leq n\}$. 

Assume now that $Y \in U(\mfb)_{\beta}$ is a non-zero element with $\beta\notin Q_X^+$, and assume $Y \in \msB$. Write 
\[
Y = \sum_{J\in \mathcal{J}} T_JB_J,
\]
for uniquely determined $T_J \in U_q(\mfn_X)U_q(\mfh^{\Theta})$ with only finitely many non-zero. Since $Y$ is non-zero and weights for $U_q(\mfn_X)U_q(\mfh^{\Theta})$ lie in $Q_X^+$, there must exist $J \in \mathcal{J}$ with $|J|>0$ and $T_J\neq 0$. Let 
\[
N = \max \{|J|\mid T_J \neq 0\} >0.
\]
Then we have
\[
Y-\sum_{J\in \mathcal{J}} T_JF_J \in U(\mfb) \mathcal{F}^{N-1}(U(\mfn^-)).
\]
Since $Y \in U(\mfb)_{\beta}$, this would imply that also
\[
\sum_{J\in \mathcal{J}} T_JF_J \in U(\mfb) \mathcal{F}^{N-1}(U(\mfn^-)).
\]
However, as the $F_J$ form a basis of $U_q(\mfg)$ as a left $U_q(\mfb)$-module, this implies $T_J=0$ for all $J$ with $|J|= N$, in contradiction with the definition of $N$. This concludes the proof.
\end{proof}

\begin{Theorem}\label{TheoEqUnivEnv}
The equality $\mcU_q(\mfk') = \mcU_q(\mfu^{\theta})$ holds. 
\end{Theorem}
\begin{proof}
From Corollary \eqref{CorIncUSym} we already know that $\subseteq$ holds. For the reverse inclusion, it is by Lemma \ref{LemEqualitySwitch} sufficient to show that the weak closure of the image of $\hat{\Phi}$ equals the weak closure of $\msB$. However, by Lemma \ref{LemCart}, Lemma \ref{LemBor1} and Proposition \ref{PropBor2} we know that the range of $\hat{\Phi}$ contains the $K_{\tau(\varpi) + \Theta(\tau(\varpi))}$ for all $\varpi$, the $F_{\tau\tau_0(r)}K_{\tau(\varpi_r)+\Theta(\tau(\varpi_r))}$ for all $r\in X$, and the $B_r$ for all $r\in I \setminus X$. Since the range of $\hat{\Phi}$ is $*$-closed, its weak closure equals its bicommutant. This easily implies that the weak closure of the range of $\hat{\Phi}$ will equal the weak closure of $\msB$. 
\end{proof}

\subsection{Comparison of the coideal subalgebras $\mcO_q(K\backslash U)$ and $\mcO_q(U^{\theta}\backslash U)$}

Let again $\msK$ be the $*$-compatible $\nu$-modified universal $K$-matrix defined in \eqref{EqModKDef}, and let $\mcO_q(K\backslash U) \subseteq \mcO_q(U)$ be the right coideal $*$-subalgebra as constructed from the map $\phi$ in Theorem \ref{TheoOneToOneCorr}. Let $\mcO_q(U^{\theta}\backslash U) = \widehat{U}_q(\mfu^{\theta})\subseteq \mcO_q(U)$ be the  right coideal $*$-subalgebra dual to $U_q(\mfu^{\theta}) \subseteq U_q(\mfu)$ as defined by the general duality in Definition \ref{DefDualCoid}. We will show the following theorem.

\begin{Theorem}\label{TheoMainSymmFunct}
The equality $\mcO_q(K\backslash U) = \mcO_q(U^{\theta}\backslash U)$ holds, except possibly for $U^{\theta}\subseteq U$ containing a component of type $EIII$, $EIV$,  $EVI$, $EVII$ or $EIX$.
\end{Theorem}

The proof of this theorem will not be uniform, and will be subdivided into a separate treatment for different classes, see Proposition \ref{PropMainClass1}, Proposition \ref{PropMainClass2}, Proposition \ref{PropMainClass3} and Proposition \ref{PropMainClass4}. For some cases the proof is straightforward, and for others the proof is very ad hoc and computational. As such, our methods were not strong enough to cover also the mentioned $E$-cases in the above theorem. However, at the end of the section we will sketch a uniform proof for all cases when $q$ is sufficiently close to 1, Theorem \ref{TheoDeform}, based on deformation theory. 

Let us first start with the following easy result. 

\begin{Prop}
The inclusion $\mcO_q(K\backslash U) \subseteq \mcO_q(U^{\theta}\backslash U)$ holds.
\end{Prop}
\begin{proof}
This follows by general duality for coideals, Theorem \ref{TheoEqUnivEnv} and Proposition \ref{PropInclusion}.
\end{proof}

To obtain an inclusion in the other direction, we first recall the following classical terminology. Let us write
\[
\mcO_q(U)_{\varpi} = \textrm{linear span}\{U(\xi,\eta)\mid \xi,\eta\in V_{\varpi}\},\qquad \mcO_q(U^{\theta}\backslash U)_{\varpi} = \mcO_q(U)_{\varpi} \cap \mcO_q(U^{\theta}\backslash U)
\]
for spectral subspaces, and call 
\[
m_{q,\varpi} = \frac{\dim(\mcO_q(U^{\theta}\backslash U)_{\varpi})}{\dim(V_{\varpi})}
\]
the \emph{associated multiplicity}. We also use this notation at $q=1$.

\begin{Def}
A positive integral weight $\varpi \in P^+$ is called \emph{spherical} with respect to $U^{\theta}$ if $\mcO(U^{\theta}\backslash U)_{\varpi}\neq 0$. We denote by $P^+_{\spher}$ the set of spherical weights.
\end{Def}

The following theorem states in particular that $U^{\theta}\subseteq U$ and its quantum companion are \emph{Gelfand pairs}, i.e. the multiplicity function is $\{0,1\}$-valued. 

\begin{Theorem}
\label{ThmMultFree}
If $\varpi \in P^+$ and $q>0$, then $m_{q,\varpi} = \delta_{\varpi \in P_s^+}$. 
\end{Theorem}
\begin{proof}
For $q=1$ this is classical. For $q\neq 1$ this is proven\footnote{The references assume that $q$ is an indeterminate, but one can check that the proofs for these specific results are also valid for the case of $q$ a scalar.} in \cite{Let00}*{Theorem 4.2 and Theorem 4.3}, see also \cite{Let02}*{Theorem 7.8}. 
\end{proof}

Let now $I_{\Sigma} \subseteq I\setminus X$ be a fundamental domain for $\tau$, and define $\mu_r\in P^+$ for $r \in I_{\Sigma}$ as follows in terms of the Satake diagram $(X,\tau)$ and the fundamental weights $\varpi_r$:
\[
\mu_r = \left\{ \begin{array}{ll} \varpi_r&\textrm{if }\tau(r)=r\textrm{ and }r\textrm{ is connected to a black vertex},\\
2\varpi_r &\textrm{if }\tau(r)=r\textrm{ and }r\textrm{ is not connected to a black vertex},\\
\varpi_r + \varpi_{\tau(r)} &\textrm{if }\tau(r)\neq r.
\end{array}\right.
\]

\begin{Theorem}[\cite{Sug62}*{Theorem 2 and Theorem 4}, see also \cite{Vre76}] 
A weight $\mu$ is spherical if and only $\mu$ is a positive integer combination of the $\mu_r$. 
\end{Theorem}

We call the $\mu_r$ the \emph{fundamental spherical weights}. In particular, it follows that $P_s^+ $ is a cone. 

\begin{Cor}\label{CorImGen}
The spectral subspaces $\mcO_q(U^{\theta}\backslash U)_{\mu_r}$ generate $\mcO_q(U^{\theta}\backslash U)$ as an algebra.
\end{Cor}
\begin{proof} 
Let $v_r$ be a non-zero $U_q(\mfu^{\theta})$-invariant vector in $V_{\mu_r}$, and let $\xi_r$ be a highest weight vector in $V_{\mu_r}$. Then $f_r = U(v_r,\xi_r) \in \mcO_q(U^{\theta}\backslash U)$ must be non-zero, as it generates $\mcO_q(U^{\theta}\backslash U)_{\mu_r}$ as a $U_q(\mfg)$-module for the translation action $X \rhd f =f(-X)$. Now as $\mcO_q(U)$ does not have zero-divisors \cite{Jos95}*{Lemma 9.1.9.(1)}, it follows that $f_n = f_{r_1}^{n_1}\ldots f_{r_ a}^{n_a} \neq0$, with $r_1,\ldots,r_a$ an enumeration of $I_{\Sigma}$ and $n_s \in \N$. On the other hand, it is clear that $f_n \in \mcO_q(U^{\theta}\backslash U)_{\mu}$ for $\mu = \sum_{s=1}^a n_s \mu_s$. As the algebra generated by the $\mcO_q(U^{\theta}\backslash U)_{\mu_r}$ is closed under the $U_q(\mfu)$-action, this shows that the latter algebra must equal the whole of $\mcO_q(U^{\theta}\backslash U)$. 
\end{proof}

\begin{Lem}\label{LemAppRang}
Assume $\varpi \in P^+$ vanishes on $X$. Then $\mcO_q(U^{\theta}\backslash U)_{\varpi + \tau(\varpi)}$ lies in the range of $\phi$. 
\end{Lem} 
\begin{proof}
Let $\xi_{\varpi}$ and $\eta_{w_0\varpi}$ be respectively a non-zero highest weight and lowest weight vector in $V_{\varpi}$. Then it is clear that $Z(\eta_{w_0\varpi},\xi_{\varpi})$ is a highest weight vector for the natural $U_q(\mfg)$-action \eqref{EqInvAct} on $\mcO_q(Z_{\nu})$, with highest weight $\varpi- \tau \tau_0 w_0(\varpi)  = \varpi + \tau(\varpi)$. It follows that 
\[
\phi(Z(\eta_{w_0\varpi},\xi_{\varpi})) \in \mcO_q(U^{\theta}\backslash U)_{\varpi + \tau(\varpi)}.
\] 
Since the latter spectral subspace has multiplicity one, it now suffices by $U_q(\mfu)$-equivariance of $\phi$ to show that $\phi(Z(\eta_{w_0\varpi},\xi_{\varpi}))\neq 0$. This will follow once we show that
\[
\varepsilon(\phi(Z(\eta_{w_0\varpi},\xi_{\varpi}))) = \langle \eta_{w_0\varpi},\msK \xi_{\varpi}\rangle \neq 0. 
\]
Now from the correspondence between $\msK = \tau\tau_0(\msK)^*$ and $\mcK$, it is clear from taking contragredient representations that 
\[
 \langle \eta_{w_0\varpi},\msK \xi_{\varpi}\rangle \neq 0 \qquad \Leftrightarrow \qquad \langle \xi_{\tau_0(\varpi)},\mcK \eta_{\tau_0(\varpi)}\rangle \neq 0.
\]
Write now $\mcK = \sum_{\gamma \in Q^+} \mcK_{\gamma}$ as in \eqref{EqDecompK}. As $w_X\varpi = \varpi$ by the assumption $\varpi_{\mid X} = 0$, we have that $T_{w_X}^{-1}T_{w_0}^{-1} \eta_{\tau_0(\varpi)}$ will be a non-zero multiple of $\xi_{\tau_0(\varpi)}$. Hence $\langle\xi_{\tau_0(\varpi)},\mcK_{\gamma} \eta_{\tau_0(\varpi)}\rangle=0$ for $\gamma\neq 0$. From this, it is clear that 
\[
 \langle \xi_{\tau_0(\varpi)},\mcK \eta_{\tau_0(\varpi)}\rangle  =  \langle \xi_{\tau_0(\varpi)},\mcK_0 \eta_{\tau_0(\varpi)}\rangle \neq 0. \qedhere
\]
\end{proof}

\begin{Cor}\label{CorImFundWeight}
The spectral subspace $\mcO_q(U^{\theta}\backslash U)_{\mu_r}$ lies in the range of $\phi$ for $r\in I_{\Sigma}$ with $\tau(r)\neq r$, and for $r\in I_{\Sigma}$ with $\tau(r) = r$ but $r$ not connected to a black vertex. For $r\in I_{\Sigma}$ with $\tau(r) = r$ and $r$ connected to a black vertex, we have that $\mcO_q(U^{\theta}\backslash U)_{2\mu_r}$ lies in the range of $\phi$.
\end{Cor}

\begin{Prop}\label{PropMainClass1}
The equality $\mcO_q(K\backslash U) = \mcO_q(U^{\theta}\backslash U)$ holds in the following irreducible cases, using still the classification as in \cite{Ar62}: 
\begin{itemize}
\item $AI$, $AIII$, $AIV$, 
\item $CI$, 
\item $DI$ in the case of $\mfso(p)\times \mfso(2l-p) \subseteq \mfso(2l)$ with $p=l-1$ or $p=l$,
\item $EI$, $EII$, $EV$, $EVIII$, 
\item $FI$,
\item $G$,
\item diagonal inclusions $\mfu \subseteq \mfu \oplus \mfu$.
\end{itemize} 
In particular, equality holds for all the symmetric pairs corresponding to split real semisimple Lie algebras.  
\end{Prop}
\begin{proof}
This follows immediately from Corollary \ref{CorImGen}, Corollary \ref{CorImFundWeight} and the fact that in the Satake diagrams corresponding to the above cases there are no $\tau$-fixed white points connected to black vertices. 
\end{proof}

There are some further cases which can be obtained by an easy argument, using the following lemma.

\begin{Lem}\label{LemOpNotScal}
For all $\varpi \in P^+$, the operator $\pi_{\varpi}(\msK)$ is not a scalar.
\end{Lem}
\begin{proof}
If $\xi_{\varpi}$ is a highest weight vector, we have that $\msK \xi_{\varpi}$ is a non-zero scalar multiple of the vector $T_{w_X}^{-1}T_{w_0}^{-1} \xi_{\varpi}$. The latter is a weight vector at weight $w_Xw_0 \varpi = -w_X\tau_0(\varpi)$. We claim that $-w_X \tau_0(\varpi) \neq \varpi$, which will finish the proof. Namely, suppose this were not the case. Then $w_X \varpi = -\tau_0(\varpi)$ is a negative weight. However, let $\beta \in \Delta^+$ be the highest root. Since $\beta = \sum_{r\in I} k_{r} \alpha_r$ with $k_r>0$ for all $r$, we know that $\beta \notin \Delta_X^+$. On the other hand,  $w_X$ preserves the set $\Delta^+ \setminus \Delta_X^+$, and hence $w_X\beta \geq 0$. It follows that 
\[
0 \geq (w_X\varpi,w_X\beta) = (\varpi,\beta)>0,
\]
a contradiction. 
\end{proof}

\begin{Cor}\label{CorImContDich}
Let $\mu\in P_s^+ \setminus\{0\}$, and assume that there exists a positive weight $\varpi \in P^+$ such that $V_{\tau(\varpi)}\otimes V_{\varpi}$ contains $V_{\mu}$ as its only non-trivial spherical representation. Then $\mcO_q(U^{\theta}\backslash U)_{\mu}$ lies in the range of $\phi$.
\end{Cor}
\begin{proof}
The matrix coefficients of the $Z_{\varpi}(\xi,\eta)$ span an $U_q(\mfg)$-module which is isomorphic to the tensor product representation $V_{\tau\tau_0(\varpi)}^*\otimes V_{\varpi} \cong V_{\tau(\varpi)}\otimes V_{\varpi}$.  Hence the range of this module lies in $\C1 + \mcO_q(U^{\theta}\backslash U)_{\mu}$ by assumption. However, the range can not be $\C1$, as this would imply by \eqref{EqImZphiMatrix} that $\pi_{\varpi}(\msK)$ intertwines $\pi_{\varpi}$ and $\pi_{\varpi}\circ\tau\tau_0$, and must hence be a  scalar, in contradiction with Lemma \ref{LemOpNotScal}. This entails that $\mcO_q(U^{\theta}\backslash U)_{\mu}$ lies in the range of $\phi$. 
\end{proof}

\begin{Prop}\label{PropMainClass2}
The equality $\mcO_q(K\backslash U) = \mcO_q(U^{\theta}\backslash U)$ holds also in the following irreducible cases:
\begin{itemize}
\item $BI$, $BII$, 
\item $DI$ (remaining cases) and $DII$.
\end{itemize}
\end{Prop}
\begin{proof}
In these cases, there is a unique white vertex $s$ which is $\tau$-fixed and connected to a black vertex. By Corollary \ref{CorImGen}, Corollary \eqref{CorImFundWeight} and Corollary \ref{CorImContDich}, it is then sufficient to show that there exists an irreducible representation $V_{\varpi}$ such that $V_{\mu_s}$ is the only non-trivial spherical representation in $V_{\tau(\varpi)}\otimes V_{\varpi}$.

Let us consider the BI and BII cases first. In these cases $\mathfrak{g}$ is of type $B_\ell$, $\tau = \id$ and $X = \{p + 1, \cdots, \ell \}$ with $1 \leq p \leq \ell - 1$. The Satake diagram is as follows.
\[
 \begin{tikzpicture}[scale=.4,baseline]
\draw (0cm,0) circle (.2cm) node[above]{\small $1$} ;
\draw (0.2cm,0) -- +(0.2cm,0);
\draw[dotted] (0.4cm,0) --+ (1cm,0);
\draw (1.4cm,0) --+ (0.2cm,0);
\draw (1.8cm,0) circle (.2cm); 
\draw (2cm,0) --+ (0.6cm,0);
\draw (2.8cm,0) circle (.2cm)  node[above]{\small $p$};
\draw (3cm,0) --+ (0.6cm,0);
\draw[fill = black] (3.8cm,0) circle (.2cm);
\draw (4.2cm,0) -- +(0.2cm,0);
\draw[dotted] (4.4cm,0) --+ (1cm,0);
\draw (5.4cm,0) --+ (0.2cm,0);
\draw[fill = black] (5.8cm,0) circle (.2cm);
\draw[fill = black] (6.8cm,0) circle (.2cm) node[above]{\small $\ell$};
\draw
(6,-.06) --++ (0.6,0)
(6,+.06) --++ (0.6,0);
\draw
(6.4,0) --++ (60:-.2)
(6.4,0) --++ (-60:-.2);
\end{tikzpicture}
\vspace{0.4cm}
\]
The spherical weights are $\mu_r = 2 \varpi_r$ for $r = 1, \cdots, p - 1$ and $\mu_p = \varpi_p$. We have
\[
V_{\varpi_\ell} \otimes V_{\varpi_\ell} \cong V_{2 \varpi_\ell} \oplus \left( \bigoplus_{r = 1}^\ell V_{\varpi_{\ell - r}} \right),
\]
where by convention $V_{\varpi_0}$ is the trivial representation, see for instance \cite{OV90}*{Reference Chapter, Section 2, Table 5}. Hence $V_{\varpi_p}$ is the only non-trivial spherical representation appearing in $V_{\varpi_\ell} \otimes V_{\varpi_\ell}$.

Next consider the cases DI and DII. In these cases $\mathfrak{g}$ is of type $D_\ell$ and $X = \{p + 1, \cdots, \ell \}$ with $1 \leq p \leq \ell - 2$. The automorphism $\tau$ depends on the parity of $\ell - p$: we have $\tau = \id$ for $\ell - p$ even, while for $\ell - p$ odd we have that $\tau$ switches the two end nodes of the Dynkin diagram. 
\vspace{-1.2cm}
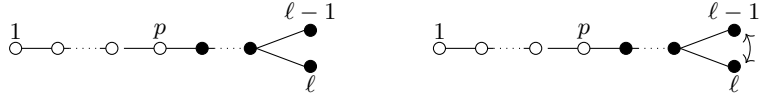
\begin{table}[ht]
\begin{center}
\bgroup
\def\arraystretch{3}
{\setlength{\tabcolsep}{1.5em}
\begin{tabular}{cc}
 \begin{tikzpicture}[scale=.4,baseline=1cm]
\node (v1) at (10,0.8) {};
\node (v2) at (10,-0.8) {};
\draw (0cm,0) circle (.2cm) node[above]{\small $1$} ;
\draw (0.2cm,0) -- +(1cm,0);
\draw (1.4cm,0) circle (.2cm);
\draw (1.6cm,0) -- +(0.2cm,0);
\draw[dotted] (1.8cm,0) --+ (1cm,0);
\draw (2.8cm,0) --+ (0.2cm,0);
\draw (3.2cm,0) circle (.2cm); 
\draw (3.6cm,0) --+ (1cm,0);
\draw (4.8cm,0) circle (.2cm)  node[above]{\small $p$};
\draw (5cm,0) --+ (1cm,0);
\draw[fill = black] (6.2cm,0) circle (.2cm);
\draw (6.4cm,0) -- +(0.2cm,0);
\draw[dotted] (6.6cm,0) --+ (1cm,0);
\draw (7.6cm,0) --+ (0.2cm,0);
\draw[fill = black] (7.8cm,0) circle (.2cm);
\draw (8cm,0) --+ (1.6,0.6);
\draw (8cm,0) --+ (1.6,-0.6);
\draw[fill = black] (9.8cm,0.6) circle (.2cm) node[above]{\small $\ell-1$} ;
\draw[fill = black] (9.8cm,-0.6) circle (.2cm) node[below]{\small $\ell$};
\end{tikzpicture}
& 
 \begin{tikzpicture}[scale=.4,baseline=1cm]
\node (v1) at (10,0.8) {};
\node (v2) at (10,-0.8) {};
\draw (0cm,0) circle (.2cm) node[above]{\small $1$} ;
\draw (0.2cm,0) -- +(1cm,0);
\draw (1.4cm,0) circle (.2cm);
\draw (1.6cm,0) -- +(0.2cm,0);
\draw[dotted] (1.8cm,0) --+ (1cm,0);
\draw (2.8cm,0) --+ (0.2cm,0);
\draw (3.2cm,0) circle (.2cm); 
\draw (3.6cm,0) --+ (1cm,0);
\draw (4.8cm,0) circle (.2cm)  node[above]{\small $p$};
\draw (5cm,0) --+ (1cm,0);
\draw[fill = black] (6.2cm,0) circle (.2cm);
\draw (6.4cm,0) -- +(0.2cm,0);
\draw[dotted] (6.6cm,0) --+ (1cm,0);
\draw (7.6cm,0) --+ (0.2cm,0);
\draw[fill = black] (7.8cm,0) circle (.2cm);
\draw (8cm,0) --+ (1.6,0.6);
\draw (8cm,0) --+ (1.6,-0.6);
\draw[fill = black] (9.8cm,0.6) circle (.2cm) node[above]{\small $\ell-1$} ;
\draw[fill = black] (9.8cm,-0.6) circle (.2cm) node[below]{\small $\ell$} ;
\draw[<->]
(v1) edge[bend left] (v2);
\end{tikzpicture}
\end{tabular}
}
\egroup
\end{center}
\end{table}

The spherical weights in these cases are $\mu_r = 2 \varpi_r$ for $r = 1, \cdots, p - 1$ and $\mu_p = \varpi_p$. We have the following tensor product decompositions:
\[
V_{\varpi_\ell} \otimes V_{\varpi_\ell} \cong V_{2 \varpi_\ell} \oplus \left( \bigoplus_{r \geq 1} V_{\varpi_{\ell - 2r}} \right), \qquad
V_{\varpi_\ell} \otimes V_{\varpi_{\ell - 1}} \cong V_{\varpi_{\ell - 1} + \varpi_\ell} \oplus \left( \bigoplus_{r \geq 1} V_{\varpi_{\ell - 2r - 1}} \right),
\]
see again \cite{OV90}*{Reference Chapter, Section 2, Table 5}. In the case $\ell - p$ even we use the first decomposition. Then we have $V_{\varpi_{\ell - 2r}} = V_{\varpi_p}$ for $r = (\ell - p) / 2$ and this is the only non-trivial spherical representation appearing. In the case $\ell - p$ odd we use the second decomposition, since $\tau(\ell) = \ell - 1$. Then we have $V_{\varpi_{\ell - 2r - 1}} = V_{\varpi_p}$ for $r = (\ell - p - 1) / 2$ and this is the only non-trivial spherical representation appearing.
\end{proof}

The only remaining classical cases to be dealt with are now the following.

\begin{Prop}\label{PropMainClass3}
The equality $\mcO_q(K\backslash U) = \mcO_q(U^{\theta}\backslash U)$ holds also in the following irreducible cases: 
\begin{itemize}
\item $AII$, 
\item $CII$, 
\item $DIII$.
\end{itemize}
\end{Prop}

\begin{proof}
We give the corresponding Satake diagrams of $\mfsl_{p+1}(\mathbb{H})$, $\mfsp(p,q)$ and $\mfsu^*_l(\mathbb{H})$  the standard ordering as can be found for example in the tables \ref{TableABC} and \ref{TableD}. Then taking the weight $\varpi = \varpi_1$ in Corollary \ref{CorImContDich}, we obtain that the range of $\phi$ contains $\mcO_q(U^{\theta}\backslash U)_{\mu_2}$, where $\mu_2 = \varpi_2$.
It suffices to show that these elements generate $\mcO_q(U^{\theta}\backslash U)$ as an algebra. This claim will be proven in Proposition \ref{PropAlgGenSph} in Appendix \ref{Ap2}. 
\end{proof}

Finally, we treat the case $FII$ of $\mfg = \mathfrak{f}_4$. This corresponds to the following Satake diagram.
\[
\begin{tikzpicture}[scale=.4,baseline]
\draw (0cm,0) circle (.2cm) node[above]{\small $1$};
\draw (0.2cm,0) --+ (0.6cm,0);
\draw[fill = black] (1cm,0) circle (.2cm) node[above]{\small $2$};
\draw[fill = black] (2cm,0) circle (.2cm)  node[above]{\small $3$};
\draw
(1.2,-.06) --++ (0.6,0)
(1.2,+.06) --++ (0.6,0);
\draw
(1.5,0) --++ (60:.2)
(1.5,0) --++ (-60:.2);
\draw (2.2cm,0) --+ (0.6cm,0);
\draw[fill = black] (3cm,0) circle (.2cm) node[above]{\small $4$};
\end{tikzpicture}
\vspace{0.4cm}
\]
For this case, we will very explicitly verify that $\mcO_q(K\backslash U)_{\mu_1}\neq 0$. 

\begin{Prop}\label{PropMainClass4}
The equality $\mcO_q(K\backslash U) = \mcO_q(U^{\theta}\backslash U)$ holds also in the case irreducible case $FII$.
\end{Prop}
\begin{proof}
By Corollary \ref{CorImGen} and equivariance of $\phi$, it is sufficient to show that $\mcO_q(U^{\theta}\backslash U)_{\varpi_1}$ contains a non-zero element. This is the content of Proposition \ref{PropExistNonZero} in Appendix \ref{Ap3}.
\end{proof}

In what follows, we will sketch a uniform argument showing that Theorem \ref{TheoMainSymmFunct} holds true, also in the exceptional cases, for $q$ close to $1$. As this result is not as strong as we would like, we will not be very detailed. 

\begin{Theorem}\label{TheoDeform}
For any compact symmetric pair $U^{\theta} \subseteq U$, the identity $\mcO_q(K\backslash U) = \mcO_q(U^{\theta}\backslash U)$ holds for $q$ sufficiently close to $1$. 
\end{Theorem}
\begin{proof}
First we note that the highest weight modules $V_{\varpi}$ for $U_q(\mfg)$ can be identified as vector spaces over all $0<q$ in such a way that the structure coefficients of $\mcO_q(G)$ depend continuously on $q$ \cite{NT11}*{Theorem 1.2}. Similarly, the structure coefficients of the algebra $\mcO_q(Z_{\nu})$ depend then continuously on $q$, becoming in the limit $q=1$ the vector space $\mcO(G)$ with product and $*$-structure
\[
f*g = (g_{(1)}\otimes f_{(1)},\Omega_{\epsilon})f_{(2)}g_{(2)},\qquad f^{\sharp}(g) = \overline{f(\tau(g)^*)}
\]
We now claim that the $*$-homomorphisms
\[
\phi = \phi_q: \mcO_q(Z_{\nu}) \rightarrow \mcO_q(U)
\] 
have a well-defined limit at $q=1$. For this, it is enough to show that the $K$-matrix $\msK = \msK_q$, or equivalently $\mcK_q$ has a well-defined limit at $q=1$. Now $T_{w_0}$ and $T_{w_X}$ converge respectively to $m_0$ and $m_X$, while the functions $\xi$ and $\xi'$ converge to $\widetilde{z}_{\tau}$. Finally, from its construction in \cite{BK15b} one sees that the quasi-$K$-matrix $\quasiK_q$ becomes the unit at $q=1$. Since also the ribbon element $v$ becomes $1$ in the limit $q =1$, it follows that the $\msK_q$ indeed vary continuously over $q$, and at the limit $q=1$ we have
\[
\msK_1 = \msE \widetilde{\epsilon}^{-1} m_0m_X\widetilde{z}_{\tau}^{-1}.
\]
If now $\phi_1$ is surjective, it follows that $\phi_q$ will hit all the $\mcO_q(U^{\theta}\backslash U)_{\mu_r}$ for $q$ close to 1, and hence $\phi_q$ will be surjective by Corollary \ref{CorImGen}. 

To see that $\phi_1$ is surjective, note that the range of $\phi_1$, being a coideal, will be of the form $\mcO(K\backslash U)$ for $K$ a closed subgroup of $U^{\theta}$. Since 
\[
\phi_1(f)(u) = f(\tau_0\tau(u)^{-1}\msK_1 u),\qquad u\in U,
\] 
it follows that 
\[
K = \{u \in U \mid \tau_0\tau(u)^{-1}\msK_1 u = \msK_1\}.
\] 
But as $\msE\widetilde{\epsilon}^{-1}$ is central, as $\theta = \Ad(\widetilde{z})\circ \tau\tau_0\circ \Ad(m_0)\Ad(m_X)$ and as $m_0m_X \widetilde{z}_{\tau} = \widetilde{z}m_0m_X$, we see that $K = U^{\theta}$. 
\end{proof}

Let us end with the following remark. 

\begin{Rem}
By construction, the Vogan automorphism $\nu$ of $U$ determined by $(\tau\tau_0,\epsilon)$ will be inner conjugate to $\theta$, say by $u \in U$, 
\[
\theta = \Ad(u)\nu\Ad(u^{-1}).
\]
Let
\[
w= w_{\nu,\theta} = u \nu(u)^* \in U.
\]
Then $w^* = \nu(w)$, i.e. $w\in H_{\nu}$ with $H_{\nu}$ as in \eqref{EqHnu}, and moreover
\[
\Ad(w)(x) = \nu(x),\qquad x\in U^{\theta}.
\]
In particular, we obtain a map 
\[
U^{\theta}\backslash U \rightarrow H_{\nu},\quad U^{\theta}x \mapsto \nu(x)^{-1}wx.
\]
Now by the proof of Theorem \ref{TheoDeform}, we see that in the classical limit $\msK$ corresponds to the element $w'= \widetilde{\varepsilon}^{-1}m_0m_X\widetilde{z}_{\tau}^{-1}$. It would hence be interesting to see if one can take $w=w'$, and if then the factorisation $w = u \nu(u)^*$ passes through to the quantum setting for $\msK$. We believe that this will be connected to a notion of \emph{quantum Cayley transform}, see \cite{Let17} for some closely related material. 
\end{Rem}

\appendix

\section{Variations on twisting}\label{Ap0} 

In this appendix, we consider some variations on the results in Section \ref{SecTwistBraid} by modifying the twist. We resume the notation of that section.

As a first variation, consider the opposite universal $R$-matrix and associated coquasitriangular structure 
\[
\widetilde{\msR} = \msR_{21}^{-1},\qquad \widetilde{\mbr} = \mbr_{21}^{-1}. 
\]
With $\nu \in \End_*(\mfb)$, we can then also consider 
\[
{\msR}_{\nu} = \msR_{\nu,21}^{-1},\qquad \widetilde{\mbr}_{\nu} = \mbr_{\nu,21}^{-1}
\]
and the associated convolution invertible real $2$-cocycle functional 
\[
\widetilde{\omega}_{\nu}: \mcO_q^{\com}(G_{\R}) \times \mcO_q^{\com}(G_{\R}) \rightarrow \C,\quad \widetilde{\omega}_{\nu}(fg^{\dag},hk^{\dag}) = \varepsilon(f) \widetilde{\mbr}_{\nu}(h,g^*)\overline{\varepsilon(k)}.
\]

\begin{Def}
For  $\nu,\mu\in \End_*(U_q(\mfb))$ we define $\widetilde{\mcO}_q^{\nu,\mu}(G_{\R})$, resp. $\dbwidetilde{\mcO}_q^{\nu,\mu}(G_{\R})$ as the vector space $\mcO_q^{\com}(G_{\R})$ endowed with the respective new multiplications
\[
\widetilde{m}_{\nu,\mu}(f,g) = \widetilde{\omega}_{\nu}(f_{(1)},g_{(1)}) f_{(2)}g_{(2)} \omega_{\mu}^{-1}(f_{(3)},g_{(3)}),\qquad f,g\in \mcO_q^{\com}(G_{\R}),
\]
\[
\dbwidetilde{m}_{\nu,\mu}(f,g) = \widetilde{\omega}_{\nu}(f_{(1)},g_{(1)}) f_{(2)}g_{(2)} \widetilde{\omega}_{\mu}^{-1}(f_{(3)},g_{(3)}),\qquad f,g\in \mcO_q^{\com}(G_{\R})
\]
and the original $*$-structure. 
\end{Def}
As before, these can be made into a connected cogroupoid with compatible $*$-structure using the tensor product comultiplication on $\mcO_q^{\com}(G_{\R})$. In particular, we have the Hopf $*$-algebra $\dbwidetilde{\mcO}_q^{\nu}(G_{\R}) = \dbwidetilde{\mcO}_q^{\nu,\nu}(G_{\R})$, and the $*$-algebra $\widetilde{\mcO}_q^{\nu,\id}(G_{\R})$ with commuting left and right coactions by respectively $\dbwidetilde{\mcO}_q^{\nu}(G_{\R})$ and $\mcO_q(G_{\R})$. We then have the following straightforward modifications of the results in Section \ref{SecTwistBraid}. Unexplained notation should be straightforward to interpret.  

\begin{Lem}[Cf. Lemma \ref{LemFundInt}]
Let $\nu,\mu \in \End_*(U_q(\mfb))$, and let $\pi,\pi'$ be representations of $U_q(\mfu)$. In $\widetilde{\mcO}_q^{\nu,\mu}(G_{\R})$, resp.~ $\dbwidetilde{\mcO}_q^{\nu,\mu}(G_{\R})$  we have the commutation relations
\[
\widetilde{Y}_{13}' \msR_{\mu,12}^{\pi',\pi} \widetilde{Y}_{23}^{\dag} = \widetilde{Y}_{23}^{\dag} \widetilde{\msR}_{\nu,12}^{\pi',\pi}\widetilde{Y}_{13}', \qquad 
\dbwidetilde{Y}_{13}' \widetilde{\msR}_{\mu,12}^{\pi',\pi}\dbwidetilde{Y}_{23}^{\dag} = \dbwidetilde{Y}_{23}^{\dag} \widetilde{\msR}_{\nu,12}^{\pi',\pi}\dbwidetilde{Y}_{13}'.
\]
\end{Lem}

\begin{Prop}[cf. Proposition \ref{PropExtPairing}]
There is a unique pairing $(-,-)_{\epsilon}'$ of Hopf algebras between $U_q^{\epsilon}(\mfg)$ and $\mcO_q(G)$ such that
\begin{equation}\label{EqRestrHol}
(K_{\omega},f)_{\epsilon}' = (K_{\omega},f),\quad (E_r,f)_{\epsilon}' = (E_r,f),\quad (F_r,f)_{\epsilon}' = \epsilon_r(F_r,f). 
\end{equation}
Moreover, there is a unique pairing $(-,-)_{\nu}$ of Hopf $*$-algebras between $U_q(\mfg_{\nu})$ and $\dbwidetilde{\mcO}_q^{\nu}(G_{\R})$ extending the pairing $(-,-)_{\epsilon}$ extending the above pairing $(-,-)_{\epsilon}'$.
\end{Prop}

We then denote by $\dbwidetilde{\mcO}_q(G_{\nu})$ the Hopf $*$-algebra arising as the coimage of $\dbwidetilde{O}_q^{\nu}(G_{\R})$ obtained by dividing out through the kernel of this pairing. Clearly $\dbwidetilde{\mcO}_q(G_{\nu})\cong \mcO_q(G_{\nu})$ as Hopf $*$-algebras.

Let us denote $\widetilde{\mcO}_q(G_{\nu}\dbbackslash G_{\R})$ for the coinvariants in $\widetilde{\mcO}_q^{\nu,\id}(G_{\R})$ with respect to the left coaction by $\dbwidetilde{\mcO}_q^{\nu}(G_{\R})$. Denote by $\widetilde{\mcO}_q(Z_{\nu})$ the vector space $\mcO_q(G)$ with the product 
\begin{equation}\label{EqDefBraidProdAlt}
f \;\widetilde{*}\; g = \mbr(f_{(1)},g_{(2)}) (f_{(2)}\otimes g_{(3)},\Omega_{\epsilon}) f_{(3)}g_{(4)}\widetilde{\mbr}(f_{(4)},\tau(S(g_{(1)}))).
\end{equation}
and the $*$-structure $f^{\sharp} = \tau(S(f)^*)$. 

\begin{Theorem}[Cf. Theorem \ref{TheoIsoZCoset}]
The map
\[
\tilde{\j}_{\nu}: \widetilde{\mcO}_q(Z_{\nu}) \cong \widetilde{\mcO}_q(G_{\nu}\dbbackslash G_{\R}),\quad f \mapsto (f_{(2)},\msE) S(f_{(1)})\tau(f_{(3)})^{*\dag}
\]
induces an isomorphism of $*$-algebras.
\end{Theorem}
In terms of the generating matrices of $\widetilde{\mcO}_q(Z_{\nu})$, which we write $\widetilde{Z}$, this means
\[
\tilde{\j}_{\nu}: \widetilde{Z}_{\pi} \mapsto \widetilde{Y}_{\pi}^{-1}(\msE_{\pi}\otimes 1)\tau(\widetilde{Y})_{\pi}^{-1,\dag}.
\] 
We then have the following form of the reflection equation: 
\[
\widetilde{\msR}_{21}^{\pi,\pi'} \widetilde{Z}_{13} \widetilde{\msR}_{\tau,12}^{\pi,\pi'} \widetilde{Z}_{23}' = \widetilde{Z}_{23}' \widetilde{\msR}_{\tau,21}^{\pi,\pi'} \widetilde{Z}_{13} \widetilde{\msR}_{12}^{\pi,\pi'},
\]
with the induced right coaction of $\mcO_q(G_{\R})$ now given by 
\[
\widetilde{Z}_{\pi} \mapsto Y_{\pi,13}^{-1} \widetilde{Z}_{\pi,12}\tau(Y_{\pi})_{13}^{-1,\dag}. 
\] 
Inverting \eqref{EqDefBraidProdAlt} leads to 
\[
(f_{(1)}\otimes g_{(1)},\Omega_{\epsilon})f_{(2)}g_{(2)} = \mbr(S(f_{(1)}),g_{(1)}) f_{(2)} \;\widetilde{*}\; g_{(3)} \widetilde{\mbr}(f_{(3)},\tau(g_{(2)})), 
\]
so that the corresponding $*$-characters of $\widetilde{\mcO}_q(Z_{\nu})$ are those $\widetilde{\msK}\in \mcU_q(\mfg)$ with 
\begin{equation}\label{EqKAlt}
\widetilde{\msK}^*  = \tau(\widetilde{\msK}),\qquad \Omega_{\epsilon}\Delta(\widetilde{\msK}) = \msR^{-1}(\widetilde{\msK}\otimes 1) \widetilde{\msR}_{\tau} (1\otimes \widetilde{\msK}) = (1\otimes \widetilde{\msK})\widetilde{\msR}_{\tau,21}(\widetilde{\msK}\otimes 1)\msR_{21}^{-1}.
\end{equation}
Let $\widetilde{\mcO}_q(K \backslash U)$ the corresponding right coideal $*$-subalgebra of $\mcO_q(U)$. Let $v$ be the ribbon element as defined in \eqref{EqDefRibbon}. 

\begin{Prop}
Assume that $\nu$ is of symmetric type. Then any $*$-compatible $\nu$-modified universal $K$-matrix is invertible, and there is a one-to-one correspondence between $*$-compatible $\nu$-modified universal $K$-matrices and elements satisfying \eqref{EqKAlt}, the correspondence being given by 
\begin{equation}\label{EqCorrKs}
 \widetilde{\msK} = v^{-1}\msK^{-1}.
\end{equation}
Moreover, 
\begin{equation}\label{EqEqualCoid}
\mcO_q(K\backslash U) = \tau(\widetilde{\mcO}_q(K\backslash U)).
\end{equation}   
\end{Prop} 
\begin{proof}
If $\msK$ is a $*$-compatible $\nu$-modified universal $K$-matrix, put $\msK_{\epsilon} = \msE^{-1}\msK$. Then 
\begin{equation}\label{EqKeps}
\Delta(\msK_{\epsilon}) = \msR^{-1} (\msK_{\epsilon}\otimes 1)\msR_{\nu}(1\otimes \msK_{\epsilon}) = (1\otimes \msK_{\epsilon})\msR_{\nu,21}(\msK_{\epsilon}\otimes 1)\msR_{21}^{-1}. 
\end{equation}
As in \cite[Lemma 3.13]{KoSt09} one finds by applying $\msK_{\epsilon}$ to $S(a_{(1)})a_{(2)}$ and using \eqref{EqKeps} that $\overline{\msK}_{\epsilon}\in \mcU_q(\mfu)$, defined by
\[
(\overline{\msK}_{\epsilon},a) = \mbr(S(a_{(2)}),a_{(4)})(\msK_{\epsilon},S(a_{(2)})) \mbr_{\nu}(S(a_{(1)}),a_{(5)}),
\]
is a left inverse to $\msK_{\epsilon}$. Similarly one constructs a right inverse, proving invertibility of $\msK$. The same argument shows that any element satisfying \eqref{EqKAlt} is invertible. 

Using now the identities \eqref{EqCommDelt} and \eqref{EqPropRibbon}, the centrality of $v$ and the fact that, in the symmetric case, $\Omega_{\epsilon}=\Omega_{\epsilon}^{-1}$, one deduces that the correspondence \eqref{EqCorrKs} is well-defined. 

To see that \eqref{EqEqualCoid} holds, note that the same argument as in Proposition \eqref{PropInclusion} shows that 
\[
\widehattilde{\mcO}_q(K\backslash U) = \{X \in \mcU_q(\mfu) \mid (1\otimes \widetilde{\msK})(\id\otimes \tau)\Delta(X) = \Delta(X) (1\otimes \widetilde{\msK})\}.
\]
Using centrality of $v$, it is then immediate that
\[
\widehattilde{\mcO}_q(K\backslash U)  =  \{X \in \mcU_q(\mfu) \mid (1\otimes \msK)\Delta(X) =(\id\otimes \tau) \Delta(X) (1\otimes \msK)\},
\]
hence $\widehattilde{\mcO}_q(K\backslash U) = \widehat{\mcO}_q(K\backslash U)$ by Proposition \eqref{PropInclusion}, and then $\widetilde{\mcO}_q(K \backslash U) = \mcO_q(K\backslash U)$ by the biduality statement in Proposition \ref{PropDualCoideal}.
\end{proof}

\begin{Rem}
It was not clear to us how (or if) the above correspondence can be generalized to the non-symmetric case, as one no longer has invertibility of $\msK$.   
\end{Rem}

Let us now present a second variation. Let us for the moment identify $\mcO_q(\bar{G})$ with $\mcO_q(G)$ by the map $f^{\dag}\mapsto f^*$, and identify then further $\mcO_q^{\nu}(G_{\R})$ with $\mcO_q(G)\otimes \mcO_q(G)$ by applying this map to the second component.  In particular, we then have $(f\otimes g)^{\dag} = g^*\otimes f^*$. By  general twisting arguments,  $\mcO_q^{\nu}(G_{\R})$ is coquasitriangular with universal $\mbr$-form
\[
\mbr_{\nu,D} = \mbr_{\nu,14}\mbr_{13}\mbr_{24}\mbr_{\nu,32}^{-1},
\]
cf. \cite[Theorem 2.3.4 and Proposition 7.3.2]{Maj95}. Moreover, $\mbr_{\nu,D}$ is real in the sense that 
\[
(S\otimes S)(\mbr_{\nu,D})^{\dag} = \mbr_{\nu,D,21} \in \mcU_q(\mfu)^{\hat{\otimes}2} \hat{\otimes} \mcU_q(\mfu)^{\hat{\otimes}2}.
\] 
It follows that we can consider the $\mbr_{\nu,D}$-twisted $*$-algebra $\mcO_q^{\nu}(G_{\R})'$, obtained by endowing $\mcO_q^{\nu}(G_{\R})$ with the original $*$-structure and the new product 
\begin{equation}\label{EqDefTwistrD}
f\cdot g= (\mbr_{\nu,D},f_{(1)}\otimes g_{(1)})f_{(2)}g_{(2)},\qquad f,g\in \mcO_q^{\nu}(G_{\R}). 
\end{equation}
This $*$-algebra fits into a connected cogroupoid linking $\mcO_q^{\nu}(G_{\R})^{\opp}$ with $\mcO_q^{\nu}(G_{\R})$. By composition, we also obtain an  $\mbr_{\nu,D}$-twisted $*$-algebra $\mcO_q^{\nu,\id}(G_{\R})'$ with new product given again by \eqref{EqDefTwistrD}, but interpreting $f,g\in \mcO_q^{\nu}(G_{\R})$. This time the $*$-algebra $\mcO_q^{\nu,\id}(G_{\R})'$ fits into a connected cogroupoid linking $\mcO_q^{\nu}(G_{\R})^{\opp}$ with $\mcO_q(G_{\R})$. The $*$-algebra $\mcO_q^{\nu,\id}(G_{\R})'$ is the $\nu$-twisted Heisenberg double analogon of the $\nu$-twisted Drinfeld double $\mcO_q^{\nu,\id}(G_{\R})$, and corresponds to\footnote{In \cite{STS85,STS94} the constructions are carried out in the complex setting, without consideration of the $*$-structure. This allows one to consider a more general class of automorphisms than the involutive ones.}  the twisted  doubles considered in respectively the Poisson and quantum setting in \cite{STS85,STS94}. 

Let us show that this second variation is actually isomorphic to the first variation. We will need some preparations. 

\begin{Theorem}\label{TheoRCobound}
There exists an invertible element $t\in \mcU_q(\mfg)$ such that the following holds: $\Ad(t)$ is an algebra and anti-coalgebra homomorphism satisfying
\begin{equation}\label{EqAdt}
\Ad(t)(K_{\omega}) = K_{-\tau_0(\omega)},\qquad \Ad(t)(E_r) = -q_r^2F_{\tau_0(r)},\qquad \Ad(t)(F_r) = -q_r^{-2}E_{\tau_0(r)},
\end{equation}
and 
\begin{equation}\label{EqRCobound}
\msR = (t\otimes t)\Delta(t^{-1}) = \Delta^{\opp}(t^{-1})(t\otimes t).
\end{equation}
Moreover, 
\begin{equation}\label{EqSstart}
S(t)^* =t.
\end{equation} 
\end{Theorem}
\begin{proof}
Let $c\in \mcU_q(\mfu)$ be defined by $c\xi = q^{(\wt(\xi),\wt(\xi))/2}\xi$, and let $T_{w_0}'$ be the alternative Lusztig braid operator at the longest root constructed from the 
\[ 
T_r' \xi = \underset{-a+b-c = (\wt(\xi),\alpha_r^{\vee})}{\sum_{a,b,c\geq 0}} (-1)^b q_r^{ac-b} E_r^{(a)}F_r^{(b)}E_r^{(c)}\xi.
\]
Then from \cite[Lemma 3.10 and Theorem 3.11]{ST09} (see also \cite{KR90,LS91,KT09}) it follows that 
\[
t = cK_{-\rho}T_{w_0}' = cT_{w_0}'K_{\rho}
\] 
satisfies \eqref{EqAdt} and \eqref{EqRCobound}, upon noting that 
\begin{itemize}
\item the above references use the opposite comultiplication, 
\item their $R$-matrices hence correspond to our $\msR^{-1}$, 
\item in \cite[Definition 3.5]{ST09} one should correct the small typo by adding an extra sign in the expression under the summation sign and changing the root to the associated coroot, and
\item we have changed the appearence of $K_{\rho}$ in  \cite[Definition]{ST09} into $K_{-\rho}$ as this does not change \eqref{EqRCobound} but is important to have the right compatibility with the $*$-structure.
\end{itemize}

Finally, to see that $S(t)^* = t$ we first note that $S(c)^* = c$. Since we can write $K_{-\rho}T_{w_0}'$ as a product of the rank one operators $K_{-\alpha_r/2}T_{r}'$, it is hence sufficient to verify that $K_{-\alpha/2}T'$ is stable under $S(-)^*$ in the rank one case. However, since $S(X)^* = R(\Ad(K_{-\alpha/2})X)^*$, and since $K_{-\alpha/2}T' = T'K_{\alpha/2}$, this is equivalent to $R(T')^* = T'K_{2\alpha}$. This now follows similarly as in Lemma \ref{LemTstarR}.
\end{proof}

Consider now the corestricted left coaction of  $\mcO_q(G_{\nu})^{\opp}$ on $\mcO_q^{\nu,\id}(G_{\R})'$, and let $\mcO_q(G_{\nu}\dbbackslash G_{\R})'$ be the associated coinvariant $*$-subalgebra. Let $t$ be as in Theorem \ref{TheoRCobound}.

\begin{Prop}
The map 
\[
F: \mcO_q^{\nu,\id}(G_{\R})' \rightarrow \widetilde{\mcO}_q^{\nu,\id}(G_{\R}), f \mapsto f((t\otimes t)\msR_{\nu,21}^{-1}-)
\]
is a right $\mcO_q(G_{\R})$-equivariant $*$-isomorphism, carrying $\mcO_q(G_{\nu}\dbbackslash G_{\R})'$ isomorphically onto $\widetilde{\mcO}_q(G_{\nu}\dbbackslash G_{\R})$. 
\end{Prop}
\begin{proof}
Pulling the algebra structures back to $\mcO_q^{\com}(G_{\R}) \cong \mcO_q(G)\otimes \mcO_q(G)$, we need to compare the two products
\[
f\cdot g = (\mbr_{\nu,14} \mbr_{13}\mbr_{24},f_{(1)}\otimes g_{(1)}) f_{(2)}g_{(2)} (\mbr_{32}^{-1},f_{(3)}\otimes g_{(3)}),\qquad f,g\in \mcO_q^{\com}(G_{\R})
\]
of $\mcO_q^{\nu,\id}(G_{\R})'$ and
\[
f g = (\mbr_{\nu,23}^{-1},f_{(1)}\otimes g_{(1)}) f_{(2)}g_{(2)} (\mbr_{32}^{-1},f_{(3)}\otimes g_{(3)}),\qquad f,g\in \mcO_q^{\com}(G_{\R})
\]
of $\widetilde{\mcO}_q^{\nu,\id}(G_{\R})$ by means of $F$. This entails checking that, with $x = (t\otimes t)\msR_{\nu,21}^{-1}$,
\begin{equation}\label{EqIntermedR}
(x^{-1}\otimes x^{-1}) \msR_{\nu,14}\msR_{13}\msR_{24} \Delta^{\otimes}(x) = \msR_{\nu,23}^{-1},
\end{equation}
where $\Delta^{\otimes}$ is the tensor product coalgebra structure. Now using \eqref{EqAdt}, we easily see that
\[
(t^{-1}\otimes t^{-1})\msR_{\nu} = \msR_{\nu,21}(t^{-1}\otimes t^{-1}).
\]
By \eqref{EqRCobound} we can then simplify \eqref{EqIntermedR} to 
\[
\msR_{\nu,21}\msR_{\nu,43} \msR_{\nu,41} \Delta^{\otimes}(\msR_{\nu,21}^{-1}) = \msR_{\nu,23}^{-1}.
\]
However, an easy calculation shows that $\Delta^{\otimes}(\msR_{\nu,21}^{-1}) = \msR_{\nu,41}^{-1}\msR_{\nu,21}^{-1}\msR_{\nu,43}^{-1}\msR_{\nu,23}^{-1}$, proving the above identity.

This shows that $F$ is an algebra isomorphism, and it is right $\mcO_q(G_{\R})$-equivariant by construction. To see that $F$ is $*$-preserving, we need to show that $x$ as above satisfies $S(x)^{\dag}= x$, i.e.~
\[
(S\otimes S)( (t\otimes t)\msR_{\nu}^{-1})^{*\otimes *} =  (t\otimes t)\msR_{21,\nu}^{-1}. 
\]
This follows from $(S\otimes S)\msR_{\nu} =  \msR_{\nu}$ and $\msR_{\nu}^{*\otimes *} = \msR_{\nu,21}$ together with \eqref{EqSstart}.

Finally, to see that  $F(\mcO_q(G_{\nu}\dbbackslash G_{\R})') = \widetilde{\mcO}_q(G_{\nu}\dbbackslash G_{\R})$, it is enough to compare the right infinitesimal actions $\lhd$ of respectively $U_q(\mfg)$ and $U_q(\mfg)^{\cop}$ on  $ \widetilde{\mcO}_q(G_{\nu}\dbbackslash G_{\R})$ and $\mcO_q(G_{\nu}\dbbackslash G_{\R})')$. Transporting these actions along the natural vector space isomorphisms $ \widetilde{\mcO}_q(G_{\nu}\dbbackslash G_{\R}) \cong \mcO_q(G)\otimes \mcO_q(G) \cong \mcO_q(G_{\nu}\dbbackslash G_{\R})'$ and taking care of implementing correctly the pairings $(-,-)_{\epsilon}$ and $(-,-)_{\epsilon}'$, we compute for example for $f\in \mcO_q(G)\otimes \mcO_q(G)$ that 
\[
F(f)\lhd K_{\omega} = f(x(K_{\omega}\otimes K_{\omega})-) = f((K_{-\tau_0(\omega)}\otimes K_{-\tau_0(\omega)})x-) = F(f \lhd K_{-\tau_0(\omega)}). 
\]
Similarly, using again \eqref{EqAdt}, we find
\begin{eqnarray*}
F(f) \lhd E_r 
&=& f((t\otimes t)\msR_{\nu,21}^{-1}(E_r\otimes 1 + \epsilon_r K_r\otimes E_{\tau(r)})-)\\
&=& f((t\otimes t)(\epsilon_r1\otimes E_r + E_{r}\otimes K_{\tau(r)})\msR_{\nu,21}^{-1}-)\\
&=& f((-\epsilon_rq_r^21\otimes F_{\tau_0(r)} -q_r^2F_{\tau_0(r)} \otimes K_{-\tau_0\tau(r)})(t\otimes t)\msR_{\nu,21}^{-1}-)\\
&=& F(f\lhd (-q_r^2F_{\tau_0(r)})),
\end{eqnarray*}
where in the last step we use that $\tau$ and $\tau_0$ commute, cf.~ Lemma \ref{LemCommTau0Tau}. One similarly shows that $F(f) \lhd F_r =  F(f\lhd (-q_r^{-2}E_{\tau_0(r)}))$, from which the preservation of invariant subalgebras under $F$ then follows. 
\end{proof}

\section{Enhanced Satake diagrams and associated Vogan diagrams}\label{Ap1}

Let $\mfg$ be a semisimple complex Lie algebra with Dynkin diagram $\Gamma$ and compact form $\mfu$. We use further notation as in Section \ref{SecPrelim}. In particular, we endow $\mfg$ with the Lie $*$-algebra structure inducing $\mfu$.

Recall that two Lie algebra involutions $\sigma,\sigma'$ of $\mfu$ or, equivalently, two Lie $*$-algebra involutions of $\mfg$ are called \emph{inner equivalent} or \emph{inner conjugate} if there exists $g\in G$ with 
\[
\sigma'= \Ad(g) \sigma \Ad(g)^{-1}. 
\] 
It is not hard to see that one may always take $g \in U$, so that $\sigma,\sigma'$ are \emph{unitarily inner equivalent}. More generally, we call $\sigma,\sigma'$ \emph{equivalent} or \emph{conjugate} if there exists $\phi \in \Aut(\mfu) = \Aut(\mfg,*)$ such that 
\[
\sigma' = \phi \sigma \phi^{-1}. 
\]

Recall from Definition \ref{DefConcSat} the construction of involutions $\theta = \theta(X,\tau,z)$ starting from a concrete Satake diagram $(X,\tau,z)$. It is straightforward to check that $\theta(X,\tau,z)$ does not depend on $z$ up to unitary inner conjugacy by an element $\Ad(t)$ for $t\in T$, the maximal torus in $U$. We then have the following theorem.

\begin{Theorem}\label{TheoUniqueInnConj}
The assignment
\[
(X,\tau,z) \mapsto \theta(X,\tau,z)
\]
descends to a one-to-one correspondence between concrete Satake diagrams $(X,\tau)$ on $\Gamma$ and unitary inner conjugacy classes of Lie algebra involutions of $\mfu$.  
\end{Theorem}
\begin{proof}
It is well-known that any $*$-compatible involution is equivalent to a Satake involution up to conjugacy with an automorphism of $\mfu$ \cite{Ar62}. We have to show then that two concrete Satake diagrams induce \emph{inner} conjugate involutions if and only if the concrete Satake diagrams are equal. This follows from \cite{Hel88}*{Theorem 3.11}. 
\end{proof}

In fact, for the proof of the previous theorem we may clearly restrict to the case of $\mfg$ simple, and then only the cases of the Satake diagrams associated to $\mfu^*_{2p}(\mathbb{H}) = \mfso^*(4p)$ and $\mfso(1,7),\mfso(2,6)$ and $\mfso(3,5)$ need to be investigated, as they are the only ones admitting Dynkin diagram automorphisms which are not Satake diagram automorphisms. In Proposition \ref{PropInnConj} and Proposition \ref{PropInnConj2} we will show explicitly that these automorphisms induce non-inner equivalences by using instead the \emph{Vogan form} for the involutions. 

\begin{Def}
A $*$-preserving involution $\nu$ of $\mfg$ is said to be in \emph{Vogan form} with respect to the Chevalley-Serre generators $\msS$ if there exists an involutive Dynkin diagram automorphism $\tau$ and a $\tau$-invariant sign function
\[
\epsilon: I \rightarrow \{\pm1\}
\]
such that 
\begin{equation}\label{EqDefnu}
\nu(h_r) = h_{\tau(r)},\qquad \nu(e_r) = \epsilon_re_{\tau(r)},\qquad \nu(f_r) = \epsilon_rf_{\tau(r)}. 
\end{equation}
\end{Def}

Conversely, whenever $\tau$ is an involutive automorphism of the Dynkin diagram $\Gamma$ and $\epsilon$ is a $\tau$-invariant sign function on the underlying set $I$, we can define a $*$-compatible involution $\nu = \nu(Y,\tau)$ by \eqref{EqDefnu}, where we write $Y =Y_{\epsilon}$ for the set of points with $\epsilon_r = -1$. One can reduce to the case with $\epsilon_{r} = 1$ for $\tau(r)\neq r$, but it will be more natural not to make this reduction a priori. The datum $(Y,\tau)$ can be encoded on the Dynkin diagram by connecting $2$-point orbits of $\tau$ via arrows and coloring the $Y$-elements black. One calls the resulting diagram a \emph{concrete Vogan diagram}. 

It is well-known that any involution of $\mfu$ is inner conjugate to some some $\nu(Y,\tau)$, see e.g.~ \cite{Hel78}*{Chapter X} or \cite{Kna96}*{Chapter VI}. To see which $\nu(Y,\tau)$ are inner conjugate, we will use the following lemma. Note first that any sign function $\epsilon$ on $I$ can be extended uniquely to a $\{\pm1\}$-valued character on the root lattice $Q$. Given a subset $Z \subseteq I$, let us further write 
\[
\eta_Z: I \rightarrow \{\pm 1\},\qquad \left\{\begin{array}{ll} r \mapsto -1 & \textrm{if }r\in Z \\ r \mapsto 1 & \textrm{if }r\notin Z\end{array}\right.
\]

\begin{Lem}\label{LemEqVog}
Two Vogan involutions $\nu(Y,\tau)$ and $\nu(Y',\tau')$ with associated sign characters $\epsilon,\epsilon'$ are inner conjugate if and 
only if $\tau = \tau'$ and $\epsilon,\epsilon'$ are equivalent with respect to the smallest equivalence relation generated by the following two types of relations:
\begin{itemize}
\item type 1: $\epsilon \sim \epsilon' = \epsilon \circ s_{r}$ for $r\in I$ with $\tau(r) = r$ and $\epsilon(r) = -1$. 
\item type 2: $\epsilon \sim  \epsilon' = \eta_{\{r,\tau(r)\}}\cdot \epsilon$ for $r\in I$  with $\tau(r) \neq r$.
\end{itemize}
\end{Lem}
\begin{proof}
It is clear that (inner) conjugacy implies $\tau = \tau'$, so we will assume this in what follows. 

Fix $\epsilon$. If $\epsilon' = \eta_{\{r,\tau(r)\}}\cdot \epsilon$ for $\tau(r)\neq r$, we can pick $z\in T$ such that 
\[
 z(\alpha_{\tau(s)})z(\alpha_s)^{-1} = \eta_{\{r,\tau(r)\}}(s),\qquad  \textrm{for all }s\in I.
\]
Then $\nu(Y',\tau)  = \Ad(z) \nu(Y,\tau)\Ad(z)^{-1}$. On the other hand, if $\epsilon' = \epsilon \circ s_{r}$ for $\tau(r) = r$ and $\epsilon(r) = -1$, we have, using the notation \eqref{EqInvSimple}, that $\tau(m_{r}) = m_{r}$, and it is then easy to see that 
\[
\Ad(m_{r})\nu \Ad(m_{r})^{-1} = \nu'.
\]
Hence $\epsilon \sim \epsilon'$ implies $\nu(Y,\tau) \sim \nu(Y',\tau)$. 

Conversely, assume that $\nu(Y,\tau) \sim \nu(Y',\tau)$. We may assume that $\mfg$ is simple. If $\tau = \id$, we only need to use the first operation, and the result follows from\footnote{Note that the proof of \cite{CH02}*{Theorem 5.1} is with respect to equivalence by \emph{inner conjugacy}, although the authors introduce the equivalence relation as being by general conjugacy with an automorphism.} \cite{CH02}*{Theorem 5.1}. If $\tau \neq \id$, we treat the three cases $A,D,E$ separately, following the arguments in \cite{CH02}*{Section 4}. 

For the $A$-case, it is clear that any two inner equivalent $\nu,\nu'$ must have diagrams related by an equivalence of type 2. 

\begin{table}[ht]
\begin{center}
\bgroup
\def\arraystretch{3}
{\setlength{\tabcolsep}{1.5em}
\begin{tabular}{ccc}
\begin{tikzpicture}[scale=.4,baseline]
\node (v1) at (0,0) {};
\node (v2) at (2,0) {};
\node (v3) at (4,0) {};
\node (v4) at (6,0) {};
\draw (0cm,0) circle (.2cm);
\draw (0.2cm,0) -- +(0.2cm,0);
\draw[dotted] (0.4cm,0) --+ (1cm,0);
\draw (1.6cm,0) --+ (0.2cm,0);
\draw (2cm,0) circle (.2cm);
\draw (2.2cm,0) -- +(0.6cm,0);
\draw[fill=black] (3cm,0) circle (.2cm);
\draw (3.2cm,0) -- +(0.6cm,0);
\draw (4cm,0) circle (.2cm);
\draw (4.2cm,0) -- +(0.2cm,0);
\draw[dotted] (4.4cm,0) --+ (1cm,0);
\draw (5.6cm,0) --+ (0.2cm,0);
\draw (6cm,0) circle (.2cm); 

\draw[<->]
(v1) edge[bend left] (v4);
\draw[<->]
(v2) edge[bend left=50] (v3);
\end{tikzpicture}
& 
\begin{tikzpicture}[scale=.4,baseline]
\node (v1) at (0,0) {};
\node (v2) at (1.8,0) {};
\node (v3) at (3.2,0) {};
\node (v4) at (5,0) {};
\draw(0cm,0) circle (.2cm);
\draw (0.2cm,0) -- +(0.2cm,0);
\draw[dotted] (0.4cm,0) --+ (1cm,0);
\draw (1.6cm,0) --+ (0.2cm,0);
\draw (2cm,0) circle (.2cm);
\draw (2.2cm,0) -- +(0.6cm,0);
\draw(3cm,0) circle (.2cm);
\draw (3.2cm,0) -- +(0.2cm,0);
\draw[dotted] (3.4cm,0) --+ (1cm,0);
\draw (4.6cm,0) --+ (0.2cm,0);
\draw (5cm,0) circle (.2cm); 
\draw[<->]
(v1) edge[bend left] (v4);
\draw[<->]
(v2) edge[bend left=70] (v3);
\end{tikzpicture}
&
\begin{tikzpicture}[scale=.4,baseline]
\node (v1) at (0,0) {};
\node (v2) at (2,0) {};
\node (v3) at (4,0) {};
\node (v4) at (6,0) {};

\draw (0cm,0) circle (.2cm);
\draw (0.2cm,0) -- +(0.2cm,0);
\draw[dotted] (0.4cm,0) --+ (1cm,0);
\draw (1.6cm,0) --+ (0.2cm,0);
\draw (2cm,0) circle (.2cm);
\draw (2.2cm,0) -- +(0.6cm,0);
\draw (3cm,0) circle (.2cm);
\draw (3.2cm,0) -- +(0.6cm,0);
\draw (4cm,0) circle (.2cm);
\draw (4.2cm,0) -- +(0.2cm,0);
\draw[dotted] (4.4cm,0) --+ (1cm,0);
\draw (5.6cm,0) --+ (0.2cm,0);
\draw (6cm,0) circle (.2cm); 

\draw[<->]
(v1) edge[bend left] (v4);
\draw[<->]
(v2) edge[bend left=50] (v3);
\end{tikzpicture}

\end{tabular}
}
\egroup
\end{center}
\caption{Equivalence classes for Vogan diagrams of type $A$ with non-trivial automorphism}\label{TableAInv}
\end{table}

For the $E$-case, there is only $E_6$ to consider. It is easy to check directly in this case that the equivalence relation on the possible signs creates two orbits, which correspond precisely to the two choices of real forms.

\begin{table}[ht]
\begin{center}
\bgroup
\def\arraystretch{3}
{\setlength{\tabcolsep}{1.5em}
\begin{tabular}{cc}
 \begin{tikzpicture}[scale=.4,baseline=-6pt]
\node (v1) at (0,0.2) {};
\node (v2) at (4,0.2) {};
\node (v3) at (1,0.2) {};
\node (v4) at (3,0.2) {};
\draw (0cm,0) circle (.2cm);
\draw (0.2cm,0) -- +(0.6cm,0);
\draw (1cm,0) circle (.2cm);
\draw (1.2cm,0) -- +(0.6cm,0);
\draw (2cm,0) circle (.2cm);
\draw (2.2cm,0) -- +(0.6cm,0);
\draw (3cm,0) circle (.2cm);
\draw (3.2cm,0) -- +(0.6cm,0);
\draw (4cm,0) circle (.2cm);
\draw (2cm,-0.2) -- +(0cm,-0.6);
\draw[fill = black] (2cm,-1) circle (.2cm);
\draw[<->]
(v1) edge[bend left] (v2);
\draw[<->]
(v3) edge[bend left] (v4);
\end{tikzpicture}
&
 \begin{tikzpicture}[scale=.4,baseline=-6pt]
\node (v1) at (0,0.2) {};
\node (v2) at (4,0.2) {};
\node (v3) at (1,0.2) {};
\node (v4) at (3,0.2) {};
\draw (0cm,0) circle (.2cm);
\draw (0.2cm,0) -- +(0.6cm,0);
\draw (1cm,0) circle (.2cm);
\draw (1.2cm,0) -- +(0.6cm,0);
\draw (2cm,0) circle (.2cm);
\draw (2.2cm,0) -- +(0.6cm,0);
\draw (3cm,0) circle (.2cm);
\draw (3.2cm,0) -- +(0.6cm,0);
\draw (4cm,0) circle (.2cm);
\draw (2cm,-0.2) -- +(0cm,-0.6);
\draw (2cm,-1) circle (.2cm);
\draw[<->]
(v1) edge[bend left] (v2);
\draw[<->]
(v3) edge[bend left] (v4);
\end{tikzpicture}
\end{tabular}
}
\egroup
\end{center}
\caption{Equivalence classes for Vogan diagrams of type $E$ with non-trivial automorphism}\label{TableAInv}
\end{table}

The same argument as in \cite{CH02}*{Section 4} can be used in the $D$-case, whereby the operation of the first kind can be used to reduce to the case of a single painted $\tau$-fixed vertex. If the single painted vertex is at position $p$, the Vogan diagram corresponds to the symmetric pair $\mfso(2p+1)\times \mfso(2q+1) \subseteq \mfso(2p+2q+2) = \mfso(2l)$, see e.g.~ \cite{Kna96}*{Appendix C.3}. The only thing left to prove is then that the two sign functions associated to the diagrams with non-trivial automorphism and a single painted vertex either at $p$ or $q=l-p-1$ (for $p\leq l-2$) are equivalent,
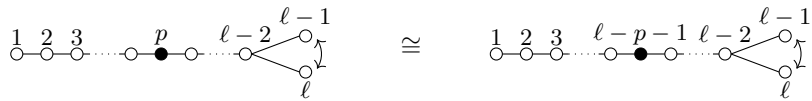
\[
 \begin{tikzpicture}[scale=.4,baseline]
\node (v1) at (9.8,0.8) {};
\node (v2) at (9.8,-0.8) {};
\draw (0cm,0) circle (.2cm) node[above]{\small $1$} ;
\draw (0.2cm,0) -- +(0.6cm,0);
\draw (1cm,0) circle (.2cm)node[above]{\small $2$};
\draw (1.2cm,0) -- +(0.6cm,0);
\draw (2cm,0) circle (.2cm)node[above]{\small $3$};
\draw (2.2cm,0) -- +(0.2cm,0);
\draw[dotted] (2.4cm,0) --+ (1cm,0);
\draw (3.4cm,0) --+ (0.2cm,0);
\draw (3.8cm,0) circle (.2cm); 
\draw (4cm,0) --+ (0.6cm,0);
\draw[fill = black] (4.8cm,0) circle (.2cm)  node[above]{\small $p$};
\draw (5cm,0) --+ (0.6cm,0);
\draw (5.8cm,0) circle (.2cm);
\draw (6cm,0) -- +(0.2cm,0);
\draw[dotted] (6.2cm,0) --+ (1cm,0);
\draw (7.2cm,0) --+ (0.2cm,0);
\draw (7.6cm,0) circle (.2cm)  node[above]{\small $\ell-2$};
\draw (7.8cm,0) --+ (1.6,0.6);
\draw (7.8cm,0) --+ (1.6,-0.6);
\draw (9.6cm,0.6) circle (.2cm) node[above]{\small $\ell-1$} ;
\draw (9.6cm,-0.6) circle (.2cm) node[below]{\small $\ell$} ;
\draw[<->]
(v1) edge[bend left] (v2);
\end{tikzpicture}
\qquad
\cong 
\qquad 
 \begin{tikzpicture}[scale=.4,baseline]
\node (v1) at (9.8,0.8) {};
\node (v2) at (9.8,-0.8) {};
\draw (0cm,0) circle (.2cm) node[above]{\small $1$} ;
\draw (0.2cm,0) -- +(0.6cm,0);
\draw (1cm,0) circle (.2cm)node[above]{\small $2$};
\draw (1.2cm,0) -- +(0.6cm,0);
\draw (2cm,0) circle (.2cm)node[above]{\small $3$};
\draw (2.2cm,0) -- +(0.2cm,0);
\draw[dotted] (2.4cm,0) --+ (1cm,0);
\draw (3.4cm,0) --+ (0.2cm,0);
\draw (3.8cm,0) circle (.2cm); 
\draw (4cm,0) --+ (0.6cm,0);
\draw[fill = black] (4.8cm,0) circle (.2cm)  node[above]{\small $\ell-p-1$};
\draw (5cm,0) --+ (0.6cm,0);
\draw (5.8cm,0) circle (.2cm);
\draw (6cm,0) -- +(0.2cm,0);
\draw[dotted] (6.2cm,0) --+ (1cm,0);
\draw (7.2cm,0) --+ (0.2cm,0);
\draw (7.6cm,0) circle (.2cm)  node[above]{\small $\ell-2$};
\draw (7.8cm,0) --+ (1.6,0.6);
\draw (7.8cm,0) --+ (1.6,-0.6);
\draw (9.6cm,0.6) circle (.2cm) node[above]{\small $\ell-1$} ;
\draw (9.6cm,-0.6) circle (.2cm) node[below]{\small $\ell$} ;
\draw[<->]
(v1) edge[bend left] (v2);
\end{tikzpicture}
\]
This follows by an easy direct verification (using also the type 2 equivalence!). 
\end{proof}

It is known from the Borel-de Siebenthal theorem \cite{Kna96}*{Theorem 6.96} that any Vogan involution is equivalent by conjugation with an automorphism to a Vogan involution coming from a diagram with at most one painted vertex. Moreover, as we have seen there is no distinction between conjugation by an automorphism and conjugation by an inner automorphism, except in the case corresponding to the real forms $\mfso^*(4p)$ or real forms of $\mfso(8)$. In these cases, we have the following. 

\begin{Prop}\label{PropInnConj}
Consider the Vogan diagrams $(\{l\},\id)$ and $(\{l-1\},\id)$ on $D_{l}$ for $l$ even,
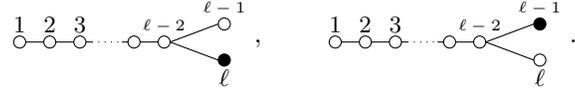
\[
  \begin{tikzpicture}[scale=.4,baseline]
\draw(0cm,0) circle (.2cm) node[above]{\small $1$} ;
\draw (0.2cm,0) -- +(0.6cm,0);
\draw (1cm,0) circle (.2cm)node[above]{\small $2$};
\draw (1.2cm,0) -- +(0.6cm,0);
\draw (2cm,0) circle (.2cm)node[above]{\small $3$};
\draw (2.2cm,0) -- +(0.2cm,0);
\draw[dotted] (2.4cm,0) --+ (1cm,0);
\draw (3.4cm,0) --+ (0.2cm,0);
\draw (3.8cm,0) circle (.2cm); 
\draw (4cm,0) --+ (0.6cm,0);
\draw (4.8cm,0) circle (.2cm)  node[above]{\tiny $\ell-2$};
\draw (5cm,0) --+ (1.6,0.6);
\draw (5cm,0) --+ (1.6,-0.6);
\draw (6.8cm,0.6) circle (.2cm) node[above]{\tiny $\ell-1$} ;
\draw[fill = black] (6.8cm,-0.6) circle (.2cm) node[below]{\small $\ell$} ;
\end{tikzpicture},\qquad
  \begin{tikzpicture}[scale=.4,baseline]
\draw(0cm,0) circle (.2cm) node[above]{\small $1$} ;
\draw (0.2cm,0) -- +(0.6cm,0);
\draw (1cm,0) circle (.2cm)node[above]{\small $2$};
\draw (1.2cm,0) -- +(0.6cm,0);
\draw (2cm,0) circle (.2cm)node[above]{\small $3$};
\draw (2.2cm,0) -- +(0.2cm,0);
\draw[dotted] (2.4cm,0) --+ (1cm,0);
\draw (3.4cm,0) --+ (0.2cm,0);
\draw (3.8cm,0) circle (.2cm); 
\draw (4cm,0) --+ (0.6cm,0);
\draw (4.8cm,0) circle (.2cm)  node[above]{\tiny $\ell-2$};
\draw (5cm,0) --+ (1.6,0.6);
\draw (5cm,0) --+ (1.6,-0.6);
\draw[fill = black] (6.8cm,0.6) circle (.2cm) node[above]{\tiny $\ell-1$} ;
\draw (6.8cm,-0.6) circle (.2cm) node[below]{\small $\ell$} ;
\end{tikzpicture}.
\]

 Then $\nu(\{l\},\id)$ and $\nu(\{l-1\},\id)$ are equivalent but not inner equivalent. 
\end{Prop}
\begin{proof}
As in the proof of Lemma \ref{LemEqVog}, we have that the two diagrams are inner equivalent if and only if the associated sign functions satisfy $\epsilon = \epsilon'\circ w$ for $w\in W$. But consider the value 
\[
c = \epsilon(\alpha_1+ \alpha_3 + \ldots + \alpha_{l-1}).
\]
Then it is easily seen that $c$ is the same value on the whole of $W\epsilon$. However, $c$ differs for the two choices of Vogan diagrams. 
\end{proof}

\begin{Prop}\label{PropInnConj2}
There are nine inner equivalence classes for Vogan diagrams of $\mfso(8)$, obtained by rotations of the following three cases:
\begin{enumerate}
\item $\mfso(1,7)$: $Y= \emptyset$, $\tau(3) = 4$,
\item $\mfso(2,6)$: $Y \in \{\{1\},\{1,2\},\{3,4\},\{2,3,4\}\}$, $\tau =\id$,
\item $\mfso(3,5)$: $Y \in \{\{1\},\{1,2\},\{3,4\},\{2,3,4\}\}$, $\tau(3) = 4$.
\end{enumerate}

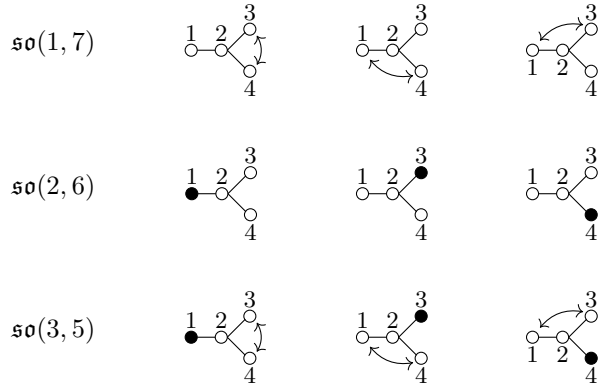
\begin{table}[ht]
\caption{Vogan diagrams for $\mfso(p,8-p)$ with $1\leq p \leq 3$ with at most one colored vertex}\label{Tableso8Vogan}
\begin{center}
\bgroup
\def\arraystretch{4}
{\setlength{\tabcolsep}{1.5em}
\begin{tabular}{cccc}
$\mfso(1,7)$
& 
\begin{tikzpicture}[scale=.4,baseline]
\node (v1) at (2,0.8) {};
\node (v2) at (2,-0.8) {};
\draw (0cm,0) circle (.2cm) node[above]{\small $1$} ;
\draw (0.2cm,0) -- +(0.6cm,0);
\draw (1cm,0) circle (.2cm) node[above]{\small $2$} ;
\draw (1.2cm,0) --+ (0.6cm,0.6cm);
\draw (1.2cm,0) --+ (0.6cm,-0.6cm);
\draw (1.95cm,0.7) circle (.2cm) node[above] {\small $3$}; 
\draw (1.95cm,-0.7) circle (.2cm)  node[below]{\small $4$};
\draw[<->]
(v1) edge[bend left] (v2);
\end{tikzpicture}
& 
\begin{tikzpicture}[scale=.4,baseline]
\node (v1) at (0,0) {};
\node (v2) at (2,-0.8) {};
\draw (0cm,0) circle (.2cm) node[above]{\small $1$} ;
\draw (0.2cm,0) -- +(0.6cm,0);
\draw (1cm,0) circle (.2cm) node[above]{\small $2$} ;
\draw (1.2cm,0) --+ (0.6cm,0.6cm);
\draw (1.2cm,0) --+ (0.6cm,-0.6cm);
\draw (1.95cm,0.7) circle (.2cm) node[above] {\small $3$}; 
\draw (1.95cm,-0.7) circle (.2cm)  node[below]{\small $4$};
\draw[<->]
(v1) edge[bend right](v2);
\end{tikzpicture}
& 
\begin{tikzpicture}[scale=.4,baseline]
\node (v1) at (0,0) {};
\node (v2) at (2,0.8) {};
\draw (0cm,0) circle (.2cm) node[below]{\small $1$} ;
\draw (0.2cm,0) -- +(0.6cm,0);
\draw (1cm,0) circle (.2cm) node[below]{\small $2$} ;
\draw (1.2cm,0) --+ (0.6cm,0.6cm);
\draw (1.2cm,0) --+ (0.6cm,-0.6cm);
\draw (1.95cm,0.7) circle (.2cm) node[above] {\small $3$}; 
\draw (1.95cm,-0.7) circle (.2cm)  node[below]{\small $4$};
\draw[<->]
(v1) edge[bend left](v2);
\end{tikzpicture}
\\ 
$\mfso(2,6)$  
& 
\begin{tikzpicture}[scale=.4,baseline]
\draw[fill = black] (0cm,0) circle (.2cm) node[above]{\small $1$} ;
\draw (0.2cm,0) -- +(0.6cm,0);
\draw (1cm,0) circle (.2cm) node[above]{\small $2$} ;
\draw (1.2cm,0) --+ (0.6cm,0.6cm);
\draw (1.2cm,0) --+ (0.6cm,-0.6cm);
\draw (1.95cm,0.7) circle (.2cm) node[above] {\small $3$}; 
\draw (1.95cm,-0.7) circle (.2cm)  node[below]{\small $4$};
\end{tikzpicture}
& 
\begin{tikzpicture}[scale=.4,baseline]
\draw (0cm,0) circle (.2cm) node[above]{\small $1$} ;
\draw (0.2cm,0) -- +(0.6cm,0);
\draw (1cm,0) circle (.2cm) node[above]{\small $2$} ;
\draw (1.2cm,0) --+ (0.6cm,0.6cm);
\draw (1.2cm,0) --+ (0.6cm,-0.6cm);
\draw[fill = black] (1.95cm,0.7) circle (.2cm) node[above] {\small $3$}; 
\draw (1.95cm,-0.7) circle (.2cm)  node[below]{\small $4$};
\end{tikzpicture}
& 
\begin{tikzpicture}[scale=.4,baseline]
\draw (0cm,0) circle (.2cm) node[above]{\small $1$} ;
\draw (0.2cm,0) -- +(0.6cm,0);
\draw (1cm,0) circle (.2cm) node[above]{\small $2$} ;
\draw (1.2cm,0) --+ (0.6cm,0.6cm);
\draw (1.2cm,0) --+ (0.6cm,-0.6cm);
\draw (1.95cm,0.7) circle (.2cm) node[above] {\small $3$}; 
\draw[fill = black] (1.95cm,-0.7) circle (.2cm)  node[below]{\small $4$};
\end{tikzpicture}\\
$\mfso(3,5)$ 
&
\begin{tikzpicture}[scale=.4,baseline]
\node (v1) at (2,0.8) {};
\node (v2) at (2,-0.8) {};
\draw[fill = black] (0cm,0) circle (.2cm) node[above]{\small $1$} ;
\draw (0.2cm,0) -- +(0.6cm,0);
\draw (1cm,0) circle (.2cm) node[above]{\small $2$} ;
\draw (1.2cm,0) --+ (0.6cm,0.6cm);
\draw (1.2cm,0) --+ (0.6cm,-0.6cm);
\draw (1.95cm,0.7) circle (.2cm) node[above] {\small $3$}; 
\draw (1.95cm,-0.7) circle (.2cm)  node[below]{\small $4$};
\draw[<->]
(v1) edge[bend left] (v2);
\end{tikzpicture}
& 
\begin{tikzpicture}[scale=.4,baseline]
\node (v1) at (0,0) {};
\node (v2) at (2,-0.8) {};
\draw (0cm,0) circle (.2cm) node[above]{\small $1$} ;
\draw (0.2cm,0) -- +(0.6cm,0);
\draw (1cm,0) circle (.2cm) node[above]{\small $2$} ;
\draw (1.2cm,0) --+ (0.6cm,0.6cm);
\draw (1.2cm,0) --+ (0.6cm,-0.6cm);
\draw[fill = black] (1.95cm,0.7) circle (.2cm) node[above] {\small $3$}; 
\draw (1.95cm,-0.7) circle (.2cm)  node[below]{\small $4$};
\draw[<->]
(v1) edge[bend right](v2);
\end{tikzpicture}
& 
\begin{tikzpicture}[scale=.4,baseline]
\node (v1) at (0,0) {};
\node (v2) at (2,0.8) {};
\draw (0cm,0) circle (.2cm) node[below]{\small $1$} ;
\draw (0.2cm,0) -- +(0.6cm,0);
\draw (1cm,0) circle (.2cm) node[below]{\small $2$} ;
\draw (1.2cm,0) --+ (0.6cm,0.6cm);
\draw (1.2cm,0) --+ (0.6cm,-0.6cm);
\draw (1.95cm,0.7) circle (.2cm) node[above] {\small $3$}; 
\draw[fill = black] (1.95cm,-0.7) circle (.2cm)  node[below]{\small $4$};
\draw[<->]
(v1) edge[bend left](v2);
\end{tikzpicture}
\\
\end{tabular}
}
\egroup
\end{center}
\end{table}

\end{Prop}
\begin{proof}
It follows by an immediate verification on the Dynkin diagram $D_4$ by means of Lemma \ref{LemEqVog}. 
\end{proof}

\begin{Def}\label{DefCompSatVog}
We call a concrete Satake diagram $(X,\tau)$ and a concrete Vogan diagram $(Y,\tau')$ \emph{compatible} if $\theta(X,\tau)$ and $\nu(Y,\tau')$ are inner conjugate. 
\end{Def}

The following lemma is clear by \eqref{EqCompInnOut}. 

\begin{Lem}
If  a concrete Satake diagram $(X,\tau)$ and a concrete Vogan diagram $(Y,\tau')$ are compatible, then $\tau' =\tau\tau_0$. 
\end{Lem}

We can hence reformulate Definition \ref{DefCompSatVog} as follows. Recall that if $\epsilon$ is a sign-function, we denote by $Y_{\epsilon}$ the set of points with value $-1$. 

\begin{Def}\label{DefTauAdmi}
Let $(X,\tau)$ be a concrete Satake diagram. We call \emph{$(X,\tau)$-admissible sign function} $\epsilon: I \rightarrow \{\pm 1\}$ any sign function which is $\tau \tau_0$-invariant and such that $\nu(Y_{\epsilon},\tau\tau_0)$ is inner conjugate to $\theta(X,\tau)$. We call two sign functions $\epsilon,\epsilon'$ \emph{equivalent} if they are $(X,\tau)$-admissible for the same Satake diagram $(X,\tau)$. 
\end{Def}

One can find at least one $(X,\tau)$-admissible sign function by comparing the classifications of involutions in terms of Satake diagrams on the one hand, and of special Vogan diagrams with at most one painted root on the other. Again, one can use standard tables to look up the equivalence, but we need to know more specifically the correspondence up to inner conjugacy in the case of $\mfso^*(4p)$. 

\begin{Lem}
Consider the concrete Satake diagram $(X,\tau)$ corresponding to $\mfso^*(4p)$ with the $2p-1$th root painted as in Table \ref{Tablesostar}. Then $\theta = \theta(X,\tau)$ is inner conjugate to $\nu(\{2p\},\id)$. 
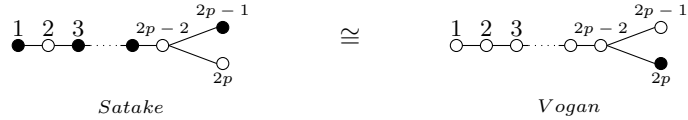
\begin{table}[ht]
\begin{center}
\bgroup
\def\arraystretch{3}
{\setlength{\tabcolsep}{1.5em}
\begin{tabular}{ccc}
$\underset{Satake}{
\begin{tikzpicture}[scale=.4,baseline]
\draw[fill = black] (0cm,0) circle (.2cm) node[above]{\small $1$} ;
\draw (0.2cm,0) -- +(0.6cm,0);
\draw (1cm,0) circle (.2cm)node[above]{\small $2$};
\draw (1.2cm,0) -- +(0.6cm,0);
\draw[fill = black] (2cm,0) circle (.2cm)node[above]{\small $3$};
\draw (2.2cm,0) -- +(0.2cm,0);
\draw[dotted] (2.4cm,0) --+ (1cm,0);
\draw (3.4cm,0) --+ (0.2cm,0);
\draw[fill = black] (3.8cm,0) circle (.2cm); 
\draw (4cm,0) --+ (0.6cm,0);
\draw (4.8cm,0) circle (.2cm)  node[above]{\tiny $2p-2$};
\draw (5cm,0) --+ (1.6,0.6);
\draw (5cm,0) --+ (1.6,-0.6);
\draw[fill = black] (6.8cm,0.6) circle (.2cm) node[above]{\tiny $2p-1$} ;
\draw (6.8cm,-0.6) circle (.2cm) node[below]{\tiny $2p$} ;
\end{tikzpicture}
}$
& 
$\cong$
&$\underset{Vogan}{
\begin{tikzpicture}[scale=.4,baseline]
\draw(0cm,0) circle (.2cm) node[above]{\small $1$} ;
\draw (0.2cm,0) -- +(0.6cm,0);
\draw (1cm,0) circle (.2cm)node[above]{\small $2$};
\draw (1.2cm,0) -- +(0.6cm,0);
\draw (2cm,0) circle (.2cm)node[above]{\small $3$};
\draw (2.2cm,0) -- +(0.2cm,0);
\draw[dotted] (2.4cm,0) --+ (1cm,0);
\draw (3.4cm,0) --+ (0.2cm,0);
\draw (3.8cm,0) circle (.2cm); 
\draw (4cm,0) --+ (0.6cm,0);
\draw (4.8cm,0) circle (.2cm)  node[above]{\tiny $2p-2$};
\draw (5cm,0) --+ (1.6,0.6);
\draw (5cm,0) --+ (1.6,-0.6);
\draw (6.8cm,0.6) circle (.2cm) node[above]{\tiny $2p-1$} ;
\draw[fill = black] (6.8cm,-0.6) circle (.2cm) node[below]{\tiny $2p$} ;
\end{tikzpicture}
}$
\end{tabular}
}
\egroup
\end{center}
\end{table}

\end{Lem} 
\begin{proof}
Let us realize $\mfso(4p)$ concretely on $\C^{4p}$ with basis $\{e_k\}$. We have that 
\[
\theta =\Ad(m_0 m_X).
\]
We can easily see that for a constant sign $c$
\[
m_0 e_k = c(-1)^k e_k.
\]
On the other hand, 
\[
m_X= m_1m_3\ldots m_{2p-1}
\]
 is a diagonal block matrix with constant blocks in $M_4(\C)$. It follows that $\theta$ can be put into Vogan form inside $SO(4)\times \ldots \times SO(4)$. Looking at the bottom block, we see that this corresponds to the transformation of the Satake diagram of $SO(4) \sim SU(2) \times SU(2)$, with the top vertex colored, into Vogan form, with the lower vertex colored. 
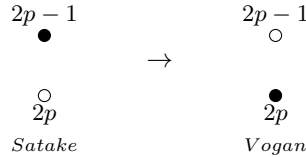
\[
\underset{Satake}
{
\begin{tikzpicture}[scale=.4,baseline=-12pt]
\draw[fill = black] (0cm,0) circle (.2cm) node[above]{\small $2p-1$} ;
\draw (0cm,-2) circle (.2cm) node[below]{\small $2p$};
\end{tikzpicture}
}
\qquad 
\rightarrow
\qquad 
\underset{Vogan}
{
\begin{tikzpicture}[scale=.4,baseline=-12pt]
\draw (0cm,0) circle (.2cm) node[above]{\small $2p-1$} ;
\draw[fill = black] (0cm,-2) circle (.2cm) node[below]{\small $2p$};
\end{tikzpicture}
}
\]
As the same happens in each copy, and only the simple roots of the form $e_i-e_{i+1}$ appear in the remainder of the Dynkin diagram, it follows that the complete Vogan diagram is the one described in the lemma. 
\end{proof}

We will be interested in verifying a certain compatibility between the sign function of a Vogan diagram and a particular sign function constructed from an equivalent Satake diagram.

Fix $\chi_0$ as in Lemma \ref{LemChoiceInvExt}, and recall the notations $\mcS_0,\mcS_X,\widetilde{z}$ and $\widetilde{z}_{\tau}$ from \eqref{EqDefS}, \eqref{EqDefSX}, Lemma \ref{LemChoiceInvExt} and \eqref{EqDefzTau}.

\begin{Theorem}\label{TheoremSignEps}
Let $(X,\tau,z)$ be an enhanced Satake diagram, and let $\epsilon$ be an $(X,\tau)$-compatible sign function. Then there exists an extension $\tilde{\epsilon}\in T$ of $\epsilon$  such that 
\begin{equation}\label{EqSol2}
\tilde{\epsilon}\tau\tau_0(\tilde{\epsilon})=\mathcal{S}_0\mathcal{S}_X\tilde{z}\tilde{z}_{\tau}^{-1}.
\end{equation}
\end{Theorem}
\begin{proof}
We may assume that $\mfg$ simple. Note that by \eqref{EqProdzz}, the right hand side of \eqref{EqSol2} lies in the center $\mathscr{Z}(U) = \Char(P/Q)$, while by Lemma \ref{LemCommTau0Tau} and the choice of $\widetilde{z}$ we have that the right hand side of \eqref{EqSol2} is $\tau\tau_0$-invariant. It is then easily seen that if $\epsilon$ admits an extension $\tilde{\epsilon}$ satisfying \eqref{EqSol2}, any $\epsilon' = s_r(\epsilon)$, for $\tau(r) = r$ and $\epsilon_r =-1$, admits the extension $\tilde{\epsilon}' = s_r(\tilde{\epsilon})$ satisfying \eqref{EqSol2}. It is also easy to see that the existence of an extension is stable under an equivalence of type 2 in  Lemma \ref{LemEqVog}, since we can extend any $\eta_{\{r,\tau\tau_0(r)\}}$ for $r\neq \tau\tau_0(r)$ to an element $\tilde{\eta} \in T$ with $\tilde{\eta}\tau\tau_0(\tilde{\eta}) = 1$: for example, choosing $n\in \N$ such that $\frac{1}{n}Q \supseteq P \supseteq Q$, and choosing a fundamental domain for the $\tau\tau_0$-action on $I$, we can put $\widetilde{\eta}\left(\frac{1}{n}\alpha_r\right) = 1$ for $r$ fixed under $\tau\tau_0$ and 
\[
\widetilde{\eta}\left(\frac{1}{n}\alpha_r\right) = e^{\pi i/n},\qquad \widetilde{\eta}\left(\frac{1}{n}\alpha_{\tau\tau_0(r)}\right) = e^{-\pi i/n}
\]
for $r\neq \tau\tau_0(r)$ in the fundamental domain. By Lemma \ref{LemEqVog} we may hence restrict to the case of a reduced Vogan diagram with at most one painted root. By applying an automorphism, we may also assume that the Satake diagram corresponds to the standard presentation in, say,  \cite{OV90}*{Reference Chapter, Section 2}. 

For $A$ an abelian group we write $\Char(A)$ for the characters $A \rightarrow U(1)$. Consider the group homomorphism
\[
\pi: P \times Q \rightarrow P,\qquad (\omega,\alpha) \mapsto \omega + \tau\tau_0(\omega) + \alpha. 
\]
Then $\pi$ dualizes to a homomorphism
\[
\hat{\pi}: \Char(P) \rightarrow \Char(P)\times \Char(Q),\qquad \rho \mapsto (\rho\tau\tau_0(\rho),\rho_{\mid Q}).
\]
Let us write 
\[
\eta: P \rightarrow U(1),\qquad \omega \mapsto \mathcal{S}_0\mathcal{S}_X\tilde{z}\tilde{z}_{\tau}^{-1}(\omega) =  e^{2\pi i (\rho^{\vee} + \rho_X^{\vee} + \chi_0 -\tau(\chi_0),\omega)}.
\]
Then we are to show that $(\eta,\epsilon)$ lies in the range of $\hat{\pi}$. This is equivalent with $(\eta,\epsilon)$ vanishing on $\Ker(\pi)$. Now since $\chi_0$ is $\tau\tau_0$-invariant, this means that we have to check
\begin{equation}\label{EqToCheck}
\epsilon(\alpha)  = e^{\pi i (\rho^{\vee} + \rho_X^{\vee} + \chi_0 -\tau(\chi_0),\alpha)}, \qquad\forall \alpha \in Q',
\end{equation}
where
\[
Q' = \{\alpha \in Q \mid \exists \omega \in P \textrm{ such that } \alpha  =  \omega + \tau\tau_0(\omega)\}. 
\]
Now it is easy to see that 
\[
Q' = \{\alpha \in Q \mid  \tau\tau_0(\alpha) =\alpha,\quad (\alpha,\alpha_r^{\vee}) \in 2\Z \textrm{ for }\tau\tau_0(r) = r\}.
\]
Write 
\[
\delta_r =  e^{\pi i (\rho^{\vee} + \rho_X^{\vee} + \chi_0 -\tau(\chi_0),\alpha_r)} \in \{\pm1\}.
\]
Identify $\Z^l \cong Q$ via $k \mapsto \sum k_r \alpha_r$, where $l$ is the rank of $\mfg$, and let $I_0 \subseteq I$ be the set of $\tau\tau_0$-fixed points. Let $A'$ be the rectangular matrix obtained by restricting the rows of the Cartan matrix $A$ to the index set $I_0$. Then we see that \eqref{EqToCheck} becomes
\begin{equation}\label{EqSignCheck}
\prod_{r=1}^l (\epsilon_r\delta_r)^{k_r} = 1,\qquad \forall k\in \Z^l \textrm{ such that } \tau\tau_0(k) = k,\quad A'k \in 2\Z^{|I_0|}. 
\end{equation}
As the left hand side takes values in $\pm1$, we can consider the condition on $k$ modulo two, and are thus to check \eqref{EqSignCheck} on 
\[
B = \{ k = (k_r) \in (\Z/2\Z)^l \mid \tau\tau_0(k) = k, \quad A'k =0 \textrm{ mod } 2\}.
\]
Let us present the elements of $B$ in the different cases, listing only those for which $B \neq \{0\}$. We will use the ordering of simple roots as in \cite{OV90}*{Reference Chapter, Section 2}. 
\begin{enumerate}
\item Case of $\tau\tau_0 = \id$. Then $B = \Ker(A) \textrm{ mod } 2$, and we find the following non-zero elements of $B$:
\begin{enumerate}
\item $A_l$ for $l$ odd: $k= (1,0,1,\ldots, 1,0,1)$.
\item $B_l$: $k= (1,0,1,0,\ldots,1,0)$ for $l$ even, $(1,0,1,0,\ldots, 1,0,1)$ for $l$ odd. 
\item $C_l$: $k= (0,0,\ldots, 0,1)$,
\item $D_l$ for $l$ odd: $k = (0,0,0,\ldots, 0,1,1)$.
\item $D_l$ for $l$ even: $k = (a,0,a,0,\ldots, a,0,b,c)$ with $a+b+c= 0$. 
\item $E_7$:  $k = (1,0,1,0,0,0,1)$.
\end{enumerate}
\item  Case of $\tau\tau_0 \neq \id$:
\begin{enumerate}
\item $A_l$ for $l=2p$: $k = (a_1,\ldots,a_p,a_p,\ldots, a_1)$, 
\item $A_l$ for $l= 2p+1$: $k = (a_1,\ldots,a_p,a_{p+1},a_p,\ldots, a_1)$.
\item $D_l$ for $l$ odd:  $k = (a,0,a,0,\ldots, a,0,a,b,b)$.
\item $D_l$ for $l$ even: $k = (0,0,\ldots, 0,0,a,a)$. 
\item $E_6$: $k = (a,b,0,b,a,0)$.
\end{enumerate}
\end{enumerate}

One can now check \eqref{EqSignCheck} by an easy case-by-case verification, using Tables \ref{TableABC},\ref{TableD},\ref{TableE} with the following legend:
\begin{itemize}
\item The first column presents an enhanced Satake diagram, with $z_r = \pm 1$ indicated whenever the value is not a priori determined, and (contrary to custom) with also the action of $\tau$ drawn for black vertices when non-trivial (to avoid possible confusion). 
\item The second column encodes $\tau\tau_0$ and the function $\delta$ with a root colored black if $\delta_r = -1$. Note that $\delta$ depends on the choice of $\chi$, but this will only come into play in the $DIII$-case $\mfso^*(2l) = \mfsu_{l}^*(\mathbb{H})$ for $l$ odd, where we list the two possibilities.
\item The third column presents an associated standard Vogan diagram, which can for example  be deduced from the Kac diagrams in \cite{OV90}*{Reference Chapter, Table 7}.
\end{itemize}
\end{proof}

\begin{Rem}
\begin{enumerate}
\item Note that the theorem is no longer true if we work with conjugacy instead of inner conjugacy, as then the case $\mfso^*(4p)$ fails!
\item It is clear that the converse of the theorem does not hold. It would be interesting to determine which extra conditions are needed to make the converse hold. 
\end{enumerate}
\end{Rem}

\newpage

\begin{table}[ht]
\caption{Satake diagrams, $\delta$-function and Vogan diagrams, Type A/B/C}\label{TableABC}
\begin{center}
\bgroup
\def\arraystretch{2.4}
{\setlength{\tabcolsep}{1.5em}
\begin{tabular}{|c|c|c|c|}
\hline
Lie algebra & Satake $(X,\tau,z)$ & $(\delta,\tau\tau_0)$ & Vogan $(\epsilon,\tau\tau_0)$\\ \hline
$\underset{\ell \textrm{ odd}}{\mfsl(\ell,\R)}$& 
\multirow{2}{*}{
  \begin{tikzpicture}[scale=.4,baseline]
\draw (0cm,0) circle (.2cm) node[above]{\small $1$} ;
\draw (0.2cm,0) -- +(0.6cm,0);
\draw (1cm,0) circle (.2cm);
\draw (1.2cm,0) -- +(0.2cm,0);
\draw[dotted] (1.4cm,0) --+ (1cm,0);
\draw (2.4cm,0) --+ (0.2cm,0);
\draw (2.8cm,0) circle (.2cm); 
\draw (3cm,0) --+ (0.6cm,0);
\draw (3.8cm,0) circle (.2cm)  node[above]{\small $\ell$};
\end{tikzpicture}
}
& 
\begin{tikzpicture}[scale=.4,baseline]
\node (v1) at (0,0) {};
\node (v2) at (2,0) {};
\node (v3) at (4,0) {};
\node (v4) at (6,0) {};

\draw[fill=black] (0cm,0) circle (.2cm);
\draw (0.2cm,0) -- +(0.2cm,0);
\draw[dotted] (0.4cm,0) --+ (1cm,0);
\draw (1.6cm,0) --+ (0.2cm,0);
\draw[fill=black] (2cm,0) circle (.2cm);
\draw (2.2cm,0) -- +(0.6cm,0);
\draw[fill=black] (3cm,0) circle (.2cm);
\draw (3.2cm,0) -- +(0.6cm,0);
\draw[fill=black] (4cm,0) circle (.2cm);
\draw (4.2cm,0) -- +(0.2cm,0);
\draw[dotted] (4.4cm,0) --+ (1cm,0);
\draw (5.6cm,0) --+ (0.2cm,0);
\draw[fill=black] (6cm,0) circle (.2cm); 

\draw[<->]
(v1) edge[bend left] (v4);
\draw[<->]
(v2) edge[bend left=50] (v3);
\end{tikzpicture}
& 
\begin{tikzpicture}[scale=.4,baseline]
\node (v1) at (0,0) {};
\node (v2) at (2,0) {};
\node (v3) at (4,0) {};
\node (v4) at (6,0) {};

\draw (0cm,0) circle (.2cm);
\draw (0.2cm,0) -- +(0.2cm,0);
\draw[dotted] (0.4cm,0) --+ (1cm,0);
\draw (1.6cm,0) --+ (0.2cm,0);
\draw (2cm,0) circle (.2cm);
\draw (2.2cm,0) -- +(0.6cm,0);
\draw[fill=black] (3cm,0) circle (.2cm);
\draw (3.2cm,0) -- +(0.6cm,0);
\draw (4cm,0) circle (.2cm);
\draw (4.2cm,0) -- +(0.2cm,0);
\draw[dotted] (4.4cm,0) --+ (1cm,0);
\draw (5.6cm,0) --+ (0.2cm,0);
\draw (6cm,0) circle (.2cm); 

\draw[<->]
(v1) edge[bend left] (v4);
\draw[<->]
(v2) edge[bend left=50] (v3);
\end{tikzpicture}
\\
\hline
\end{tabular}
}
\egroup
\end{center}
\end{table}

\newpage

\begin{table}[ht]
\caption{Satake diagrams, $\delta$-function and Vogan diagrams, Type D}\label{TableD}
\begin{center}
\bgroup
\def\arraystretch{3}
{\setlength{\tabcolsep}{1em}
\begin{tabular}{|c|c|c|c|}
\hline
Lie algebra & Satake $(X,\tau)$ & $(\delta,\tau\tau_0)$ & Vogan $(\epsilon,\tau\tau_0)$\\
\hline 
\end{tabular}
}
\egroup
\end{center}
\end{table}

\newpage

\begin{table}[ht]
\caption{Satake diagrams, $\delta$-function and Vogan diagrams, Type E}\label{TableE}
\begin{center}
\bgroup
\def\arraystretch{3}
{\setlength{\tabcolsep}{1em}
\begin{tabular}{|c|c|c|c|}
\hline
Lie algebra & Satake $(X,\tau)$ & $(\delta,\tau\tau_0)$ & Vogan $(\epsilon,\tau\tau_0)$\\ 
\hline
$EI$& 
 \begin{tikzpicture}[scale=.4,baseline=-6pt]
\draw (0cm,0) circle (.2cm) node[above]{\small $1$};
\draw (0.2cm,0) -- +(0.6cm,0);
\draw (1cm,0) circle (.2cm) node[above]{\small $2$};
\draw (1.2cm,0) -- +(0.6cm,0);
\draw (2cm,0) circle (.2cm) node[above]{\small $3$};
\draw (2.2cm,0) -- +(0.6cm,0);
\draw (3cm,0) circle (.2cm) node[above]{\small $4$};
\draw (3.2cm,0) -- +(0.6cm,0);
\draw (4cm,0) circle (.2cm) node[above]{\small $5$};
\draw (2cm,-0.2) -- +(0cm,-0.6);
\draw (2cm,-1) circle (.2cm) node[below]{\small $6$};
\end{tikzpicture}
&
 \begin{tikzpicture}[scale=.4,baseline=-6pt]
\node (v1) at (0,0.2) {};
\node (v2) at (4,0.2) {};
\node (v3) at (1,0.2) {};
\node (v4) at (3,0.2) {};
\draw[fill = black] (0cm,0) circle (.2cm);
\draw (0.2cm,0) -- +(0.6cm,0);
\draw[fill = black] (1cm,0) circle (.2cm);
\draw (1.2cm,0) -- +(0.6cm,0);
\draw[fill = black] (2cm,0) circle (.2cm);
\draw (2.2cm,0) -- +(0.6cm,0);
\draw[fill = black] (3cm,0) circle (.2cm);
\draw (3.2cm,0) -- +(0.6cm,0);
\draw[fill = black] (4cm,0) circle (.2cm);
\draw (2cm,-0.2) -- +(0cm,-0.6);
\draw[fill = black] (2cm,-1) circle (.2cm);
\draw[<->]
(v1) edge[bend left] (v2);
\draw[<->]
(v3) edge[bend left] (v4);
\end{tikzpicture}
&
 \begin{tikzpicture}[scale=.4,baseline=-6pt]
\node (v1) at (0,0.2) {};
\node (v2) at (4,0.2) {};
\node (v3) at (1,0.2) {};
\node (v4) at (3,0.2) {};
\draw (0cm,0) circle (.2cm);
\draw (0.2cm,0) -- +(0.6cm,0);
\draw (1cm,0) circle (.2cm);
\draw (1.2cm,0) -- +(0.6cm,0);
\draw (2cm,0) circle (.2cm);
\draw (2.2cm,0) -- +(0.6cm,0);
\draw (3cm,0) circle (.2cm);
\draw (3.2cm,0) -- +(0.6cm,0);
\draw (4cm,0) circle (.2cm);
\draw (2cm,-0.2) -- +(0cm,-0.6);
\draw[fill = black] (2cm,-1) circle (.2cm);
\draw[<->]
(v1) edge[bend left] (v2);
\draw[<->]
(v3) edge[bend left] (v4);
\end{tikzpicture}
\\
\hline
\end{tabular}
}
\egroup
\end{center}
\end{table}

\section{Invariant vectors and exterior algebras}\label{Ap2}

In this appendix we will prove a result concerning the spectral subspaces $\mcO_q(U^{\theta}\backslash U)_{\mu_i}$, for certain fundamental spherical weights $\mu_i$ in the cases AII, CII and DIII. The proof will make use of the explicit results of Noumi and Sugitani \cite{NS95}, as well as some appropriate $q$-analogues of exterior algebras.

\subsection{Noumi-Sugitani coideals}

In \cite{NS95} Noumi and Sugitani construct some quantum analogs of $U(\mfu^{\theta})$ for certain involutions $\theta$ and $\mfg$ of classical type. The construction is based on finding explicit solutions $J$ of the reflection equation, from which one can build coideals $B_J \subseteq U_q(\mfu)$ which specialize to $U(\mfu^{\theta})$. In \cite{Let99}*{Section 6}, Letzter
shows that the coideals $B_{J}$ are subalgebras of appropriate $B_{\theta}$, where $B_{\theta}$ is the coideal corresponding to the involution $\theta$ that she constructs in the cited paper. Moreover it follows from \cite{NS95}*{Theorem 1} and \cite{Let02}*{Theorem 7.7} that one has equality of the invariant subspaces $V_\varpi^{B_J} = V_\varpi^{B_\theta}$ for all $\varpi \in P^+$.

Let us give some more details regarding \cite{NS95}. Let $V = V_{\varpi_1}$ be the $N$-dimensional fundamental representation for $\mfg$ of classical type. Then, for the classical symmetric pairs considered in the cited paper, $V\otimes V$ contains only one non-trivial spherical representation, namely $V_{\varpi_2}$. A vector invariant under $B_J$ is then given by
\[
w_{J}=\sum_{i,j=1}^{N}J_{ij}v_{i}\otimes v_{j}\in V\otimes V,
\]
where $\{v_{i}\}_{i=1}^{N}$ is a basis of $V$ and $J=\sum_{i,j}J_{ij}e_{ij}$ in terms of the matrix units $e_{ij}$. The matrices $J$ are given explicitly for the classical symmetric pairs under consideration.

\subsection{Classical and quantum exterior algebras}

Let us consider the symmetric pairs AII, CII and DIII. We are concerned with those spherical weights $\mu_i$ such that $\tau(i) = i$ and the node $i$ is connected to a black vertex in the Satake diagram. These spherical weights are summarized in Table \ref{TableSpherical}, where we recall that we use the standard ordering for the Dynkin diagrams as can be found in the tables \ref{TableABC} and \ref{TableD}.

\begin{table}[h]
\caption{Relevant spherical weights for the AII, CII and DIII cases.}
\label{TableSpherical}
\begin{centering}
\begin{tabular}{|c|c|c|}
\hline 
Case & $\mathfrak{g}$ & Relevant spherical weights\tabularnewline
\hline 
\hline 
AII & $A_{2n-1} = \mathfrak{sl}_{2n}$ & $\varpi_2, \varpi_4, \cdots,\varpi_{2n-2}$\tabularnewline
\hline 
CII ($\ell\leq[n/2]$) & $C_{n}=\mathfrak{sp}_{n}$ & $\varpi_{2},\varpi_{4},\cdots,\varpi_{2\ell}$\tabularnewline
\hline 
DIII (first case) & $D_{2\ell}=\mathfrak{so}_{4\ell}$ & $\varpi_{2},\varpi_{4},\cdots,\varpi_{2\ell-2}$\tabularnewline
\hline 
DIII (second case) & $D_{2\ell+1}=\mathfrak{so}_{4\ell+2}$ & $\varpi_{2},\varpi_{4},\cdots,\varpi_{2\ell-2}$\tabularnewline
\hline 
\end{tabular}
\end{centering}
\end{table}

Recall that most of the representations $V_{\varpi_i}$ can be constructed as exterior powers of the fundamental representation $V_{\varpi_1}$, see for instance \cite{GW09}*{Section 5.5.2}.
In the case $A_{n - 1} = \mathfrak{sl}_n$ we have $\Lambda^k(V_{\varpi_1}) \cong V_{\varpi_k}$ for $k = 1, \cdots, n - 1$. In the case $D_n = \mathfrak{so}_{2n}$ we have $\Lambda^k(V_{\varpi_1}) \cong V_{\varpi_k}$ for $k = 1, \cdots, n - 2$. In the case $C_n = \mathfrak{sp}_n$ the exterior powers are reducible and we have the decomposition
\[
\Lambda^k(V_{\varpi_1}) \cong \bigoplus_{p = 0}^{[k/2]} V_{\varpi_{k - 2p}},
\]
with the convention that $V_{\varpi_0}$ is the trivial representation. Observe that the weight space $\Lambda^k(V_{\varpi_1})_{\varpi_k}$ is one-dimensional, since $V_{\varpi_i}$ with $i < k$ does not have the weight $\varpi_k$.

For each $V = V_{\varpi_1}$ as above, it is possible to construct a $q$-deformation $\Lambda_{q}(V)$ of the exterior algebra of $V$ which has the same graded dimension as the classical one.
The relations in $\Lambda_{q}(V)$ are more complicated that those of the classical exterior algebra, but  nevertheless we have the following result, see \cite{HK06}*{Proposition 3.6} and \cite{KTS15}*{Proposition 4.6}.

\begin{Prop}
Let $\{v_i\}_{i = 1}^N$ be a basis of $V$. Then there is a filtration $\mathcal{F}$ of $\Lambda_q(V)$ such that $\mathrm{gr}_{\mathcal{F}} \Lambda_q(V)$ is generated the $v_i$ with relations $v_i \wedge v_j = -q_{ij} v_j \wedge v_i$ for some $q_{ij} > 0$.
\end{Prop}

From this result it can be readily seen that the elements $v_{i_1} \wedge \cdots \wedge v_{i_k}$ with $i_1 < \cdots < i_k$ give a basis of $\Lambda_q^k(V)$. Hence $\dim \Lambda_q^k(V) = \dim \Lambda^k(V)$ for $k = 1, \cdots, N$.
Moreover the $U_{q}(\mathfrak{g})$-module algebra $\Lambda_{q}^{k}(V)$ decomposes as in the classical case.
The algebra $\Lambda_q(V)$ can be realized as a subspace of the
tensor algebra $T(V)$, see for instance \cite{KTS15}*{Section 3.4}. Write $\pi_\Lambda : T(V) \to \Lambda_q(V)$ for the projection. Then we denote by $\pi : T(V) \to \mathrm{gr}_{\mathcal{F}} \Lambda_q(V)$ the map obtained by composing $\pi_\Lambda$ with the projection $\Lambda_q(V) \to \mathrm{gr}_{\mathcal{F}} \Lambda_q(V)$.

\subsection{Spectral subspaces}

The content of the previous subsections will be used for the following result.

\begin{Prop}
\label{PropAlgGenSph}
Let $\mu_i$ be a spherical weight from Table \ref{TableSpherical} for AII, CII or DIII. Then the spectral subspace $\mcO_q(U^{\theta}\backslash U)_{\mu_i}$ is contained in the algebra generated by $\mcO_q(U^{\theta}\backslash U)_{\varpi_2}$.
\end{Prop}

\begin{proof}
Recall that by Theorem \ref{ThmMultFree} the subspace of $U_q(\mfu^{\theta})$-invariant vectors in $V_{\mu_i}$ is one-dimensional. Fix non-zero invariant vectors $w_i \in V_{\mu_i}$ for all $i$. Observe that if $w$ is an invariant vector, then so is $w^{\otimes n}$ for any $n \in \mathbb{N}$, since $U_q(\mfu^{\theta})$ is a coideal. Now consider the invariant vector $w_2$ corresponding to $\mu_2 = \varpi_2$. Suppose that, for each $i$ as in Table \ref{TableSpherical}, there exists some $n_i \in \mathbb{N}$ such that the component of $w_2^{\otimes n_i}$ in $V_{\mu_i}$ is non-zero. Then this component is a non-zero multiple of $w_i$. If this holds then the claim follows from
\[
U(w_2, v_1) \cdots U(w_2, v_{n_i}) = U(w_2^{\otimes n_i}, v_1 \otimes \cdots v_{n_i}).
\]
Upon changing conventions, it is equivalent to prove the same statement for the algebra $\mcO_q(G)^{B_\theta}$, where $B_\theta$ is Letzter's coideal. Moreover we have $\mcO_q(G)^{B_J} = \mcO_q(G)^{B_\theta}$, where $B_J$ is the coideal of Noumi and Sugitani.
In \cite{NS95} a $B_J$-invariant vector $w_J \in V \otimes V$ is constructed explicitly for the cases AII, CII and DIII, where $V = V_{\varpi_1}$ is the fundamental representation. The component of $w_J$ in $V_{\varpi_2} \subseteq V \otimes V$ is non-zero. We will show in Lemma \ref{LemWJ} that $w_J^{\otimes m}$ has non-zero component in $V_{\varpi_{2m}}$ for the appropriate values of $m$. Then the conclusion follows from the previous discussion.
\end{proof}

In the next lemma we will use the explicit invariant vectors $w_J$ given in \cite{NS95}.

\begin{Lem}
\label{LemWJ}
Let $w_J$ be the appropriate invariant vector for AII, CII or DIII.
\begin{enumerate}
\item For AII the component of $w_J^{\otimes m}$ in $V_{\varpi_{2m}}$ is non-zero for $1 \leq m \leq n - 1$.
\item For DIII the component of $w_J^{\otimes m}$ in $V_{\varpi_{2m}}$ is non-zero for $1 \leq m \leq \ell - 1$.
\item For CII the component of $w_J^{\otimes m}$ in $V_{\varpi_{2m}}$ is non-zero for $1 \leq m \leq \ell$.
\end{enumerate}
\end{Lem}

\begin{proof}

(1) We have $\mathfrak{g} = A_{2n - 1}$ and $V_{\varpi_1}$ has dimension $N = 2n$. The invariant vector is
\[
w_J = \sum_{k = 1}^n a_{2k} (v_{2k-1} \otimes v_{2k} - q v_{2k} \otimes v_{2k-1}),
\]
where the $a_{2k}$ are non-zero. Applying the projection $\pi$ we get $\pi(w_{J})=\sum_{k=1}^{n}b_{k}v_{2k-1}\wedge v_{2k}$ for some non-zero $b_{k}$. It is enough to show that $\pi(w_J)^{\wedge m} \neq 0$ for $1\leq m\leq n-1$. Let us focus on the term $w_{2m} = v_{1}\wedge v_{2}\wedge\cdots\wedge v_{2m-1}\wedge v_{2m}$.
It follows from the commutation relations that we have
\[
v_{2j-1}\wedge v_{2j}\wedge v_{2k-1}\wedge v_{2k}=c\cdot v_{2k-1}\wedge v_{2k}\wedge v_{2j-1}\wedge v_{2j},
\]
for some $c>0$. Then $w_{2m}$ appears with non-zero coefficient in $\pi(w_{J})^{\wedge m}$ and hence $\pi(w_{J})^{\wedge m}\neq0$.

(2) We have $\mathfrak{g} = D_n$ and $V_{\varpi_1}$ has dimension $N = 2n$. We use the notation
$j^{\prime} = N + 1 - j$. First we consider the case when $n = 2\ell$ is even. The invariant vector is given by
\[
w_{J}=\sum_{k=1}^{\ell}a_{2k}(v_{2k-1}\otimes v_{2k}-qv_{2k}\otimes v_{2k-1})+\sum_{k=1}^{\ell}a_{(2k-1)^{\prime}}(v_{(2k)^{\prime}}\otimes v_{(2k-1)^{\prime}}-qv_{(2k-1)^{\prime}}\otimes v_{(2k)^{\prime}}),
\]
where the coefficients are non-zero. Therefore its projection is given by
\[
\pi(w_{J})=\sum_{k=1}^{\ell}b_{k}v_{2k-1}\wedge v_{2k}+\sum_{k=1}^{\ell}b_{k}^{\prime}v_{(2k)^{\prime}}\wedge v_{(2k-1)^{\prime}}.
\]
Observe that $(2k)^{\prime}>2\ell$ for $1\leq k\leq\ell$. It is enough to show that $\pi(w_{J})^{\wedge m}\neq 0$ for $1\leq m\leq\ell-1$.
As for the AII case, we see that the term $w_{2m}=v_{1}\wedge v_{2}\wedge\cdots\wedge v_{2m-1}\wedge v_{2m}$ appears with non-zero coefficient, hence $\pi(w_{J})^{\wedge m}\neq 0$.
The odd case $n = 2\ell+1$ is very similar. The only difference is that in $w_{J}$
we also have a term proportional to $v_{n}\otimes v_{n^{\prime}}-v_{n^{\prime}}\otimes v_{n}$.
The rest of the argument is completely identical.

(3) We have $\mathfrak{g} = C_n$ and $V_{\varpi_1}$ has dimension $N = 2n$. We will use the notation
$j^{\prime} = 2n + 1 - j$ and consider the parameter $\ell \leq [n/2]$. We have the invariant
vector
\[
\begin{split}w_{J} & =\sum_{k=1}^{\ell}a_{2k}(v_{2k-1}\otimes v_{2k}-qv_{2k}\otimes v_{2k-1})+\sum_{k=1}^{\ell}a_{(2k-1)^{\prime}}(v_{(2k)^{\prime}}\otimes v_{(2k-1)^{\prime}}-qv_{(2k-1)^{\prime}}\otimes v_{(2k)^{\prime}})\\
 & + \sum_{j = 2\ell + 1}^n (a_j^\prime v_j \otimes v_{j^{\prime}} - a_j^{\prime -1} v_{j^\prime} \otimes v_j) + \sum_{j = 1}^{2\ell} a_j^{\prime \prime} v_j \otimes v_{j^\prime},
\end{split}
\]
where the coefficients are non-zero. Therefore applying the projection we get
\[
\pi(w_{J}) = \sum_{k=1}^{\ell} b_k v_{2k-1}\wedge v_{2k} + \sum_{k=1}^{\ell} b_k^\prime v_{(2k)^{\prime}}\wedge v_{(2k-1)^{\prime}} + \sum_{j=2\ell+1}^{n} c_j v_j \wedge v_{j^\prime} + \sum_{j = 1}^{2\ell} c_j^\prime v_j \wedge v_{j^\prime}.
\]
First we show that $\pi(w_{J})^{\wedge m} \neq 0$ for $1\leq m\leq\ell$.
Let us consider again $w_{2m} = v_{1}\wedge v_{2}\wedge\cdots\wedge v_{2m-1}\wedge v_{2m}$.
We claim that this element arises only from products of the terms
$v_{2k-1}\wedge v_{2k}$ with $1 \leq k \leq \ell$. Indeed, as $j^{\prime} > n$ for $j \leq n$ and $j^{\prime} \leq n$ for $j > n$, the element $w_{2m}$ can not contain any of the terms $v_{j}\wedge v_{j^{\prime}}$. Then, as in the other cases, we conclude that $w_{2m}$ appears with non-zero
coefficient and hence $\pi(w_{J})^{\wedge m}\neq0$.

Finally, since $\Lambda_q^{2m}(V)$ is reducible, we still need to show that we obtain a non-zero component in $V_{\varpi_{2m}}$. Recall that the fundamental representation $V_{\varpi_1}$ of $C_{n}$ has weights
$\{\lambda_{i}\}_{i=1}^{n}\cup\{-\lambda_{i}\}_{i = 1}^n$, where $\lambda_i = \varpi_i - \varpi_{i-1}$ and we use the convention $\varpi_{0}=0$.
The vectors $v_i$ for $i = 1,\cdots,n$ have weight $\lambda_i$. Then we see that the term $w_{2m}$ has weight $\sum_{i = 1}^{2m} \lambda_{i} = \varpi_{2m}$ and hence belongs to $V_{\varpi_{2m}}$.
\end{proof}

\section{Computations for the symmetric pair of type $FII$}\label{Ap3} 

We realize the root system of $\mfg = \mathfrak{f}_4$ explicitly in $\R^4$ with the usual orthonormal basis $\{\varepsilon_r\}$ by putting 
\[
\alpha_1 = \frac{1}{2}(\varepsilon_1 -\varepsilon_2-\varepsilon_3-\varepsilon_4),\quad \alpha_2 = \varepsilon_4,\quad \alpha_3 = \varepsilon_3-\varepsilon_4,\quad \alpha_4 = \varepsilon_2-\varepsilon_3.
\]
In particular, with $d_r = \frac{1}{2}(\alpha_r,\alpha_r)$ we have $d_1=d_2 =1/2$ and $d_3,d_4 = 1$. Then $\varpi_1 = \varepsilon_1$, and $V=V_{\varpi_1}$ is a quasi-minuscule $26$-dimensional $*$-representation of $U_q(\mathfrak{f}_4)$. To realize it explicitly, let us use the notation
\[
[n] = \frac{q^{n/2} -q^{-n/2}}{q^{1/2}-q^{-1/2}},
\] 
so in particular $[1]= 1,[2] = q^{1/2}+q^{-1/2}$ and $[3] = q + 1 + q^{-1}$. Fix in $V$ an orthonormal basis
\[
e_k^{s_0},\quad f_{s_1s_2s_3s_4},\quad e_0,e_0',\qquad 1\leq k \leq 4, s_i \in \{\pm\},
\]
and put 
\[
f_0 = [2]^{-1} (e_0 + [3]^{1/2}e_0'),
\]
so that $f_0$ is a unit vector. Then we can let $U_q(\mathfrak{f}_4)$ act uniquely by the following rules: the vectors $f_{s_1s_2s_3s_4}$ have weight $\frac{1}{2}\sum_{i} s_i\varepsilon_i$, the vectors $e_k^{\pm}$ have weight $\pm \varepsilon_k$, and the vectors $e_0,f_0$ have weight zero. Further, the $F_r$ act as in the diagram \ref{FigF} below.
\renewcommand{\figurename}{Diagram}
\begin{figure}
{
\vspace{-0.4cm}\footnotesize
\[
\xymatrix{ 
&&e_1^+ \ar[d]_{F_1}^{q^{-1/4}} &&\\ &&f_{++++} \ar[d]_{F_2}^{q^{-1/4}} &&\\ &&f_{+++-} \ar[d]_{F_3}^{q^{-1/2}} &&\\ && f_{++-+} \ar[ld]_{F_2}^{q^{-1/4}} \ar[rd]_{F_4}^{q^{-1/2}}&& \\ &f_{++--}  \ar[ld]_{F_1}^{q^{-1/4}} \ar[rd]_{F_4}^{q^{-1/2}} && f_{+-++} \ar[ld]_{F_2}^{q^{-1/4}} & \\ e_2^+ \ar[rd]_{F_4}^{q^{-1/2}} && f_{+-+-} \ar[ld]_{F_1}^{q^{-1/4}} \ar[rd]_{F_3}^{q^{-1/2}} && \\ &e_3^+ \ar[rd]_{F_3}^{q^{-1/2}}  &&  f_{+--+}  \ar[ld]_{F_1}^{q^{-1/4}} \ar[rd]_{F_2}^{q^{-1/4}} && \\ && e_4^+ \ar[d]_{F_2}^{q^{-1/2}[2]^{1/2}} && f_{+---} \ar[d]_{F_1}^{q^{-1/2}[2]^{1/2}}&  \\ &&e_0 \ar[d]_{F_2}^{[2]^{1/2}} \ar[rrd]^(.20){F_1}_(.75){\!\!\!\![2]^{-1/2}} && f_0 \ar[lld]_(.20){F_2}^(.75){\;[2]^{-1/2}} \ar[d]_{F_1}^{[2]^{1/2}}  &  \\ && e_4^- \ar[ld]_{F_3}^{q^{-1/2}} \ar[rd]_{F_1}^{q^{-1/4}} && f_{-+++} \ar[ld]_{F_2}^{q^{-1/4}} \\ & e_3^- \ar[ld]_{F_4}^{q^{-1/2}} \ar[rd]_{F_1}^{q^{-1/4}} && f_{-++-} \ar[ld]_{F_3}^{q^{-1/2}} & \\ e_2^- \ar[rd]_{F_1}^{q^{-1/4}} && f_{-+-+} \ar[ld]_{F_4}^{q^{-1/2}} \ar[rd]_{F_2}^{q^{-1/4}}&& \\ & f_{--++} \ar[rd]_{F_2}^{q^{-1/4}} && f_{-+--} \ar[ld]_{F_4}^{q^{-1/2}} \\ && f_{--+-}  \ar[d]_{F_3}^{q^{-1/2}}&& \\ && f_{---+} \ar[d]_{F_2}^{q^{-1/4}} && \\ && f_{----} \ar[d]_{F_1}^{q^{-1/4}} && \\ && e_1^-&& 
}
\]
}
\caption{Action of the $F_r$ on $V$}\label{FigF}
\end{figure}

The operators $E_r = K_rF_r^*$ act in the obvious way by the adjoint operation, for example
\[
E_1f_{-+++} = [2]^{1/2}f_0,\quad E_1f_0 = q^{1/2}[2]^{1/2}f_{+---},\quad E_1e_0 = q^{1/2}[2]^{-1/2} f_{+---},
\]
\[
E_2 e_4^- = [2]^{1/2} e_0,\quad E_2e_0 = q^{1/2}[2]^{1/2}e_4^+,\quad E_2f_0 = q^{1/2} [2]^{-1/2} e_4^+.
\]

Put $X = \{\alpha_2,\alpha_3,\alpha_4\}$. 
\begin{Lem}\label{LemActTX}
On basis vectors, we have the following action of $T_{w_X}$,
\[
T_{w_X} f_{s_1s_2s_3s_4} = s_2s_4 q^{9/4} f_{s_1,-s_2,-s_3,-s_4},
\]
\[
T_{w_X} e_1^{\pm} = e_{1}^{\pm},\quad T_{w_X} e_2^{\pm} = q^{5/2} e_2^{\mp},\quad  T_{w_X} e_3^{\pm} = -q^{5/2} e_3^{\mp},\quad  T_{w_X} e_4^{\pm} = q^{5/2} e_4^{\mp}
\]
and 
\[
T_{w_X} e_0 = -q^3e_0,\qquad T_{w_X}f_0 = f_0 -q^{3/2}([3]-2)e_0. 
\]
\end{Lem} 
\begin{proof}
The longest word in $W_X$ is given by 
\[
w_X\varepsilon_1 = \varepsilon_1,\qquad w_X \varepsilon_r = -\varepsilon_{r}\textrm{ for }r\in \{2,3,4\},
\]
with reduced expression
\[
w_X = s_{\varepsilon_2}s_{\varepsilon_3}s_{\varepsilon_4} =  (s_4s_3s_2s_3s_4)(s_3s_2s_3)s_2.
\]

Now consider for $U_{q_r}(\mfsu(2))$ the spin $1/2$-representation and spin $1$-representation determined by respective orthonormal weight bases $\{v_{\pm 1/2}\}$ and $\{v_{-},v_0,v_+\}$ with actions 
\[
F_rv_{+1/2} = q_r^{-1/2} v_{-1/2},\qquad F_rv_+ = q_r^{-1}(q_r+q_r^{-1})^{1/2} v_0,\qquad F_rv_0 = (q_r+q_r^{-1})^{1/2} v_{-}. 
\]
Then with respect to these bases, we have from \eqref{EqActT} that the Lusztig braid operator $T_r$ acts via
\[
T_rv_{+1/2} = -q_r^{1/2} v_{-1/2},\quad T_rv_{-1/2} = q_r^{1/2} v_{+1/2},\quad T_rv_+ = q_rv_{-},\quad T_rv_{-} = q_r v_{+},\quad T_rv_0 = -q_r^2v_0. 
\]
One can then easily compute from this the action of $T_{w_X}$ on the $e_{r}^{\pm}$. For the $f_{s_1s_2s_3s_4}$ one can compute $T_{w_X}$ on $f_{++++}$, and use the formula \eqref{EqIdAdT_0} for the remaining $f_{+s_2s_3s_4}$. For the $f_{-s_2s_3s_4}$ one  can then use the $U_q'(\mfg_X)$-isomorphism $f_{+s_2s_3s_4} \mapsto f_{-s_2s_3s_4}$. Finally, for $e_0$ the value of $T_{w_X}$ is directly computed. Since $e_0 - [2]f_0$ is a $U_q'(\mfg_X)$-fixed vector, it must also be a $T_{w_X}$-fixed vector, from which the value of $T_{w_X}f_0$ can be computed. 
\end{proof}

Consider now the enhanced Satake diagram $(X,\id,1)$ with associated Satake involution $\theta$ and coideal $*$-subalgebra $U_q(\mathfrak{f}_4^{\theta}) \subseteq U_q(\mathfrak{f}_4)$. Using that $\alpha_1^+ = -\frac{1}{2}(\epsilon_2+\epsilon_3+\epsilon_4)$ and that $z_1 = 1$, consider as in \eqref{DefNewGenCoid} the generator
\[
C_1 = E_1 - q^{3/4}T_{w_X} F_1T_{w_X}^{-1} K_1  \in U_q(\mathfrak{f}_4^{\theta}).
\] 
Then by direct computations using Lemma \ref{LemActTX} one finds the following values: 
\[
C_1 e_1^+  = -q^{13/4}f_{+---},\qquad C_1 e_1^- = q^{1/4} f_{----},
\]
\begin{equation}\label{EqValueC}
C_1 f_{++++} =  q^{1/4}e_1^+ -q^{-5/2}[2]^{1/2} f_0 + q^{-1}[2]^{1/2}([3]-2)e_0,\qquad  C_1 f_{-+++} = [2]^{1/2}f_0 - q^{-11/4} e_1^-,
\end{equation}
\[
C_1^* f_{+--+} =  q^{1/4} e_4^+,\qquad C_1^* f_{---+} = q^{-1/4} e_4^+.
\]

Let now $\msK$ be the $*$-compatible $\nu$-modified universal $K$-matrix for $U_q(\mathfrak{f}_4^{\theta})$. 
\begin{Lem}\label{LemComputeK}
There exists a non-zero scalar $a\in \C$ such that 
\[
a\msK e_1^+ = e_1^-,\qquad a\msK f_{+s_2s_3s_4} = -q^{-3} f_{-s_2s_3s_4},
\]
\[
a\msK e_0 = -q^{-7/2}e_0,\qquad  a\msK f_0 = q^{-1/2}f_0 + [2]^{-1/2}q^{-1/4}(q^{3}-q^{-3})e_1^-  -q^{-2}([3]-2)e_0. 
\]
\end{Lem}
\begin{proof}
Since
\[
\msK: V(\omega) \mapsto V(-w_X\omega) \oplus V(-w_X\omega - \varepsilon_1) \oplus V(-w_X\omega - 2\varepsilon_1)
\]
on weight spaces, there will exist a non-zero $a\in \C$ such that $a\msK e_1^+ = e_1^-$. Now since $\msK$ commutes with $C_1$ and $C_1e_1^+$ is a multiple of $f_{+---}$, we can compute the action of $a\msK$ on $f_{+---}$ using the formulas in \eqref{EqValueC}. As $\msK$ commutes with $U_q'(\mfg_X)$, this then determines $a\msK$ on all $f_{+s_2s_3s_4}$. Since $e_0$ is a scalar multiple of $F_2C_1^*f_{+--+}$, we can once again use the commutation of $\msK$ with $C_1^*$ and $F_2$ to determine the value of $a\msK$ on $e_0$. 

Finally, using that 
\[
C_1  a\msK f_{++++} = -q^{-3} C_1 f_{-+++} = -q^{-3}([2]^{1/2}f_0 -q^{-11/4} e_1^-)
\]
equals
\[
a\msK C_1 f_{++++} = q^{1/4}e_1^- -q^{-5/2}[2]^{1/2} a\msK f_0 -q^{-9/2} [2]^{1/2}([3]-2)e_0,
\]
we find the expression for $a\msK f_0$. 
\end{proof}

Consider now the span of the $Z_V(\xi,\eta)$ in $\mcO_q(Z_{\nu})$ with its left $U_q(\mathfrak{f}_4)$-action 
\[
x \rhd Z_V(\xi,\eta) = Z_V(S(x_{(1)})^*\xi,x_{(2)}\eta).
\]
Write
\[
z_0 = \sum_{s_2,s_3,s_4} a_{s_2,s_3,s_4} Z_V(f_{-s_2s_3s_4},f_{+s_2s_3s_4}),
\]
where 
\[
a_{+++} = 1,\quad a_{++-}=q^{-1},\quad a_{+-+}=q^{-3},\quad a_{+--} = q^{-4},
\]
\[
a_{-++}=q^{-5},\quad q_{-+-} = q^{-6},\quad a_{--+} = q^{-8},\quad a_{---} = q^{-9}. 
\]
Write further
\[
z_+ = Z_V(e_0 - [2]f_0,e_1^+),\qquad z_- = Z_V(e_1^-,e_0 - [2]f_0).   
\]
Then a straightforward computation shows the following lemma.

\begin{Lem}\label{LemInvVect}
The elements $z_0,z_+,z_-$ are $U_q'(\mfg_X)$-invariant, and 
\[
z = q^{5/4}z_+ - \frac{[3]}{[2]^{1/2}} z_0 + q^{-41/4} z_-
\]
is a highest weight vector for $U_q(\mathfrak{f}_4)$ at weight $\varepsilon_1$. 
\end{Lem}

\begin{Prop}\label{PropExistNonZero}
The element $\phi(z)$ is a non-zero element in $\mcO_q(U^{\theta}\backslash U)_{\varpi_1}$. 
\end{Prop} 
\begin{proof}
By Lemma \ref{LemInvVect} and equivariance of $\phi$ it is clear that $\phi(z) \in \mcO_q(U^{\theta}\backslash U)_{\varpi_1}$. To see that $\phi(z)\neq 0$ it is sufficient to compute that $\varepsilon(\phi(z)) \neq 0$. We have however by Lemma \ref{LemComputeK} that
\begin{eqnarray*}
&& \hspace{-1cm} a\varepsilon(\phi(z))  \\  &=& q^{5/4} \langle e_0 - [2]f_0,a\msK e_1^+\rangle - \frac{[3]}{[2]^{1/2}} \sum_{s_2,s_3,s_4} a_{s_2s_3s_4} \langle f_{-s_2s_3s_4},a\msK f_{+s_2s_3s_4}\rangle  + q^{-41/4} \langle e_1^-,a\msK(e_0 - [2]f_0)\rangle \\
&=& [2]^{-1/2}(q^{-3}[3]\sum a_{s_2s_3s_4} + q^{-21/2}[2](q^{-3}-q^3)) \\
&=& q^{-1/2}[2]^{1/2}(q^{-3}[3](1+q^{-3}+q^{-5}+q^{-8}) + q^{-10}(q^{-3}-q^3))\\
&>& 0.
\end{eqnarray*} 
\end{proof}

\end{document}